\documentclass[11pt]{article}
\usepackage{amsmath}
\usepackage{amssymb}
\usepackage{amsfonts}
\usepackage{mathrsfs}
\usepackage[usenames]{color}
\usepackage[numbers]{natbib}
\usepackage{enumitem} 
\usepackage{amscd} 

\usepackage[normalem]{ulem} 
\usepackage{comment}

\usepackage{hyperref}
\hypersetup{colorlinks=true,linkcolor=blue,
citecolor=blue} 

\usepackage{amsthm}

\def\ben#1{\begin{equation}#1\end{equation}}

\def\al#1{\begin{align*}#1\end{align*}}
\def\aln#1{\begin{align}#1\end{align}}

\usepackage{latexsym}
\newsavebox{\toy}
\savebox{\toy}{\framebox[0.65em]{\rule{0cm}{1ex}}}
\newcommand{\QED}{\usebox{\toy}\end{demo}}

\numberwithin{equation}{section}

\newtheorem{theorem}{Theorem}[section]
\newtheorem{lemma}[theorem]{Lemma}
\newtheorem{proposition}[theorem]{Proposition}
\newtheorem{cor}[theorem]{Corollary}
\newtheorem{rem}[theorem]{Remark}

\newcommand{\bd}{\begin{displaymath}}
\newcommand{\ed}{\end{displaymath}}
\newcommand{\N}{{\mathbb{N}}}

\newcommand{\R}{{\mathbb{R}}}
\newcommand{{\rd}}{\R^d}

\newcommand{\IP}{{\mathbb P}}
\newcommand{\IE}{\mathbb E}
\newcommand{\DP}{{\mathrm P}}
\newcommand{\DE}{{\mathrm E}}
\newcommand{\lan}{\langle}
\newcommand{\ran}{\rangle}

\renewcommand{\b}{\beta}

\newcommand{\rmd}{\mathrm{d}}
\newcommand{\D}{\Delta}

\newcommand{\e}{\varepsilon}
\newcommand{\ve}{\epsilon}
\newcommand{\z}{\zeta}

\newcommand{\kU}{\mathscr{U}}
\newcommand{\kW}{\mathcal{W}}

\newcommand{\dd}{\,\text{\rm d}}             
\newcommand{\dB}{\xi}

\newcommand{\cF }{{\cal F}}
\newcommand{\cG }{{\cal G}}

\newcommand{\vphi}{f}

\newcommand{\sZ}{{\bf{\mathcal{Z}}}}

\newcommand{\ssup}[1] {{\scriptscriptstyle{({#1}})}}

\newcommand{\cvlaw}{\stackrel{(d)}{\longrightarrow}}
\newcommand{\cvIP}{\stackrel{\IP}{\longrightarrow}}
\newcommand{\cvLone}{\stackrel{L^1}{\longrightarrow}}
\newcommand{\cvLtwo}{\stackrel{L^2}{\longrightarrow}}
\newcommand{\eqlaw}{\stackrel{(d)}{=}}

\newcommand{\MN}{\color{magenta}}

\AtBeginDocument{
   \def\MR#1{}  }

\author{Cl\'ement Cosco\footnote{ Weizmann Institute of Science.
\texttt{clement.cosco@gmail.com}}
\and
Shuta Nakajima\footnote{Department of Mathematics and Computer Science of the University of Basel. \texttt{shuta.nakajima@unibas.ch}}
   \and
   Makoto Nakashima\footnote{Graduate School of Mathematics, Nagoya University. \texttt{nakamako@math.nagoya-u.ac.jp}}
}

\begin{document}


%
%
%
\title{Law of large numbers and fluctuations in the sub-critical and $L^2$ regions for SHE and KPZ equation in dimension $d\geq 3$}


%

\date{}
\maketitle
\begin{center}
\date{\today}
\end{center}
\begin{abstract}
There have been recently several works studying the regularized stochastic heat equation (SHE) and Kardar-Parisi-Zhang (KPZ) equation in dimension $d\geq 3$ as the smoothing parameter is switched off, but most of the results did not hold in the full temperature regions where they should. Inspired by martingale techniques coming from the directed polymers literature, we first extend the law of large numbers for SHE obtained in \cite{MSZ16} to the full weak disorder region of the associated polymer model and to more general initial conditions. We further extend the Edwards-Wilkinson regime of the SHE and KPZ equation studied in \cite{GRZ18,MU17,DGRZ20} to the full $L^2$-region, along with multidimensional convergence and general initial conditions for the KPZ equation (and SHE), which were not proven before. To do so, we 
rely on a martingale CLT combined with a refinement of the local limit theorem for polymers.
\end{abstract}

\\


{\footnotesize 

%
%
%

\section{Introduction}
 We refer the reader to Section \ref{sec:statResults} for precise statements of the theorems that we show.
\subsection{Regularization scheme of the KPZ equation and SHE}
The stochastic heat equation (with multiplicative white noise) is formally given by the SPDE
\begin{equation} \label{eq:formalSHE}
\frac{\partial}{\partial t} u(t,x) = \frac12 \D u(t,x) +
 \b u(t,x) \dB(t,x),
\end{equation}
for $t\geq 0$ and $x\in\mathbb R^d$ and
where $\xi$ is a space-time white noise on $\mathbb R_+ \times \mathbb R ^d$. Let us assume that $\xi$ is defined on a probability space $(\Omega,\mathcal F,\IP)$ and let $\IE$ denote the expectation with respect to $\IP$.

The equation is closely related to the KPZ equation which is formally described by
\begin{equation} \label{eq:formalKPZ}
\frac{\partial}{\partial t} h = \frac12 \D h +  \frac12   |\nabla h |^2+
 \b  \dB,
\end{equation}
and the two equations are related through the so-called Hopf-Cole transform 
\begin{equation}
h=\log u.
\end{equation}

In dimension $d=1$, much work has been concentrated in the past years in order to make sense of equation \eqref{eq:formalKPZ}, with the notable contributions \cite{BG97,H13}. One of the difficulties coming when trying to do so is that the derivative $\partial_x h$ fails to be a function and so one has to make sense of the square of a distribution. When $d\geq 3$, the solution $u$ should also fails to be a function, hence giving sense to the product $u\xi$ is also an issue. In this case, both SPDEs fall into the so-called \emph{super-critical} dimensions classification and do not enter in the framework of the recent theories \cite{H14,GIP15,KM17,GP18}. 

The general approach to studying equations as \eqref{eq:formalSHE} and \eqref{eq:formalKPZ} is to begin with mollifying the noise, and then look at the behaviour of the regularized solutions when the mollification is removed.
Supposing from now on that $d\geq 3$, we consider the regularized SHE:
\begin{equation}\label{eq:SHEintro}
 \frac{\partial}{\partial t} u_{\e} = \frac12 \D u_{\e} +
 \b \e^{\frac{d-2}2} u_{\e} \, \dB_{\e}, \quad u_\e(0,\cdot)\equiv u_0, 
\end{equation}
 where $\xi_\e$ is a mollification in space of $\xi$ such that $\xi_\e \Rightarrow \xi$ as $\e\to 0$, i.e.\
\[
\xi_{\e}(t,x) = (\xi(t,\cdot)\star \phi_\e)(x)=  \int \phi_\e(x - y) \xi(t,y) \dd y,
\]
with $\phi_\e = \e^{-d} \phi(\e^{-1} x)$ and $\phi$ being a smooth, compactly supported, symmetric function on $\mathbb R ^d$, so that $\phi_\e$ converges in distribution to the Dirac mass $\delta_0$. 

The authors in \cite{MSZ16} have shown that if one rescales the noise in \eqref{eq:SHEintro} by a $\e^{\frac{d-2}2}$ factor, the solution $u_\e$ of \eqref{eq:SHEintro} can be expressed in terms of a partition function of a directed polymer in white noise environment, at inverse temperature $\b$ (see also \cite{C17} for a general introduction to the directed polymer model). Let us briefly explain how this connection goes.  

Let $\sZ_T(x)$ be the \emph{normalized partition function} of the associated polymer, i.e.
\[\sZ_T(x)=\sZ_{T}(x,\xi) = \DE_x\left[\exp{\left(\beta\int_0^T \int_{{\mathbb R}^d} \phi(y-B_s) \ \xi({\rm d} s,{\rm d} y)-\frac{\b^2 T V(0)}{2}\right)} \right],
\]
where $(B_s)_{s\geq 0}$ is a Brownian motion path and $\DE_x$ the expectation associated to BM started at $x\in\mathbb R^d$,  $\b\geq 0$ and $V(x)=\int \phi(x-y)\phi(y)\dd y$.\footnote{When it is clear from the context, we sometimes use notations $\sZ_T=\sZ_T(\xi)=\sZ_T(0,\xi)$.} Further define:
\begin{equation}\label{eq:dBetx}
\xi^{\ssup{\e,t}}(s,y)= \e^{\frac{d+2}2}\dB\Big( t-\e^2 s,\e y  \Big)\quad \text{and} \quad \xi^{\ssup{\e,t,x}}(s,y)= \e^{\frac{d+2}2}\dB\Big( t-\e^2 s,\e y -x  \Big),
\end{equation}
respectively as a diffusively rescaled, time-reversed (resp.\ space-shifted) white noise (note that by scaling properties of the white noise, both fields 
have the same law as $\xi$). Then, the Feynmann-Kac formula ensures that  for the \emph{flat initial condition} $u_{\e}(0,\cdot) \equiv 1$,
\begin{equation}\label{eq:uZ}
u_\e(t,x)= \sZ_{\frac t{\e^2}} \left( \frac x \e, \xi^{\ssup{\e,t}} \right) = \sZ_{\frac t{\e^2}} \left(\xi^{\ssup{\e,t,x}} \right),
\end{equation}
see \cite{MSZ16} for more details. Note that similiar relations hold for more general initial conditions $u_0$ (cf.\ \eqref{eq:FKgeneralIC} or \eqref{eq:diracIC}). 

A regularized solution of the KPZ equation \eqref{eq:formalKPZ} can then be obtained via the Hopf-Cole transformation by letting
\begin{equation} \label{eq:hopf-coleIntro}
h_\e(t,x) = \log u_\e(t,x), \quad h_0 = \log u_0,
\end{equation}
which by It\^o calculus solves the regularized KPZ equation
\begin{equation}\label{eq:KPZe}
 \frac{\partial}{\partial t} h_{\e} = \frac12 \D h_{\e} +  \bigg[\frac12   |\nabla h_\e |^2  - C_\e\bigg]+
 \b \e^{\frac{d-2}2}   \dB_{\e} \;,\quad\,\,   h_{\e}(0,x) = h_0(x),
\end{equation}
with $C_\e$ that is the diverging parameter
\begin{equation}\label{C-eps}
C_\e= \b^2 V(0) \e^{-2}/2.
\end{equation}

Recently, there have been a few papers studying the asymptotic behavior $u_\e$ and $h_\e$ as the mollification parameter is switched off ($\e\to 0$) \cite{MSZ16,GRZ18,MU17,DGRZ20,DGRZ18b,CCM20,CCM19b}, but most of the results that have been obtained are not optimal in the sense that they were not proved to hold in the optimal $\b$-region where they should hold. In this paper, we bring new tools inspired from the polymer literature in order to extend some of the results to their optimal $\b$-regions. Let us first introduce these $\b$-regions of interest.

\subsection{The weak disorder and the $L^2$-region}
It is a standard result of the polymer litterature that the following $0$-$1$ law holds:
\begin{equation}
\text{either:}\quad (i)\ \sZ_T \to  0\quad \IP\text{-a.s,} \quad \text{or:} \quad (ii) \,\sZ_T \to  \sZ_\infty \text{ where } \sZ_\infty> 0\quad  \IP\text{-a.s,}
\end{equation}
and that there is a critical parameter $\b_c\geq 0$ such that
 \begin{equation}
\sZ_\infty = 0 \text{ a.s. } \text{ if }\b>\b_c \quad \text{and} \quad
\sZ_\infty > 0 \text{ a.s. } \text{ if } \b<\b_c,
\end{equation} 
and $\b_c>0$ when $d\geq 3$. Moreover, (ii) holds if and only if $(\sZ_T)_{T\geq 0}$ is uniformly integrable, see\ \cite{MSZ16,CY06}.

When situation (i) holds, we say that the system is in the \emph{strong disorder} regime, while if (ii) holds, we say that it is in the \emph{weak disorder} regime.  Localization properties for the polymer model and the solution $u_\e$ of SHE have been shown to hold inside the strong disorder region, see \cite{BC20,C17,BM19}. On the other hand, the polymer is diffusive in the full weak disorder region \cite{CY06} (see Appendix \ref{app:LLNDirac} in the continuous case).

Another region of interest is the $L^2$-region $[0,\b_{L^2})$, where
\begin{equation} \label{eq:defBeta2}
\b_{L^2} = \sup\{\b\geq 0,\, \sup\nolimits_t\IE\sZ_t^2 = \IE[\sZ_\infty^2] < \infty\},
\end{equation} 
satisfies $\b_{L^2}\leq \b_c$. In fact, it is widely believed that $\b_{L^2} < \b_c$, this strict inequality being proved for the discrete polymer when $d\geq 3$ \cite{BGH11,BT10,BS10,BS11} (see in particular \cite[Section 1.4]{BS10} for $d=3,4$). The $L^2$-region has been first introduced in \cite{B89,IS88} to show that diffusivity could happen for the polymer for small $\b$. It has the particular interest of allowing second moment computations that are not accessible in the weak disorder region part where $\b \geq \b_{L^2}$. The region has also been recently involved in the study of the speed of convergence and fluctuations of the polymer partition function \cite{CL17,CN19}, as well as for asymptotic properties of $u_\e$ and $h_\e$ that we describe now.

\subsection{Presentation of the results}

It was first observed in \cite{MSZ16} that when $\b<\b_{L^2}$, the following law of large numbers holds if $u_0= u_\e(0,\cdot)$ is regular enough:
\begin{equation} \label{eq:LLN}
\int_{\mathbb R ^d} u_\e(t,x) f(x) \dd x \cvIP \int_{\mathbb R ^d} \bar{u}(t,x) f(x) \dd x,
\end{equation}
 as $\e\to 0$  for any test function $f\in \mathcal{C}_c^\infty$, where
\begin{equation} \label{eq:Def_ubar}
\partial_t \bar{u}(t,x) = \frac{1}{2} \Delta \bar{u}(t,x),\qquad \bar{u}(0,\cdot) = u_0.
\end{equation}
The reason behind \eqref{eq:LLN} is that $u_\e(t,x)$ and $u_\e(t,y)$ become asymptotically independent for $x\neq y$ when $\e\to 0$, which is well captured by a covariance computation when $\b<\b_{L^2}$.

Our first main result (cf. Section \ref{sec:LLNstatements}) is to extend \eqref{eq:LLN} to the full \emph{weak disorder} region.
 Extending \eqref{eq:LLN} to the full weak disorder region requires more sophisticated arguments than a covariance computation, in particular because second moments blow up when $\b\geq \b_{L^2}$. To get around this issue, we rely on both the central limit theorem for polymer paths \cite{CY06} (which was shown in the full weak disorder region) and a homogenization argument as in \cite{CN19} that builds on the uniform integrability of $(\sZ_T)_{T\geq 0}$. In addition, we obtain a similar convergence to \eqref{eq:LLN} in the case of Dirac initial condition $u_0=\delta_{x_0}$ (although the limit will be purely random), see Theorem \ref{LNDirac}.
 
Fluctuations in the law of large numbers \eqref{eq:LLN} have been studied in \cite{GRZ18} in the case where the noise can also be colored in time. The authors showed that for $\b$ sufficiently small (namely for $\b<\b_0$ with $\b_0<\b_{L^2}$) and for all $u_0\in\mathcal C_b(\mathbb R^d)$, 
\begin{equation} \label{eq:CV_toEW_intro}
\e^{1-\frac d 2}\int_{\rd} f(x) \left(u_\e(t,x) - \bar{u}(t,x)\right)  \dd x \cvlaw \int_{\rd} f(x)\, \kU_1(t,x) \dd x\;,
\end{equation}
with $\mathscr U_1$ solving
the  stochastic heat equation with additive noise (also called the Edwards-Wilkinson (EW) equation):
\begin{equation}\label{eq:EW_GRZ}
\partial_t \kU_1(t,x)= \frac 12 \Delta \kU_1(t,x)+ \bar{u}(t,x) {\gamma}(\beta)  \xi(t,x),\qquad \kU_1(0,x)=0,
\end{equation}
where ${\gamma}(\beta)$ is defined in \eqref{def-sigma-beta}. To prove this convergence, the authors relied on Malliavin calculus techniques.

Our second main result (c.f.\ Section \ref{sec:ResultsSHE}) is a proof that the Edwards-Wilkinson regime described in \eqref{eq:CV_toEW_intro} can be extended to the region $\b\in (0,\b_{L^2})$ with joint in time convergence. Note that the parameter ${\gamma}(\beta)$ is finite in $(0,\b_{L^2})$ and blows up precisely at $\b_{L^2}$, indicating that \eqref{eq:CV_toEW_intro} should indeed hold up to that point (and not above), but the proof in \cite{GRZ18} does not carry up until there. In order to reach $\b_{L^2}$, we follow here a distinct strategy of proof that builds on stochastic calculus and the martingale central limit theorem, in the same spirit as in \cite{CN19,CCM19b,CL17,CN95}. Like in \cite{CN19}, the additional ingredients that we use to control the  quadratic variations and the cross-brackets of the martingales are the local limit theorem for polymers (Theorem \ref{th:errorTermLLT} in the appendix, see also \cite{CCM20,V06,Si95}) and a homogenization argument, c.f.\ Section \ref{sec:heuristicsEW} for a quick description of the proof. We believe that in addition to leading us to $\b_{L^2}$, this approach to the problem has the advantage of being quite straightforward and simple and that it should be applicable in different contexts (for example discrete structures). Furthermore, the method gives directly multi-dimensional in time convergence in \eqref{eq:CV_toEW_intro}, which has not been considered in \cite{GRZ18}. 

Let us turn to results on the KPZ equations. In \cite{DGRZ20}, the authors have considered the \emph{flat initial condition} $h_0\equiv 0$ and proved that there is some $\b_0<\b_{L^2}$ such that for all $\b<\b_0$,
\begin{equation} \label{eq:EWKPZintro}
\e^{1-\frac d 2}\int_{\rd} f(x) \left(h_\e(t,x)-\IE[h_\e(t,x)]\right)  \dd x \cvlaw \int_{\rd} f(x)\, \kU_{1}(t,x) \dd x,
\end{equation}
with $\bar u(t,x) = 1$ in the definition of $\kU_{1}(t,x)$. See also \cite{MU17} for another approach to the problem. Both of these proofs are restricted to small $\b$ and some additional arguments have to be brought to go up to $\b_{L^2}$. In particular, higher moments for the partition function (which do not hold when approaching $\b_{L^2}$) are technically required in the proof of \cite{DGRZ20}, while the renormalization approach in \cite{MU17} is restricted by essence to a small $\b$ region.

We prove in this paper (c.f.\ Section \ref{sec:resultsOnKPZ}) that \eqref{eq:EWKPZintro} extends to $(0,\b_{L^2})$ with joint in time convergence and general continuous initial conditons. Note that in \cite{MU17}, the authors consider  initial conditions at a macroscopic level (i.e.\ $u_\e(0,x) = u_0(\e^{-1}x)$), while ours are at a microscopic level (i.e.\ $u_\e(0,x) = u_0(x)$) and contrary to our result (see Theorem \ref{th:EWlimitKPZ}), the dependence in the initial condition disappears in the limit in their case. Our strategy of proof is based on the same martingale method as for SHE and relies on It\^o's formula. The proof requires significant additional work compared to the SHE case. One of the difficulties to overcome is to deal with ratios of partition functions and we do so by relying on a concentration inequality for $\sZ_T$ proven in \cite{CCM20} (see also \cite{CH02}) and precise estimates in the local limit theorem for polymers that we obtain in Theorem \ref{th:errorTermLLT}. Another issue to tackle is to control the It\^o correction and show that it satisfies some homogenization property by studying some 4-th moments.

As we just mentioned, we needed some precise estimates (in particular a speed  of convergence) on the local limit theorem for polymers which were not accessible via the pre-existing statements \cite{CCM20,V06,Si95}. Our Theorem \ref{th:errorTermLLT} improves the previous ones on the following points: (i) it is not restricted to the diffusive scale between the starting and ending points, in particular it shows that the error still vanishes up to a linear scale; (ii) it gives a speed of convergence with exponents that do not depend on $\b$. We think that these improvements are interesting on their own and may be useful for other purposes.

\subsection{Results on convergence to stationary Edwards-Wikinson and GFF}
As additional results, we show (stationary) Edwards-Wilkinson asymptotics for the \emph{stationary solutions} of the regularized SHE and KPZ (cf.\ Theorem \ref{th:CV_stationarySHE} and Theorem \ref{th:CV_stationaryKPZ}). These considerations have a counterpart in the polymer framework, namely that the diffusively rescaled and normalized (infinite-horizon) partition function $\sZ_\infty(x)$ and log-partition function $\log \sZ_\infty(x)$ converge to a Gaussian Free Field (cf.\ Corollary \ref{th:PolymerCV} and Corollary \ref{th:LogPolymerCV}), which we believe are an interesting results on their own.

 The notion of stationary solutions of the regularized equations can be defined as follows. If $\xi\to \mathcal V(\xi)$ denotes a measurable function such that $\mathcal V(\xi) = \sZ_\infty(0,\xi)$, then (similarly to \eqref{eq:uZ})
\begin{equation} \label{eq:defustat}
\quad u^{stat}_\e(t,x) = \mathcal V (\xi^{\ssup{\e,t,x}}),
\end{equation}
defines a stationary solution of the SPDE in \eqref{eq:SHEintro}, where by stationary we mean that the distribution of $u^{stat}_\e(t,x)$ is invariant by time-space translation (see \cite{CCM20,CCM19b,DGRZ18b}). 

The function $u^{stat}_\e$ plays a central role in the \emph{pointwise} asymptotic approximation of $u_{\e}$ (in contrast to \eqref{eq:LLN} which is averaged), at first order, for diffent types of initial conditions $u_0$ (continous and bounded, Dirac). For example, for any continous and bounded $u_0$, it holds that for all $\b<\b_{L^2}$, as $\e \to 0$,
\[\forall t>0,\,x\in\mathbb R^d, \quad u_\e(t,x) - \bar{u}(t,x)u^{stat}_\e(t,x) \cvLtwo 0,\]
see \cite{DGRZ18b,CCM20} for further discussion on this matter. 

In the KPZ case, a stationary solution to the SPDE in \eqref{eq:KPZe} is given by the Hopf-Cole transform
\begin{equation} \label{eq:defhstat}
 h^{stat}_\e = \log u^{stat}_\e,
\end{equation}
and as for SHE, $h^{stat}_\e$ is a central object in the \emph{pointwise} approximation (of first order) of $h_\e(t,x)$ as $\e\to 0$ for a variety of initial conditions, c.f.\ \cite{CCM20,DGRZ18b}. In particular, when $\b<\b_c$ and $h_0=0$,
\begin{equation} \label{eq:pointwiseApproxKPZ}
\forall t>0,\, x\in\mathbb R^d, \quad h_\e(t,x) - h^{stat}_\e(t,x) \cvIP 0.
\end{equation}
The (pointwise) fluctuations in the convergence \eqref{eq:pointwiseApproxKPZ} have been studied in \cite{CCM19b} deep in the $L^2$-region, and the limit of the correctly rescaled fluctuations involves a convolution of the heat kernel with a Gaussian free field. By a heuristic combination of this result and of \eqref{eq:EWKPZintro}, it was conjectured therein that convergence \eqref{eq:statKPZresult} below should hold in the $L^2$-region. Our method enables us to prove this property (Theorem \ref{th:CV_stationaryKPZ}).

\subsection{Additonal comments}
Similarly to the discussion in \cite{CN19}, it is not believed that \eqref{eq:CV_toEW_intro} nor \eqref{eq:EWKPZintro} extend to the region $[\b_{L^2},\b_c)$ since $\gamma(\b)$ blows up at $\b_{L^2}$. In particular, different scalings and limits for the fluctuations are expected above $\b_{L^2}$.

Let us end this introduction by mentioning some related works. Analogous Edwards-Wilkinson type of fluctuations for the partition functions of the discrete polymer (or equivalently discrete SHE and KPZ equation) when $d\geq 3$ have been obtained, in the full $L^2$-region, by the authors of \cite{LZ20} simultaneously and independently to our work. Their approach relies on chaos expansion and the 4-th moment theorem as used in \cite{CSZ17b} and is quite different from ours. This powerful technique allows them in particular to consider moments slightly above $2$ for the partition function. They moreover introduce a thoughtful martingale argument in order to overcome the technical difficulties that arise when dealing with the log-partition function (i.e.\ KPZ) in the full $L^2$-region.

In dimension $d=2$, similar questions on the regularized SHE and KPZ equation have been studied in a series of papers \cite{CSZ17b,CD18,G18,CSZ18a,CSZ18b}. Edwards-Wilkinson type of results for the KPZ equation have been proved to hold in a restricted part of the $L^2$-region in \cite{CD18,G18}, and extended to the full $L^2$-region in \cite{CSZ18b}. Interesting phenomena appear at the corresponding critical point $\b_{L^2}$ (which equals $\b_c$ in this case) \cite{GQT19,CSZ18a,CSZ19}. As discussed above, the situation is quite different in dimension $d=1$ and the corresponding $u_\e$ and $h_\e$ actually converge to random continuous processes that coincide with the notion of solutions of SHE and KPZ. Nevertheless, the link with directed polymers still holds and there is moreover a notion of polymer on white noise (not mollified) environment, see \cite{AKQ14,AKQ14b,CSZ17a} for more details.

\section{Statement of the results}\label{sec:statResults}
\subsection{On the law of large number}
\label{sec:LLNstatements}
\begin{theorem}[Case of regular initial condition] \label{th:LLNreg} Suppose that $u_0$ is continuous and bounded. Then, if weak disorder holds, we have that for any test function $f\in \mathcal{C}_c^\infty$,
\begin{equation} \label{eq:LLNreg}
\int_{\mathbb R ^d} u_\e(t,x) f(x) \dd x \cvIP \int_{\mathbb R ^d} \bar{u}(t,x) f(x) \dd x,
\end{equation}
as $\e\to 0$, where $\bar u$ is as in \eqref{eq:Def_ubar}.
\end{theorem}

\begin{cor}[The stationary case]\label{th:LLNstat} Recall $u^{stat}_\e$ from \eqref{eq:defustat}. If weak disorder holds, then for $f\in \mathcal{C}_c^\infty$, as $\e\to 0$,
\begin{equation} \label{eq:LLNstat}
\int_{\mathbb R ^d} u^{stat}_\e(t,x) f(x) \dd x \cvIP \int_{\mathbb R ^d} f(x) \dd x.
\end{equation}
\end{cor}

\begin{theorem}[Case of Dirac initial condition] \label{LNDirac}\label{th:LLNDirac} Suppose that $u_0(x) = \delta_{x_0}$ and that weak disorder holds. Then, for all $f\in \mathcal{C}_c^\infty$,
\begin{equation} \label{eq:LLNdirac}
\int_{\mathbb R ^d} u_\e(t,x) f(x) \dd x \cvlaw \sZ_\infty \int_{\mathbb R ^d} \bar{u}(t,x) f(x) \dd x,
\end{equation}
as $\e\to 0$, where $\bar u$ is as in \eqref{eq:Def_ubar}, i.e.\ $\bar{u}(t,x) = \rho_t(x-x_0)$ with $\rho_t(x)$ is the standard heat kernel.
\end{theorem}
\begin{rem} In the Dirac case, $u_\e(t,x)$ can be represented by a product of $\rho_t(x)$ and a polymer partition function that starts at $(0,\e^{-1}x_0)$ and ends at $(\e^{-2}t,\e^{-1}x)$, see  \eqref{eq:diracIC}. When $\b<\b_{L^2}$, the local limit theorem for polymers (Theorem \ref{th:errorTermLLT})  states that this partition function factorizes in the product $\sZ_\ell(\e^{-1} x_0) \overset{\leftarrow}{\sZ}_{\e^{-2}t,\ell}(\e^{-1} x)$ with $\ell=o(\e^{-2})$ and $\ell\to\infty$ and where $ \overset{\leftarrow}{\sZ}_{T,\ell}(x)$ the time reversed partition function defined in \eqref{eq:timeRevPartitionf}. 
When $\e$ is small, $\overset{\leftarrow}{\sZ}_{\e^{-2}t,\ell}(\e^{-1} x)$ and $\overset{\leftarrow}{\sZ}_{\e^{-2}t,\ell}(\e^{-1} x')$ are morally independent for $x\neq x'$, hence one expects that
\begin{align*}
\int_{\mathbb R ^d} u_\e(t,x) f(x) \dd x &\approx \int_{\mathbb R ^d} \rho_t(x) \sZ_\ell(\e^{-1} x_0) \overset{\leftarrow}{\sZ}_{\e^{-2}t,\ell}(\e^{-1} x) f(x) \dd x\\
& \approx \sZ_\infty(\e^{-1} x_0)  \int_{\mathbb R ^d} \rho_t(x) \IE \overset{\leftarrow}{\sZ}_{\e^{-2}t,\ell}(\e^{-1} x) f(x) \dd x\\
& \eqlaw \sZ_\infty \int_{\mathbb R ^d} \rho_t(x)  f(x) \dd x.
\end{align*}
Unfortunately, there is no proof of a local limit theorem for $\b\geq \b_{L^2}$ and different arguments have to be brought in the higher $\b$-part of the weak disorder region.
\end{rem}

\subsection{On the fluctuations for SHE} \label{sec:ResultsSHE}
In the following, $\mathcal C^\infty_c$ denotes the set of infinitely differentiable, compactly supported functions on $\mathbb R^d$.
\begin{theorem}[Case of regular initial condition] \label{th:EWlimit} Suppose $u_\e(0,\cdot)= u_0$ where $u_0$ is a continuous and bounded function. For all $\b<\b_{L^2}$, $t_1,\dots,t_n\geq 0$ and $f_1,\dots,f_m\in\mathcal C^\infty_c$, the following convergence holds jointly for $i\leq n$ and $j\leq m$ as $\e\to 0$,
\begin{equation} \label{eq:EWGICSHEResult}
\e^{1-\frac d 2}\int_{\rd} f_j(x) \left(u_\e(t_i,x) - \bar{u}(t_i,x)\right)  \dd x \cvlaw \int_{\rd} f_j(x)\, \kU_1(t_i,x) \dd x\;,
\end{equation}
with $\kU_1$ defined in \eqref{eq:EW_GRZ} and $\bar u$ in \eqref{eq:Def_ubar}.
\end{theorem}
 
\begin{theorem}[The stationary case] \label{th:CV_stationarySHE}
For all $\b<\b_{L^2}$, $t_1,\dots,t_n\geq 0$ and $f_1,\dots,f_m\in\mathcal C^\infty_c$, the following convergence holds jointly for $i\leq n$ and $j\leq m$ as $\e\to 0$,
\begin{equation}
\e^{1-\frac d 2}\int_{\rd} f_j(x) (u^{stat}_\e(t_i,x)-1) \, \dd x \cvlaw \int_{\rd} f_j(x)\, \kU_2(t_i,x) \dd x\;,
\end{equation}
with $\kU_2$ being the \textbf{stationary} solution of the EW equation:
\begin{equation}\label{EW2}
\partial_t \kU_2(t,x)= \frac 12 \Delta \kU_2(t,x)+ {\gamma}(\beta) \xi(t,x),\qquad \kU(0,x)\eqlaw \mathscr H(x),
\end{equation}
where $\gamma(\b)$ is given in \eqref{def-sigma-beta} and $\mathscr H$ is the Gaussian free field on $\mathbb{R}^d$ satisfying:
\begin{equation}\label{cov:H:GFF}
\mathrm{Cov}\big(\mathscr H(x),\mathscr H(y)\big)=\gamma^2(\b)\,\int_0^\infty \rho_{2\sigma}(x-y) \dd \sigma=  \, \frac{\gamma^2(\b) \,\Gamma(\frac d2-1)}{\pi^{\frac{d}{2}}} \frac 1 {|x-y|^{d-2}}.
\end{equation}
\end{theorem}

The particular case when $t=0$ in Theorem \ref{th:CV_stationarySHE} gives that $\e^{1-\frac{d}{2}}(u^{stat}_\e(0,\cdot)-1)\cvlaw \mathscr H$ in distribution, which entails the following result since $u^{stat}_\e(0,\cdot) \eqlaw \sZ_\infty(\cdot)$.
\begin{cor}[Gaussian free field limit of the polymer partition function] \label{th:PolymerCV}
For all $\b<\b_{L^2}$, the following convergence holds jointly for finitely many $f\in\mathcal C^\infty_c$ as $T\to\infty$, 
\begin{equation}
\int_{\mathbb{R}^d} f(x)\, T^{\frac{ (d-2)}{ 4}}\left(\mathcal{Z}_{\infty}(\sqrt T x)-1\right) \dd x \cvlaw \int_{\mathbb{R}^d} f(x)\, \mathscr H(x) \dd x,
\end{equation}
i.e.\  the spatially rescaled and centered normalized partition function $T^{\frac{ (d-2)}{ 4}}(\mathcal{Z}_{\infty}(\sqrt T\, \cdot)-1)$ converges in distribution towards the Gaussian free field $\mathscr H$.
\end{cor}

\subsection{On the fluctuations for the KPZ equation} \label{sec:resultsOnKPZ}
Let $h_\e$ be a solution of the regularized KPZ equation satisfying \eqref{eq:hopf-coleIntro}-\eqref{eq:KPZe}. 
\begin{theorem}[Case of regular initial conditions] \label{th:EWlimitKPZ} Suppose that
  $0<\inf_{x\in\R}u_0(x)\leq \sup_{x\in\R}u_0(x)<\infty$ and $\b<\b_{L^2}$. For all $\b<\b_{L^2}$, $t_1,\dots,t_n\geq 0$ and $f_1,\dots,f_m\in\mathcal C^\infty_c$, the following convergence holds jointly for $i\leq n$ and $j\leq m$ as $\e\to 0$,
\[
\e^{1-\frac d 2}\int_{\rd} f_j(x) \left(h_\e(t_i,x)-\IE[h_\e(t_i,x)]\right)  \dd x \cvlaw \int_{\rd} f_j(x)\, \kU_{3}(t_i,x) \dd x,
\]
with $\mathscr U_3$ solving
the following Edwards-Wilkinson type of  stochastic heat equation with additive noise:
\begin{equation} \label{eq:defU3}
 \partial_t  \kU_3(t,x)=\frac{1}{2} \Delta \kU_3(t,x) + \nabla \log{ \bar{u}(x,\tau)} \cdot\nabla \kU_3(t,x) +\xi(x,\tau).
 \end{equation}
\end{theorem}
\begin{rem}
The solution of the above equation satisfies $\kU_3(t,x)=\bar{u}(x,t)^{-1}\kU_1(t,x)$ in distribution.
\end{rem}

\begin{theorem}[The stationary case] \label{th:CV_stationaryKPZ}
Recall $h^{stat}_\e$ from \eqref{eq:defhstat}. For all $\b<\b_{L^2}$, $t_1,\dots,t_n\geq 0$ and $f_1,\dots,f_m\in\mathcal C^\infty_c$, the following convergence holds jointly for $i\leq n$ and $j\leq m$ as $\e\to 0$,
\begin{equation} \label{eq:statKPZresult}
\e^{1-\frac d 2}\int_{\rd} f_j(x) (h^{stat}_\e(t_i,x)-\IE[h^{stat}_\e(t_i,x)]) \, \dd x \cvlaw \int_{\rd} f_j(x)\, \kU_2(t_i,x) \dd x\;,
\end{equation}
with $\mathscr U_2$ from \eqref{EW2}.
\end{theorem}

Again, letting $t=0$ in the above theorem leads to the following:
\begin{cor}[Gaussian free field limit of the polymer log-partition function] \label{th:LogPolymerCV}
For all $\b<\b_{L^2}$, the following convergence holds jointly for finitely many $f\in\mathcal C^\infty_c$ as $T\to\infty$, 
\begin{equation}
\int_{\mathbb{R}^d} f(x)\, T^{\frac{ (d-2)}{ 4}}\left(\log \mathcal{Z}_{\infty}(\sqrt T x)-\IE\log \mathcal{Z}_{\infty}(\sqrt T x)\right) \dd x \cvlaw \int_{\mathbb{R}^d} f(x)\, \mathscr H(x) \dd x,
\end{equation}
i.e.\  the rescaled and centered log-partition function $T^{\frac{ (d-2)}{ 4}}(\log \mathcal{Z}_{\infty}(\sqrt T\, \cdot)-\IE\log \mathcal{Z}_{\infty}(\sqrt T\, \cdot))$ converges in distribution towards the Gaussian free field $\mathscr H$.
\end{cor}

\paragraph{Structure of the article}
 Section \ref{sec:LLN} is dedicated to the proof of the laws of large numbers from Section \ref{sec:LLNstatements}.
Section \ref{sec:EWproof} concentrates on showing the Edwards-Wilkinson fluctations of Section \ref{sec:ResultsSHE} and the section starts with a concise summary of the proof. The Edwards-Wilkinson results for the KPZ equation (Section \ref{sec:resultsOnKPZ}) are presented in Section \ref{sec:proofEWKPZ}. The proof builds on the martingale strategy for the SHE so we recommend the interested reader to begin with Section \ref{sec:ResultsSHE}.

\paragraph{Some notations}
\begin{itemize}
\item $\IE$ and $\IP$ (resp.\ $\DE_x$ and $\DP_x$) denote the expectations and probability with respect to the white noise $\xi$ (resp.\ with respect to the Brownian motion $B$ started at $x$).
\item $\DP_{0,x}^{t,y}$ and $\DE_{0,x}^{t,y},$ be the probability measure and expectation of the Brownian bridge from $(0,x)$ to $(t,y)$.
\item $\rho_t(x)=(2\pi t)^{-\frac{d}{2}}e^{-|x|^2/(2t)}$ the heat kernel and $\rho_t(x,y)=\rho_t(x-y)$.
\item $(\mathcal F_t)_{t\geq 0}$ is the filtration associated to the white noise.

\item $\Phi_T(B,\xi) = \exp\left\{\beta \int_0^T \int_{{\mathbb R}^d} \phi(y-B_s) \ \xi(\dd s,\dd y)-\frac{\b^2 V(0)T}{2}\right\}$ which satisfies
\begin{equation}
\sZ_T(x,\xi) = \DE_x \left[\Phi_T(B,\xi)\right].\end{equation}
 Note that we will repeatedly use the fact that $\IE[\sZ_T(x,\xi)] = \IE[\Phi_T] =1$.
\item ${\gamma}^2(\beta)$ is given by:
\begin{equation}\label{def-sigma-beta}
 \begin{aligned}
\gamma^2(\beta) & = \b^2 \int_{\rd}  \,\, V(Z)\,\,\DE_Z\bigg[\mathrm e^{\beta^2\int_0^\infty V( B_{2s})\,\dd s}\bigg] \dd Z,\\
 & = \b^2 \int_{\mathbb R^d} V(Z) \IE[\sZ(u) \sZ(u+Z)]\dd Z, \quad \forall u\in\mathbb R^d.
\end{aligned}
\end{equation}
\item When switching from solutions $u_\e$ and $h_\e$ to their polymer representations, we use $T=T(\e)=\e^{-2}$ to denote the time horizon of the polymer.
\item We denote by $\overset{\leftarrow}{\sZ}_{T,\ell}(z)$ the time-reversed partition function of time horizon $\ell$:
\begin{equation} \label{eq:timeRevPartitionf}
\overset{\leftarrow}{\sZ}_{T,\ell}(z) = \DE_{z} \left[\exp\left\{\int_{T-\ell}^T \int_{\mathbb R^d} \phi(B_{T-s}-y)  \xi(\dd s,\dd y) -\frac{\b^2  V(0) \ell}{2}\right\} \right].
\end{equation}
\item Given $z=(z_1,\cdots,z_d)\in\R^d$,  $\ell>0$, and $T$, we denote by $\mathtt{R}_{\ell,T}(z)$ the cuboid in $[0,\infty)\times \R^d$
  \begin{align}
{\mathtt{R}_{\ell,T}(z)=[0,T]\times \prod_{i=1}^d[z_i-\ell,z_i+\ell].} \label{Def:CUBE}
\end{align}
\item We define the constrained partition function as
  \begin{align}
\overline{\sZ}_{\ell}(\z)&=\DE_{\z}\left[\Phi_{\ell};\,B[0,\ell]\subset \mathtt{R}_{\ell,\ell}(\zeta) \right],\label{eq:barZ}
  \end{align}
  where the event $\{B[0,\ell]\subset \mathtt{R}_{\ell,\ell}(\zeta)\}$ means that the graph of $B$ in time $[0,\ell]$ lies in $\mathtt{R}_{\ell,\ell}(\zeta)$.
\end{itemize}

\textbf{Note:} throughout the paper and if clear from the context, the constant $C$ that appears in successive upper-bounds may take different values.

\section{Proofs of the law on large numbers results} \label{sec:LLN}

Throughout this section, we set $T=\ve^{-2}$. 

\subsection{Case of Dirac initial conditions $u_0 = \delta_{x_0}$} 
For Dirac initial condition $u_0=\delta_{x_0}$, the solution $u_\e$ of \eqref{eq:SHEintro} admits the representation (see equation (2.11) in \cite{CCM20}):
\begin{equation} \label{eq:diracIC}
u_\e(t,x) = \rho_t(x-x_0) \DE_{0,0}^{Tt,\sqrt{T}(x-x_0)}\left[\Phi_{T t}(B,\dB_{(\e,x_0)})\right],
\end{equation}
where $\dB_{(\e,x_0)}(s,y)= \e^{\frac{d+2}2}\dB(\e^2 s, \e y + x_0 )$ is a diffusively rescaled white noise (without time-reversal), which again \emph{has the same law as} $\xi$.
Hence, letting $\mathtt{E}_{x}^{\b,T}$ denote the expectation under the polymer measure
\[\dd \mathtt{P}_{x}^{\b,T} = \frac{\Phi_T}{\sZ_T(x)} \dd \DP_x,\]
 of time horizon $T$ with  inverse temperature $\b$ started at $B_0=x\in\mathbb R^d$, we have by the heat kernal scaling $\rho(T t,\sqrt{T}x) = T^{-\frac{d}{2}} \rho_t(x)$ and the change of variable $X=\sqrt{T}(x-x_0)$, 
\begin{align}
\int_{\mathbb R ^d} u_\e(t,x) f(x) \dd x & \eqlaw \int_{\mathbb R ^d} \rho_{Tt}(X) \DE_{0,0}^{Tt,X}\left[\Phi_{T t}(B,\xi)\right] f(T^{-\frac{1}{2}} X + x_0) \dd X \nonumber \\
& = \sZ_{Tt} \, \mathtt{E}_{0}^{\b,Tt}[f(T^{-\frac{1}{2}} B_{T t} + x_0)].\label{eq:IntexprToPolymerMeasure}
\end{align}

 Then, the continuous analogue of the CLT in \cite{CY06} implies that in the weak disorder,
\begin{equation}
\sZ_{Tt} \, \mathtt{E}_{0}^{\b,Tt}[f(T^{-\frac{1}{2}} B_{T t} + x_0)] {\cvLone} \sZ_\infty\,  \DE_{x_0}[f(B_t )],\label{eq:CYCLT}
\end{equation}
 as  $T\to\infty$. 
  The proof is almost the same as \cite[Theorem 5.1]{CY06}, see Appendix \ref{app:LLNDirac}, and hence  Theorem \ref{th:LLNDirac} holds.

\subsection{Case of continuous and bounded initial condition} \label{subsec:LLNuzero}
 In this case, the Feynmann-Kac formula leads to the following representation similar to \eqref{eq:uZ} (see also (2.4) in \cite{CCM20}) with $\xi^{(\e,t)}\eqlaw \xi$,
\begin{align} 
u_\e(t,x) & = \DE_{\sqrt{T} x}  \left[  u_0(T^{-\frac{1}{2}} B_{Tt}) \, \Phi_{Tt}(B;\xi^{\ssup{\e,t}})\right]\label{eq:FKgeneralIC}\\
&{ \eqlaw  \DE_{\sqrt{T} x}  \left[  u_0(T^{-\frac{1}{2}} B_{Tt}) \, \Phi_{Tt}(B;\xi)\right] }\label{eq:FKgeneralIC2}\\
&={ \sZ_{Tt}(\sqrt T x) \mathtt{E}_{\sqrt{T}x}^{\b,Tt}\left[u_0(T^{-\frac{1}{2}}B_{Tt})\right]},\notag
\end{align} 
where the equality in distribution holds jointly in $x\in\mathbb R^d$ for any fixed $t$.

Therefore, we know that \begin{align}
\int_{\R^d}u_\e(t,x)f(x)\dd x
& { \eqlaw \int_{\R^d} f(x)\sZ_{Tt}(\sqrt{T}x)\mathtt{E}_{\sqrt{T}x}^{\b,Tt}\left[u_0(T^{-\frac{1}{2}}B_{Tt})\right]\dd x.}\label{eq:eq:IntexprToPolymerMeasure2}
\end{align}



Now, by shift-invariance in distribution of the white noise, and thus of the partition function and the polymer measure,
\begin{align*}
& \IE\left[\left| \sZ_{Tt}(\sqrt{T} { x}) \mathtt{E}_{\sqrt{T} { x}}^{\b,Tt}[{ u_0}(T^{-\frac{1}{2}} B_{T t})] - \sZ_{Tt}(\sqrt{T} { x}) \DE_{ x}[{ u_0}(B_t)] \right| \right]\\
& = \IE\left[\left| \sZ_{Tt}(0) \mathtt{E}_{0}^{\b,Tt}[{ u_0}(T^{-\frac{1}{2}} B_{T t} + { x})] - \sZ_{Tt}(0) \DE_{ x}[{ u_0}(B_t)] \right| \right]\to 0,
\end{align*}
where the convergence is as $\e\to 0$ for all fixed ${ x}\in\mathbb R^d$ {by \eqref{eq:CYCLT}.} 

This implies with dominated convergence theorem that 
\[ \IE\left[ \left|\int f(x) \sZ_{Tt}(\sqrt{T} x) \mathtt{E}_{\sqrt{T} x}^{\b,Tt}[u_0(T^{-\frac{1}{2}} B_{T t})] \dd x - \int f(x) \sZ_{Tt}(\sqrt{T} x) \DE_x[u_0(B_t)] \dd x \right| \right] \to 0.
\]
Combined with \eqref{eq:eq:IntexprToPolymerMeasure2}, what remains to show in order to prove Theorem \ref{th:LLNreg} is that
\[ \int f(x) \sZ_{Tt}(\sqrt{T} x) \DE_x[u_0(B_t)] \dd x \rightarrow  \int f(x) \DE_x[u_0(B_t)] \dd x,\]
where the RHS equals $\int \bar{u}(t,x) f(x) \dd x$.
\begin{lemma}
For any $t> 0$,
$$\lim_{T\to \infty}\mathbb{E}\left[\left|\int u_0(y) \sZ_{Tt}(\sqrt{T} y) \DE_y[f(B_t)] \dd y - \int u_0(y)  \DE_y[f(B_t)] \dd y\right|\right]=0.$$
\end{lemma}

\begin{proof}

 Let us define for $\zeta\in \mathbb{R}^d$, $\ell>0$, and $M>0$			
\begin{align*}
\overline{\sZ}_{\ell}(\z)&=\DE_{\z}\left[\Phi_{\ell};\,B[0,\ell]\subset \mathtt{R}_{\ell,\ell}(\zeta) \right]\\
\widetilde{\sZ}_{\ell,M}(\z)&=\overline{\sZ}_{\ell}(\z) \land M,\notag
\end{align*}
where the event $\{B[0,\ell]\subset \mathtt{R}_{\ell,\ell}(\zeta)\}$ means that the graph of Brownian motion $B$ in time $[0,\ell]$ lies in $\mathtt{R}_{\ell,\ell}(\zeta)$. In this section, we simply write 
$$\widetilde{\sZ}_T(\z)=\widetilde{\sZ}_{T^{\frac{1}{4}},T^{\frac{d}{4}}}(\zeta).$$
Since $\overline{\sZ}_{T^{\frac{1}{4}}}(\sqrt{T}x)\leq \sZ_{T^{\frac{1}{4}}}(\sqrt{T}x)$ for $x\in\R^d$, we have   
\al{
&\quad \mathbb{E}\left|\widetilde{\sZ}_T(\sqrt{T}x)-\sZ_{Tt}(\sqrt{T}x)\right|\\
&\leq \mathbb{E}\left|\widetilde{\sZ}_{T^{\frac{1}{4}},T^{\frac{d}{4}}}(\sqrt{T}x)-\overline{\sZ}_{T^{\frac{1}{4}}}(\sqrt{T}x)\right|+\mathbb{E}\left|\overline{\sZ}_{T^{\frac{1}{4}}}(\sqrt{T}x)-\sZ_{T^{\frac{1}{4}}}(\sqrt{T}x)\right|\\
& \quad +
\mathbb{E}\left|\sZ_{T^{\frac{1}{4}}}(\sqrt{T}x)-\sZ_{Tt}(\sqrt{T}x)\right|\\
&\leq \mathbb{E}\left[\overline{\sZ}_{T^{\frac{1}{4}}};~\overline{\sZ}_{T^{\frac{1}{4}}}\geq T^{\frac{d}{4}}\right]+\mathbb{E}\DE_{0}\left[\Phi_{T^{\frac{1}{4}}};\,B[0,T^{\frac{1}{4}}]\not\subset \mathtt{R}_{T^{\frac{1}{4}},T^{\frac{1}{4}}}(0)\right]+\mathbb{E}\left|\sZ_{T^{\frac{1}{4}}}-\sZ_{Tt}\right|\\
&\leq \mathbb{E}\left[\sZ_{T^{\frac{1}{4}}};~\sZ_{T^{\frac{1}{4}}}\geq T^{\frac{d}{4}}\right]+\DP_{0}\left(B[0,T^{\frac{1}{4}}]\not\subset \mathtt{R}_{T^{\frac{1}{4}},T^{\frac{1}{4}}}(0)\right)+\mathbb{E}\left|\sZ_{T^{\frac{1}{4}}}-\sZ_{Tt}\right|,
}
all of which converge to $0$ uniformly in $y$ by uniform integrability of $\sZ_T$.

 Next we will prove
\ben{\label{homogen-general}
\lim_{\e\to 0}\mathbb{E}\left(\int_{\R^d} \left(f(x) \widetilde{\sZ}_T(\sqrt{T} x) \DE_x[u_0(B_t)]  -  f(x) \mathbb{E}[\widetilde{\sZ}_T(\sqrt{T} x)] \DE_x[u_0(B_t)] \right)\dd x\right)^2= 0.
}
Indeed, since if $|x-y|>T^{-\frac{1}{4}}(1{ +2\delta_\phi})$ {where $\delta_\phi>0$ is a constant such that $\textrm{Supp}(\phi)\subset \{x: |x|<\delta_\phi\}$}, then $\widetilde{\sZ}_T(\sqrt{T} x)$ and $\widetilde{\sZ}_T(\sqrt{T} y)$ are independent, 
\al{
&\quad \mathbb{E}\left(\int_{\R^d} \left(f(x) \widetilde{\sZ}_T(\sqrt{T} x) \DE_x[u_0(B_t)]  -  f(x) \mathbb{E}[\widetilde{\sZ}_T(\sqrt{T} x)] \DE_x[u_0(B_t)] \right)\dd x\right)^2\\
&=\iint_{\R^d\times \R^d} \dd y \dd z f(x)f(y)\DE_x[u_0(B_t)]   \DE_y[u_0(B_t)]\\
& \qquad \qquad \qquad \qquad \times  \mathbb{E}\left[\left(\widetilde{\sZ}_T(\sqrt{T} x) - \mathbb{E}[\widetilde{\sZ}_T(\sqrt{T} x)]   \right)\left(\widetilde{\sZ}_T(\sqrt{T} y) - \mathbb{E}[\widetilde{\sZ}_T(\sqrt{T}y)]\right)\right]\\
&\leq \iint_{|x-y|\leq T^{-\frac{1}{4}}(1+2\delta_\phi)}\dd y \dd z \left|f(x)f(y)\DE_x[u_0(B_t)]   \DE_y[u_0(B_t)]	\right|			\\
& \qquad \qquad \qquad \qquad \times  \left|	\mathbb{E}\left[\left(\widetilde{\sZ}_T(\sqrt{T} x) - \mathbb{E}[\widetilde{\sZ}_T(\sqrt{T} x)]   \right)\left(\widetilde{\sZ}_T(\sqrt{T} y) - \mathbb{E}[\widetilde{\sZ}_T(\sqrt{T}y)]\right)\right]\right|		\\
&\leq CT^{-\frac{d}{4}}\mathbb{E}\left[\widetilde{\sZ}_T(0)^2\right],
}
which converges to $0$ as $\e\to 0$ by the lemma below and uniform integrability of $(\sZ_T)_{T\geq 0}$.
\begin{lemma}{ \cite[Lemma 3.3]{CN19}}\label{CN Lemma 3.3}
Let $(X_{t})_{t\geq 0}$ be a non-negative, uniformly integrable family of random variables. Then, for any $a(t)\to\infty$ as $t\to\infty$, $$a(t)^{-1}\IE\left[\left(X_{t} \wedge a(t)\right)^2\right] \to 0.$$
\end{lemma}
Putting things together, it follows that 
\al{
&\quad \mathbb{E}\left[\left|\int f(x) \sZ_{Tt}(\sqrt{T} x) \DE_x[u_0(B_t)] \dd x - \int f(x)  \DE_y[u_0(B_t)] \dd x\right|\right]\\
&\leq \mathbb{E}\left[\int_{\R^d} |f(x)| \left|\sZ_{Tt}(\sqrt{T} x)-\widetilde{\sZ}_T(\sqrt{T} x) \right| \DE_x[|u_0(B_t)|] \dd x  \right]\\
&\hspace{8mm}+\mathbb{E}\left|\int_{\R^d} \left(f(x) \widetilde{\sZ}_T(\sqrt{T} x) \DE_x[u_0(B_t)]  -  f(x) \mathbb{E}[\widetilde{\sZ}_T(\sqrt{T} x)] \DE_x[u_0(B_t)] \right)\dd x\right|\\
&\hspace{16mm}+ \int_{\R^d} |f(x)| \left|\mathbb{E}\sZ_{Tt}(\sqrt{T} x)-\mathbb{E}\widetilde{\sZ}_T(\sqrt{T} x) \right| \DE_x[u_0(B_t)] \dd x ,
}
where the last three terms converge to $0$. 
\end{proof}

\subsection{Case of stationary initial condition}\label{subsec:LLNstat}
\begin{proof}
We recall that for almost sure realization $\xi$ \begin{align*}
\sZ_\infty (\xi)=\lim_{T_0\to \infty}\sZ_{T_0}(\xi)=\lim_{T_0\to \infty}\DE_0\left[\exp\left(\beta\int_0^{T_0} \phi(y-B_s)\xi(\dd s,\dd y)-\frac{\beta^2V(0)T_0}{2}\right)\right].
\end{align*}

Then,  that for each $\ve>0$ and $x\in \R^d$\begin{align*}
u^{stat}_\ve(t,x)=\lim_{T_0\to \infty}\DE_0\left[\exp\left(\beta\int_0^{T_0} \phi(y-B_s)\xi^{(\e,t,x)}(\dd s,\dd y)-\frac{\beta^2V(0)T_0}{2}\right)\right].
\end{align*}
In particular,  the same argument as in \eqref{eq:eq:IntexprToPolymerMeasure2} {with $u_0\equiv 1$}  yields that 
 \begin{align*}
\int_{\R^d}(u^{stat}_\ve(t,x)-u_\ve(t,x))f(x)\dd x\eqlaw \int_{\R^d}f(x)(\sZ_\infty(\sqrt{T}x)-\sZ_{Tt}(\sqrt{T}x))\dd x
\end{align*}
and hence this converges to $0$ for any $f\in C_c(\R^d)$ by the shift-invariance of the white noise and the uniform integrability of $\{\sZ_T(0):T\geq 0\}$. Thus, Theorem \ref{th:LLNstat} follows from Theorem \ref{th:LLNreg}.
\end{proof}

\section{Proofs of the Edwards-Wilkinson limits for SHE} \label{sec:EWproof} 
\subsection{Heuristics} \label{sec:heuristicsEW}
We explain here shortly the strategy of proof of Theorems \ref{th:EWlimit}-\ref{th:PolymerCV}. For simplicity, we restrict ourselves to the case $u_0\equiv 1$ and consider only $1$-dimensional convergence in time. For multi-dimensional convergence, we refer the reader to Section \ref{sec:multi-dim}.

Equation \eqref{eq:uZ} tells us that for fixed $t$, $u_\e(t,\cdot) \eqlaw \sZ_{T t}(\sqrt{T} \cdot)$ (we let $T=\e^{-2}$), so in order to prove convergence in \eqref{eq:EWGICSHEResult}, it is enough to show that the martingale 
\begin{equation} \label{eq:definitionNTftauIntro}
(Z^{\ssup{T}}(f)):\tau\in[0,\infty)
\to 
T^{\frac{d-2}{4}}\int_{\mathbb{R}^d}f(x) \left(\sZ_{{T\tau}}(\sqrt{T} x)-1\right)\rmd x. 
\end{equation}
converges to an adequate Gaussian process\footnote{This gives joint convergence for finitely many test functions $f$ thanks to the Cram\'er-Wold theorem}.
To this end, we rely on the functional central limit theorem for martingales \cite[Theorem 3.11 in Chap.~8]{JS87}, which states that if the  quadratic vaariation $\langle Z^{\ssup{T}}(f) \rangle_\tau$ converges towards a deterministic function as $T\to\infty$, then the continous process $Z^{\ssup{T}}(f)$ converges in distribution towards a Gaussian process whose covariance structure is given by the limiting quadratic variation. 

We have,
\begin{equation} \label{eq:bracketNheurist}
\left\langle Z^{\ssup{T}}(f)\right \rangle_\tau = T^{\frac{d-2}{2}} \iint f(x) f(y) \langle \sZ(\sqrt T x), \sZ(\sqrt T y) \rangle_{Tt},
\end{equation}
and since
\[\dd \sZ_s(x)=\beta\,\DE_x \left[\int_{\mathbb R^d}\phi(B_s-b) \Phi_s \xi(\dd s,\dd b)\right] = \beta\int_{(\mathbb R^{d})^2} \phi(z-b) \DE_{0,x}^{s,z}[\Phi_s]  \xi(\dd s,\dd b)\rho_s(z)\dd z,\]
 we find, using that $V=\phi \star \phi$,
\begin{align}
& \langle \sZ(\sqrt T x), \sZ(\sqrt T y) \rangle_{T\tau}\\
 & = \b^2 \int_0^{T \tau} \iint V(z_1-z_2) \rho_s(z_1-\sqrt T x) \rho_s(z_2-\sqrt T y)\times  \nonumber  \DE_{0,\sqrt T x}^{s,z_1} \left[\Phi_s\right] \DE_{0,\sqrt T y}^{s,z_2}\left[\Phi_s \right] \dd s \dd z_1 \dd z_2\nonumber\\
& = \b^2 T^{-\frac{d-2}{2}} \int_0^{\tau} \int_{(\mathbb{R}^d)^2} V(Z) \rho_\sigma(z -x) \rho_\sigma(z + T^{-\frac{1}{2}} Z - y)\times  \nonumber \\
& \hspace{5cm} \times \DE_{0,\sqrt T x}^{T\sigma,\sqrt T z} \left[\Phi_{T\sigma}\right] \DE_{0,\sqrt T y}^{T\sigma,\sqrt T z + Z}\left[\Phi_{T\sigma} \right] \dd \sigma \dd z \dd Z,
 \label{eq:bracketExpansion_intro}
\end{align}
where we have set $s=T\sigma$, $z_1 = \sqrt T z$ and $z_2 = z_1 + Z$ while using the scaling property $\rho_{T\sigma}(\sqrt T z) = T^{-\frac{d}{2}} \rho_\sigma(z)$ and symmetry of $V$. (Note that we see here the correct scaling appearing).

Then, 
we appeal to the local limit theorem for polymers, which tells us that if $\b<\b_{L^2}$ and if $\overset{\leftarrow}{\sZ}_{T\sigma,\ell}(z)$ denotes a time-reversed polymer partition function starting at $(T\sigma,z)$ of time horizon $\ell=\ell(T)=o(T)$, then
 \[ \DE_{0,\sqrt T x}^{T\sigma,\sqrt T z} \left[\Phi_s\right] \approx_{L^2}  \sZ_\ell(\sqrt T x) \overset{\leftarrow}{\sZ}_{T\sigma,\ell}(\sqrt T z),
 \]
where by $\approx_{L^2}$ we mean that the difference between the RHS and LHS goes to $0$ in $L^2$-norm as $T\to\infty$ (see Theorem \ref{th:errorTermLLT} or \cite{V06,Si95}).
Then, coming back to \eqref{eq:bracketNheurist}, we can argue that
\begin{equation*} \label{eq:mainLLTapprox_intro}
 \langle Z^{\ssup T}(f)\rangle_\tau \approx_{L^1} \iint_{\mathbb{R}^d\times \mathbb{R}^d}f(x) f(y) \sZ_\ell(\sqrt T x) \sZ_\ell(\sqrt T y) \Psi_T(x,y) \dd x \dd y,
\end{equation*}
with
\begin{align*}
& \Psi_T(x,y) \\
&=   \b^2
\int_{0}^\tau \int V(Z) \rho_\sigma(z-x) \rho_\sigma( z+T^{-\frac{1}{2}}Z-y) \overset{\leftarrow}{\sZ}_{T\sigma,\ell}(\sqrt T z) \overset{\leftarrow}{\sZ}_{T\sigma,\ell}(\sqrt T z + Z) \dd  z \dd Z \dd \sigma.
\end{align*}

Since $\sZ(z)$ and $\sZ(z')$ become asymptotically inpdependent for $|z-z'|\to\infty$, we can justify via a homogenization argument that 
\begin{align}
\Psi_T(x,y) \approx_{L^1} \IE[\Psi_T(x,x)] \to \Psi(x,y)& := \gamma^2(\b) \int_{0}^{\tau}\rho_{2\sigma}(x-y),
\end{align}
as $T\to\infty$, where $\gamma^2(\b)$ is given in \eqref{def-sigma-beta}.
By an analogous homogenization argument, we further find that
\begin{equation*} \label{eq:CvBracketIntro}
\langle Z^{\ssup T}(f)\rangle_\tau \to \gamma^2(\b)\int_{0}^\tau \iint f(x) f(y) \rho_{2\sigma}(x-y)\dd \sigma \dd x \dd y,
\end{equation*}
in $L^1$, where the RHS is precisely equal to the variance of $\int f(x) \kU_1(t,x) \dd x$ when $\tau=t$ and $u_0\equiv 1$.

\subsection{Proof of Theorem \ref{th:EWlimit} (1-dimensional in time)} \label{ref:proofOfEWOnetimeGeneralIC}
For some fixed $t$, let us define the martingale
\begin{equation}
    \tau \to \kW_{\tau}(x)= \kW_{\tau,t,T}(x)=\DE_x\left[ \Phi_\tau(B) \, u_0\left(\frac{B_{t T}}{\sqrt{T}}\right)\right].
\end{equation}
as well as the martingale
   \begin{equation} \label{eq:definitionNTftauIntroGeneral}
(W^{\ssup{T}}(f)):\tau\in[0,t] \to 
T^{\frac{d-2}{4}}\int f(x) \left(\kW_{T \tau }\left({\sqrt{T}x}\right)-\bar{u}\left(t,x\right)\right)dx, 
\end{equation}
with respect to the filtration $\{\mathcal{F}_s^{(T)}:=\mathcal{F}_{Ts}:0\leq s\leq t\}$ such that ${\kW}^{\ssup{T}}_0=0$. We also define: 
\begin{equation} \label{eq:defNtau}
{\mathscr U_1(t,f)}:= \int f(x) \kU_1({t},x) \dd x,
\end{equation}
with $\kU_1$  given in \eqref{eq:EW_GRZ}.
  Relying on \eqref{eq:FKgeneralIC2}, we see that Theorem \ref{th:EWlimit} (for one time $t$) is proven once we show that as $T\to\infty$,

\begin{equation} \label{eq:cvNtftoU1f}
{W}_{t}^{\ssup{T}}(f)\cvlaw \mathscr U_1(t,f).
\end{equation}

For $0<  \tau_0\leq \tau \leq t$, we start again by decomposing
\begin{equation} \label{eq:decomposemartingale}
W^{\ssup{T}}_\tau(f) = W^{\ssup{T}}_{\tau_0,\tau}(f) + W^{\ssup{T}}_{\tau_0}(f),
\end{equation}
where $W^{\ssup{T}}_{\tau_0,\tau}(f)$ is defined by the relation. We first show that $W^{\ssup{T}}_{\tau_0}(f)$ is going to $0$ as $\tau_0\to 0$ uniformly in $T$:
\begin{lemma} \label{lem:unifNeglTau0}  As $\tau_0\to0$,
\begin{equation} \label{eq:unifNeglTau0General}
\sup_{T>1}\IE\left[W^{\ssup{T}}_{\tau_0}(f)^2\right] \to 0.
\end{equation}
\end{lemma}
By It\^o calculus for the white noise we have $\dd \Phi_s(B) = \int_{\mathbb R^d} \xi(\dd s,\dd b) \phi(B_s-b)\Phi_s(B)$, therefore equation \eqref{eq:definitionNTftauIntroGeneral} writes
\begin{align*}
W^{\ssup{T}}_{\tau_0,\tau}(f)&  =  T^{\frac{d-2}{4}} \int_{T\tau_0}^{T \tau} \int_{\mathbb R} \xi(\dd s, \dd b) \int_{(\mathbb R^d)^2}  f(x) \phi(z-b) \rho_s(z-\sqrt Tx)\\
&\hspace{4cm}\times  \DE_{0,\sqrt T x}^{s,z}[ \Phi_s] \DE_z\left[u_0\left(\frac{B_{tT-s}}{\sqrt{T}}\right)\right]
\dd z  \dd x.
\end{align*}
Having the local limit theorem in mind (see Theorem \ref{th:errorTermLLT} in the appendix), we define
\begin{equation} \label{eq:defMtau0tauf}
\begin{aligned}
M^{\ssup{T}}_{\tau_0,\tau}(f)  & := T^{\frac{d-2}{4}} \int_{T\tau_0}^{T \tau} \int_{\mathbb R} \xi(\dd s, \dd b) \int_{(\mathbb R^d)^2}  f(x) \phi(z-b) \rho_s(z-\sqrt Tx)  \\
&\hspace{3cm} \times  \sZ_{\ell(T)}(\sqrt T x) \overset{\leftarrow}{\sZ}_{s,\ell(T)}(z) \DE_z\left[u_0\left(\frac{B_{tT-s}}{\sqrt{T}}\right)\right]
 \dd z  \dd x,
\end{aligned}
\end{equation}
with $\ell=\ell(T)=T^a$ and $a<\frac{1}{2}$.
 Precisely, we will justify that:
 \begin{lemma} \label{lem:changeOfMartingaleSHE} For all test function $f$, for all $\tau\geq \tau_0$ and as $T\to\infty$,
\begin{equation} \label{eq:approxByMtau0General}
\IE\left[|W^{\ssup{T}}_{\tau_0,\tau}(f)-M^{\ssup{T}}_{\tau_0,\tau}(f)|^2\right]\to 0.
\end{equation}
\end{lemma}
Finally, we will show that:
\begin{proposition} \label{prop:CVbracketSHEalterGeneral} For any test function $f$, all $0< \tau_0\leq \tau \leq t$ and as $T\to\infty$,
\begin{equation} \label{eq:centralLimitThMGeneral}
\langle M^{\ssup{T}}_{\tau_0,\cdot}(f) \rangle_\tau \cvLone \gamma^2(\b) \int_{\tau_0}^\tau \int_{(\mathbb R^d)^3} f(x) f(y) \rho_\sigma(z-x) \rho_\sigma( z-y) \bar u (t- \sigma,z)^2 \dd \sigma \dd z \dd x \dd y.
\end{equation}
\end{proposition}
 Combining \eqref{eq:unifNeglTau0General}, \eqref{eq:approxByMtau0General} and \eqref{eq:centralLimitThMGeneral}, we obtain that for all $\tau\leq t$,
\begin{equation} \label{eq:CVbracketNtauf}
\langle {W}^{\ssup{T}}(f)\rangle_\tau \cvLone {\gamma^2(\b) \int_{0}^\tau \int_{(\mathbb R^d)^3} f(x) f(y) \rho_\sigma(z-x) \rho_\sigma( z-y) \bar u (t- \sigma,z)^2 \dd \sigma \dd z \dd x \dd y,}
\end{equation} 
which, by the functional central limit theorem for martingales \cite[Theorem 3.11 in Chap.~8]{JS87}, implies that $({W}^{\ssup{T}}_\tau(f))_{\tau\leq t}$ converges to the centered Gaussian process with variance given by the RHS of \eqref{eq:CVbracketNtauf}. Since $\mathscr U_1(t,f)$ is a centered Gaussian variable with the same variance as the RHS of \eqref{eq:CVbracketNtauf} when $\tau = t$, this in turn entails that \eqref{eq:cvNtftoU1f} holds.

We now turn to the proofs of the above-mentioned results and we begin with Lemma \ref{lem:changeOfMartingaleSHE} and Lemma \ref{lem:unifNeglTau0}.
\begin{proof}[Proof of Lemma \ref{lem:changeOfMartingaleSHE}]
Let $\tau_0 > 0$. Setting,
\begin{equation} \label{eq:DefGamma}
\Gamma(s,z,x) = \DE_{0,\sqrt T x}^{s,z}[ \Phi_s] - \sZ_{\ell(s)}(\sqrt T x) \overset{\leftarrow}{\sZ}_{s,\ell(s)}(z),
\end{equation}
we obtain using that $\int \phi(z_1-b)\phi(z_2-b) \dd b = V(z_1-z_2)$:
\begin{align*}
&\IE\left[\left|W^{\ssup{T}}_{\tau_0,\tau}(f)-M^{\ssup{T}}_{\tau_0,\tau}(f)\right|^2\right]\\
& = T^{\frac{d-2}{2}} \int_{T\tau_0}^{T \tau} \int_{(\mathbb R^d)^4} f(x) f(y) V(z_1-z_2) \rho_s(z_1-\sqrt Tx)  \rho_s(z_2-\sqrt Ty) \DE_{z_1}\left[u_0\left(\frac{B_{tT-s}}{\sqrt{T}}\right)\right]
 \\
&  \hspace{3cm} \times \DE_{z_2}\left[u_0\left(\frac{B_{tT-s}}{\sqrt{T}}\right)\right]
\IE[ \Gamma(s,z_1,\sqrt T x) \Gamma(s,z_2,\sqrt T y)] \dd s \dd z_1 \dd z_2 \dd x \dd y\\
& =  \int_{\tau_0}^{\tau}  \int_{(\mathbb R^d)^4} f(x) f(y) V(Z) \rho_\sigma(z-x)  \rho_\sigma(z-y+T^{-\frac{1}{2}}Z) \DE_{\sqrt{T}z}\left[u_0\left(\frac{B_{tT-\sigma T}}{\sqrt{T}}\right)\right]
 \\
& \times \DE_{\sqrt{T}z+Z}\left[u_0\left(\frac{B_{tT-\sigma T}}{\sqrt{T}}\right)\right]
 \IE\left[ \Gamma(T\sigma,\sqrt T z, \sqrt T x) \Gamma(T\sigma,\sqrt T z + Z,\sqrt T y)\right] \dd \sigma \dd z \dd Z \dd x \dd y\\
 & =  \int_{\tau_0}^{\tau}  \int_{(\mathbb R^d)^4} f(x) f(y) V(Z) \rho_\sigma(z-x)  \rho_\sigma(z-y+T^{-\frac{1}{2}}Z) \bar u (t- \sigma,z)
 \\
& \qquad  \times \bar u (t-\sigma,z+T^{-\frac{1}{2}} Z) 
 \IE\left[ \Gamma(T\sigma,\sqrt T z, \sqrt T x) \Gamma(T\sigma,\sqrt T z + Z,\sqrt T y)\right] \dd \sigma \dd z \dd Z \dd x \dd y,
\end{align*}
where we have set $s=T\sigma, z_1 = \sqrt T z, z_2 = \sqrt T z + Z$ and we have used the scaling property of Brownian motion and $\bar{u}(t,x)=\DE_x[u_0(B_t)]$.
For all fixed $z,\sigma,x$, the quantity $\IE\, \Gamma(T\sigma,\sqrt T z, \sqrt T x)^2$ vanishes as $T\to\infty$ by the local limit theorem for polymers (Theorem \ref{th:errorTermLLT}; see also the weaker statement as in \cite{V06} which suffices here), so we obtain by Cauchy-Schwarz inequality and dominated convergence that the last display goes to $0$. To obtain a domination, we use that $u_0$ (and hence $\bar u$) is bounded and we observe that  $\IE\,\Gamma(T\sigma,\sqrt T z, \sqrt T x)^2$ is uniformly bounded by a constant since for all $z$, 
\begin{equation} \label{eq:unifboundForP2P}
(i)\ \DE_{0,0}^{t,z}[\Phi_t] \eqlaw \DE_{0,0}^{t,0}[\Phi_t]; \text{ and } (ii) \text{ for } \b<\b_{L^2},\  \sup_t \IE \DE_{0,0}^{t,0}[\Phi_t]^2 <\infty,
\end{equation} 
where property (i) comes from shear-invariance property of white noise as in \cite[p.20]{CC18} and property (ii) follows from \cite[Corollary 3.8]{V06} (adapting slightly the proof to consider the simpler case of a single copy of the Brownian bridge, similarly to \cite[Proposition 3.2]{CN19} in the discrete case).
\end{proof}

\begin{proof}[Proof of Lemma \ref{lem:unifNeglTau0}]
By \eqref{eq:unifboundForP2P} and Chapman-Kolmogorov formula, we have
\begin{align}
\IE\left[W^{\ssup{T}}_{\tau_0}(f)^2\right] &= T^{\frac{d-2}{2}} \int_{0}^{T \tau_0}  \int_{(\mathbb R^d)^4} f(x) f(y) V(Z) \rho_s(z-\sqrt T x)  \rho_s(z+Z-\sqrt T y) \bar u (t- \sigma,z) \nonumber\\
& \qquad  \times \bar u (t-\sigma,z+T^{-\frac{1}{2}} Z)  \IE\left[ \DE_{0,\sqrt T x}^{s, z}[ \Phi_{\sigma}]\DE_{0,\sqrt T x}^{s, z+Z}[ \Phi_{\sigma}]\right] \dd s \dd z \dd Z \dd x \dd y\nonumber\\
& \leq C \Vert u_0 \Vert_\infty^2 \int_{0}^{\tau_0}  \int_{(\mathbb R^d)^3} |f(x) f(y)|  V(Z) \rho_\sigma(y-x + T^{-\frac{1}{2}}Z)  \dd \sigma \dd x \dd y \dd Z \nonumber\\
& \leq C \Vert f\Vert_\infty \Vert V\Vert_1 \Vert f\Vert_1 \tau_0, \label{eq:L2upperBoundZTf}
\end{align}
where in the {second} inequality we have made the same change of variables as above, while in the last inequality we have bounded $f(y)$ by its supremum and then integrated in $y$.
\end{proof}

\begin{proof}[Proof of Proposition \ref{prop:CVbracketSHEalterGeneral}]
Note that the second part of the proposition is a consequence of the first part thanks to the functional central limit for martingales (\cite[Theorem 3.11]{JS87}).
The proof of the first part of the proposition requires a few steps. By the same changes of variables to what was done in the proof of Lemma \ref{lem:changeOfMartingaleSHE}, we find that 
\begin{equation} \label{eq:bracketOfMtau0}
\langle M^{\ssup{T}}_{\tau_0,\cdot}(f) \rangle_\tau = \iint_{\mathbb{R}^d\times \mathbb{R}^d}f_1(x) f_2(y) \sZ_\ell(\sqrt T x) \sZ_\ell(\sqrt T y) \Psi_T(x,y) \dd x \dd y,
\end{equation}
where
\begin{align*}
\Psi_T(x,y) & = \b^2
\int_{\tau_0}^\tau \iint_{(\mathbb R^d)^2} V(Z) \rho_\sigma(z-x) \rho_\sigma( z+T^{-\frac{1}{2}}Z-y) \bar u (t- \sigma,z) \\
& \hspace{2cm} \bar u (t- \sigma,z+T^{-\frac{1}{2}} Z) \overset{\leftarrow}{\sZ}_{T\sigma,\ell}(\sqrt T z) \overset{\leftarrow}{\sZ}_{T\sigma,\ell}(\sqrt T z + Z)  \dd  z \dd Z \dd \sigma
\end{align*}


Also, we set \begin{align*}
\Psi(x,y)=\gamma^2(\b) \int_{\tau_0}^\tau \int_{\mathbb R^d} \rho_\sigma(z-x) \rho_\sigma( z-y) \bar u (t- \sigma,z)^2 \dd \sigma \dd z.
\end{align*} 
Then, we have \begin{align}
&\IE\left[\left|\langle M^{\ssup{T}}_{\tau_0,\cdot}(f) \rangle_\tau-\int_{(\mathbb R^d)^2} f(x) f(y)\Psi(x,y) \dd \sigma \dd z \dd x \dd y\right|\right]\notag\\
&\leq  \iint_{\mathbb{R}^d\times \mathbb{R}^d}|f(x) f(y)|\IE\left[ \sZ_\ell(\sqrt T x) \sZ_\ell(\sqrt T y)\right] \IE\left[\left|\Psi_T(x,y)-\Psi(x,y)\right|\right] \dd x \dd y\label{term:Mdiff1}\\
&+\left|\iint_{\mathbb{R}^d\times \mathbb{R}^d}f(x) f(y)\IE \left[\sZ_\ell(\sqrt T x) \sZ_\ell(\sqrt T y) -1\right]\Psi(x,y)\dd x \dd y\right|.\label{term:Mdiff2}
\end{align}
Thus, it is enough to show that \eqref{term:Mdiff1} and \eqref{term:Mdiff2} 	converge to $0$ as $T\to \infty$. Since we know that $\Psi(x,y)$ is bounded and for any $x,y\in \R^d$, $\ell\geq 0$\begin{align}
\IE\left[ \sZ_\ell(x) \sZ_\ell(y)\right]=\IE\left[\overset{\leftarrow}{\sZ}_{T\sigma,\ell}(x) \overset{\leftarrow}{\sZ}_{T\sigma,\ell}(y)\right]\leq \IE\left[\sZ_\infty^2\right]<\infty,\label{eq:expZZ}
\end{align}
when $\b<\b_{L^2}$, \eqref{term:Mdiff1} follows when the following two lemmas are proved.
\end{proof}

We first deal with this last term:
\begin{lemma} \label{lem:asymptPsi}  For all $0<\tau_0\leq \tau$,  $x,y\in \R^d$, $u_0\in C_b(\R^d)$,
  \aln{\label{eq:asymptPsi}
\lim_{T\to\infty} \IE\left|\Psi_T(x,y) -  \Psi(x,y)\right|=0.
}
\end{lemma}

\begin{lemma}\label{lem:asymptffZ} For all $f\in C_c^\infty(\R^d)$\begin{align*}
\lim_{T\to \infty}\left|\IE \left[\iint_{\mathbb{R}^d\times \mathbb{R}^d}f(x) f(y)\left(\sZ_\ell(\sqrt T x) \sZ_\ell(\sqrt T y) -1\right)\Psi(x,y)\dd x\dd y\right] \right|=0.
\end{align*}

\end{lemma}


Before turning to the proof of this estimate, we need the following lemma:
\begin{lemma}\label{lem:UIZZtxy}
For each $x,y\in \R^d$, the family $(\sZ_t(x)\sZ_t(y))_{t\geq 0}$ is uniformly integrable. 
\end{lemma}

\begin{proof}[Proof of Lemma \ref{lem:UIZZtxy}]
When we apply the Doob's inequality to an $L^2$-bounded martingale $\{\sZ_t(x):t\geq 0\}$ for $x\in\R^d$, we obtain that \begin{align}
\IE\left[\sZ_t(x)\sZ_t(y)\right]\leq \IE\left[\sup_{t\geq 0}\sZ_t(x)\sZ_t(y)\right]\leq \IE\left[\sup_{t\geq 0}\sZ_t^2\right]\leq C\IE\left[\sZ_\infty^{2}\right],\label{eq:UIZZin}
\end{align} 
where we have used the spatial shift invariance of $\sZ_t(x)$.
Thus, we obtain that for each $x,y\in \R^d$, $t\geq 0$, $M\geq 0$\begin{align}
\IE\left[\sZ_t(y)\sZ_t(y):\sZ_t(x)\sZ_t(y)\geq M\right]\leq \IE\left[\sup_{t\geq 0}\sZ_t(y)\sZ_t(y):\sup_{t\geq 0}\sZ_t(x)\sZ_t(y)\geq M\right],\label{eq:UIZZsup}
\end{align}
where the RHS converges to $0$ as $M\to\infty$ from \eqref{eq:UIZZin} . 
\end{proof}

\begin{proof}[Proof of Lemma \ref{lem:asymptPsi}]
The proof is divided into several steps of approximation. \\

\noindent(Step 1): 
\begin{align*}
&\Psi_T(x,y)\\
&\approx_{L^1} 	\b^2\int_{\tau_0}^\tau \iint V(Z) \rho_\sigma(z-x) \rho_\sigma( z-y)\bar u (t- \sigma,z)^2 \overset{\leftarrow}{\sZ}_{\ell}(\sqrt T z) \overset{\leftarrow}{\sZ}_{\ell}(\sqrt T z + Z) \dd \sigma \dd  z \dd Z,
\end{align*}
where the $\approx_{L^1}$ sign means that the difference between the left and right hand sides goes to $0$ in $L^1$-norm.

 We remark that for any $\sigma\in [\tau_0, \tau]$, there exists a constant $C=C_{\tau_0}$ such that  \begin{align*}
\rho_{\sigma }(x)\leq C,\quad x\in \R^d.
\end{align*}
Therefore, the dominated convergence theorem with the boundedness of $\bar{u}$, and \eqref{eq:expZZ} implies that  \begin{align*}
&   
\b^2\int_{\tau_0}^\tau \iint V(Z) \rho_\sigma(z-x) \rho_\sigma( z+T^{-\frac{1}{2}}Z-y)\bar u (t- \sigma,z) \\
& \hspace{2cm} \bar u (t- \sigma,z+T^{-\frac{1}{2}} Z) \overset{\leftarrow}{\sZ}_{\ell}(\sqrt T z) \overset{\leftarrow}{\sZ}_{\ell}(\sqrt T z + Z) \dd \sigma \dd  z \dd Z\\
&\approx_{L^1}\b^2\int_{\tau_0}^\tau \iint V(Z) \rho_\sigma(z-x) \rho_\sigma( z-y)\bar u (t- \sigma,z)^2 \overset{\leftarrow}{\sZ}_{\ell}(\sqrt T z) \overset{\leftarrow}{\sZ}_{\ell}(\sqrt T z + Z) \dd \sigma \dd  z \dd Z.
\end{align*}

Since  for each $\tau_0\leq \sigma\leq \tau$
$$ \overset{\leftarrow}{\sZ}_{T\sigma,\ell}(\sqrt T z) \overset{\leftarrow}{\sZ}_{T\sigma,\ell}(\sqrt T z + Z)\overset{(d)}{=} \sZ_{\ell}(\sqrt T z) \sZ_{\ell}(\sqrt T z + Z),$$
by dominated convergence theorem it suffices to show that  for each $x,y\in \R^d$, and each $\tau_0\leq \sigma\leq \tau$,
\begin{align*}
&   \b^2\iint V(Z) \rho_\sigma(z-x) \rho_\sigma( z-y)\bar u (t- \sigma,z)^2 \\
& \hspace{2cm}{\sZ}_{\ell}(\sqrt T z) {\sZ}_{\ell}(\sqrt T z + Z) \dd  z \dd Z\\
& \cvLone \gamma^2(\b) \int_{\mathbb R^d} \rho_\sigma(z-x) \rho_\sigma( z-y) \bar u (t- \sigma,z)^2  \dd z.
\end{align*}
To this end, we use the truncation method as in Section \ref{subsec:LLNuzero} (recall that  the $\overline{\sZ}$ notation from \eqref{eq:barZ}). \\

\noindent(Step 2)
 \begin{align}
  &   \b^2\int \rho_\sigma(z-x) \rho_\sigma( z-y)\bar u (t- \sigma,z)^2 \mathcal{V}_T(z)\dd z\notag\\
&\approx_{L^1}  \b^2\int \rho_\sigma(z-x) \rho_\sigma( z-y)\bar u (t- \sigma,z)^2\widetilde{\mathcal{V}}_T(z)\dd z, \label{eq:VttilVt}
  \end{align}
where we denote by \begin{align*}
\mathcal{V}_T(z)&:=\int V(Z){\sZ}_{\ell}(\sqrt T z) {\sZ}_{\ell}(\sqrt T z + Z) \dd Z.\\
\widetilde{\mathcal{V}}_T(z)&:=\int V(Z)\left(\big\{\overline{\sZ}_{\ell}(\sqrt T z) \overline{\sZ}_{\ell}(\sqrt T z + Z)\big\}\land T^{\frac{1}{2}-a}\right)\dd Z.
\end{align*}

This immediately follows when we show  that for each $z\in \R$
\ben{\label{trunc2}
\IE\left[			\mathcal{V}_T(z)-\widetilde{\mathcal{V}}_T(z)	\right]\to 0,
  }
 which follows $\IE\left[\sZ_{\ell}(\sqrt T z) \sZ_{\ell}(\sqrt T z + Z)-\left( \big\{\overline{\sZ}_{\ell}(\sqrt T z) \overline{\sZ}_{\ell}(\sqrt T z + Z)\big\}\land T^{\frac{1}{2}-a}\right)\right]\to 0$ {uniformly in $Z\in \textrm{Supp} V$} for each $z\in \R^d$. 
  
Note that
\al{
  &\quad \sZ_{\ell}(\sqrt T z) \sZ_{\ell}(\sqrt T z + Z)-\left( \big\{\overline{\sZ}_{\ell}(\sqrt T z) \overline{\sZ}_{\ell}(\sqrt T z + Z)\big\}\land T^{\frac{1}{2}-a}\right)\\
  &= \sZ_{\ell}(\sqrt T z) \sZ_{\ell}(\sqrt T z + Z)-\overline{\sZ}_{\ell}(\sqrt T z) \overline{\sZ}_{\ell}(\sqrt T z + Z)\\
  &\hspace{12mm}+\left(\overline{\sZ}_{\ell}(\sqrt T z) \overline{\sZ}_{\ell}(\sqrt T z + Z)-\left( \big\{\overline{\sZ}_{\ell}(\sqrt T z) \overline{\sZ}_{\ell}(\sqrt T z + Z)\big\}\land T^{\frac{1}{2}-a}\right)\right)\\
  &= (\sZ_{\ell}(\sqrt T z) - \overline{\sZ}_{\ell}(\sqrt T z))(\sZ_{\ell}(\sqrt T z + Z)- \overline{\sZ}_{\ell}(\sqrt T z + Z))\\
  &\hspace{12mm}{ + (\sZ_{\ell}(\sqrt T z) - \overline{\sZ}_{\ell}(\sqrt T z))\sZ_{\ell}(\sqrt T z + Z)}\\
  &\hspace{12mm}{ +\sZ_{\ell}(\sqrt T z) (\sZ_{\ell}(\sqrt T z + Z)- \overline{\sZ}_{\ell}(\sqrt T z + Z))}\\
  &\hspace{12mm}+ \left(\overline{\sZ}_{\ell}(\sqrt T z) \overline{\sZ}_{\ell}(\sqrt T z + Z)-\left( \big\{\overline{\sZ}_{\ell}(\sqrt T z) \overline{\sZ}_{\ell}(\sqrt T z + Z)\big\}\land T^{\frac{1}{2}-a}\right)\right)
}

For the first three terms, taking $p,q>1$ such that $\gamma(\sqrt{p}\beta)<\infty$ and $p^{-1}+q^{-1}=1$, by Holder inequality,
\begin{equation}\label{Difference of sZ and bar-sZ}
  \begin{split}
  &\quad \IE\left[ \int_{{\rm Supp} V}\,V(Z)  (\sZ_{\ell}(\sqrt T z) - \overline{\sZ}_{\ell}(\sqrt T z))(\sZ_{\ell}(\sqrt T z + Z)- \overline{\sZ}_{\ell}(\sqrt T z + Z))\dd Z \right]\\
  & \leq \IE\left[ \int_{{\rm Supp} V}V(Z)  \sZ_{\ell}(\sqrt T z)(\sZ_{\ell}(\sqrt T z + Z)- \overline{\sZ}_{\ell}(\sqrt T z + Z))\dd Z \right]\\
  & =\IE\left[ \int_{{\rm Supp} V}\,V(Z)  (\sZ_{\ell}(\sqrt T z) - \overline{\sZ}_{\ell}(\sqrt T z))\sZ_{\ell}(\sqrt T z + Z)\dd Z \right]\\
  &\leq  \int_{{\rm Supp} V}  V(Z)\DE_{0,Z}\left[e^{\beta^2\int_{0}^\ell V(B^{(1)}_{s}-B^{(2)}_{s})ds} : B^{(1)}[0,\ell]\not \in \mathtt{R}_{\ell,\ell}\right] \dd Z\\
    &\leq\left( \int  V(Z)\DE_{0,Z}\left[e^{p\beta^2\int_{0}^\ell V(B^{(1)}_{s}-B^{(2)}_{s})ds}\right]^{\frac{1}{p}} \dd Z\right) \left( \int V(Z)\DP_{0}\left(B^{(1)}[0,\ell]\not \in \mathtt{R}_{\ell,\ell}\right)^{\frac{1}{q}}  \dd Z\right),
    \end{split}
\end{equation}
which converges to $0$ as $T\to\infty$ (Recall the notation \eqref{eq:barZ})

For the last term, by $\overline{\sZ}_{\ell}(z)\leq  \sZ_{\ell}(z)$, the uniform integrability of $\sZ_{\ell}(0) \sZ_{\ell}(Z)$ and the dominated convergence theorem yields that
\al{
  &\quad   \IE\left[\int_{{\rm Supp} V}V(Z) \left(\overline{\sZ}_{\ell}(\sqrt T z) \overline{\sZ}_{\ell}(\sqrt T z + Z)-\left( \big\{\overline{\sZ}_{\ell}(\sqrt T z) \overline{\sZ}_{\ell}(\sqrt T z + Z)\big\}\land T^{\frac{1}{2}-a}\right)\right)\dd Z\right]\\
    &\leq \int_{{\rm Supp} V} V(Z)\IE \left[\sZ_{\ell}(\sqrt T z) \sZ_{\ell}(\sqrt T z + Z); \sZ_{\ell}(\sqrt T z) \sZ_{\ell}(\sqrt T z + Z)> T^{\frac{1}{2}-a}\right]\dd Z\to 0.
    }
Thus we obtain \eqref{trunc2}. \\

\noindent(Step 3) \begin{align}
&\b^2\int \rho_\sigma(z-x) \rho_\sigma( z-y)\bar u (t- \sigma,z)^2\tilde{\mathcal {V}}_T(z)\dd z\nonumber\\
&\approx_{L^1}\b^2\int \rho_\sigma(z-x) \rho_\sigma( z-y)\bar u (t- \sigma,z)^2\IE\left[\tilde{\mathcal {V}}_T(z)\right]\dd z\label{eq:VtEVt}
\end{align}
which completes the proof of Lemma \ref{lem:asymptPsi} with \eqref{trunc2} and \begin{align*} \IE\left[\mathcal{V}_T(z)\right]\to \int \beta^2 V(Z) \IE\left[\sZ_\infty(0)\sZ_\infty(Z)\right]\dd Z=\gamma^2(\beta).
\end{align*}

Since $\textrm{Cov} \left(\widetilde{\mathcal{V}}_T(z_1),\widetilde{\mathcal{V}}_T(z_2)\right)=0$ unless  $|z_1-z_2|< C_{\phi}T^{a-\frac{1}{2}}$  (for some constant $C_\phi$ depending of the size of the support of $\phi$),  for each $\tau_0\leq \sigma\leq \tau$, and $x,y\in \R^d$
\begin{align}
 &\IE\left( \int \rho_\sigma(z-x)  \rho_\sigma( z-y)\left(\widetilde{\mathcal{V}}_T(z)-\IE\left[\widetilde{\mathcal{V}}_T(z)\right]\right)  \bar u (t- \sigma,z)^2 \dd  z \right)^2\nonumber\\
&= \iint \rho_\sigma (z_1-x)\rho_\sigma(z_1-y)\rho_\sigma (z_2-x)\rho_\sigma(z_2-y)\\
&\hspace{3cm} \bar u (t- \sigma,z_1)^2\bar u (t- \sigma,z_2)^2 \textrm{Cov} \left(\widetilde{\mathcal{V}}_T(z_1),\widetilde{\mathcal{V}}_T(z_2)\right)\dd z_1\dd z_2\notag\\
&= \iint_{|z_1-z_2|\leq C_\phi T^{a-\frac{1}{2}}} \rho_\sigma (z_1-x)\rho_\sigma(z_1-y)\rho_\sigma (z_2-x)\rho_\sigma(z_2-y)\\
&\hspace{3cm} \bar u (t- \sigma,z_1)^2\bar u (t- \sigma,z_2)^2 \textrm{Cov} \left(\widetilde{\mathcal{V}}_T(z_1),\widetilde{\mathcal{V}}_T(z_2)\right)\dd z_1\dd z_2\notag\\   &\leq CT^{a-\frac{1}{2}}\rho_{2\sigma} (x-y)\IE \left[\widetilde{\mathcal{V}}_T(0)^2\right]
\end{align}
where in the last line we have used the spatial shift invariance of $\IE\left[\widetilde{\mathcal{V}}_T(z)^2\right]$ and the boundedness of $u(t-\sigma,x)$, $\rho_\sigma(x)$ for $\tau_0\leq \sigma\leq\tau$, $x\in\R^d$. Thus, it is enough to show that \begin{align*}
&T^{a-\frac{1}{2}}\IE \left[\widetilde{\mathcal{V}}_T(0)^2\right]\to 0.
\end{align*}
We have from Cauchy-Schwarz inequality that \begin{align*}
&T^{a-\frac{1}{2}}\IE \left[\widetilde{\mathcal{V}}_T(0)^2\right]\leq \int V(Z)dZ\int V(Z)T^{a-\frac{1}{2}}\IE\left[\left(\big\{\overline{\sZ}_{\ell}(0) \overline{\sZ}_{\ell}( Z)\big\}\land T^{\frac{1}{2}-a}\right)^2\right]dZ.
\end{align*}
Since  $T^{a-\frac{1}{2}}\IE\left[\left(\big\{\overline{\sZ}_{\ell}(0) \overline{\sZ}_{\ell}( Z)\big\}\land T^{\frac{1}{2}-a}\right)^2\right]\leq \IE \left[\sZ_\ell(0)\sZ_\ell(Z)\right]\leq \IE \left[\sZ_\infty(0)^2\right]$, it is enough to show that for each $Z\in \textrm{Supp} V$, \begin{align*}
T^{a-\frac{1}{2}}\IE\left[\left(\big\{\overline{\sZ}_{\ell}(0) \overline{\sZ}_{\ell}( Z)\big\}\land T^{\frac{1}{2}-a}\right)^2\right]\to 0\qquad\text{as $T\to \infty$. }
\end{align*} 
This follows from Lemma~\ref{CN Lemma 3.3}  and uniform integrability of $\{\sZ_{\ell_T}(0)\sZ_{\ell_T}(Z):T\geq 0\}$ .\end{proof}

We are now ready to complete the proof of  Proposition \ref{prop:CVbracketSHEalterGeneral} by proving Lemma \ref{lem:asymptffZ} but the proof  is the same as the proof of Lemma \ref{lem:asymptPsi}, so we will only show the strategy.
\begin{proof}[Proof of Lemma \ref{lem:asymptffZ}]

We have
\begin{align*}
& \iint_{\mathbb{R}^d\times \mathbb{R}^d}f(x) f(y) \sZ_\ell(\sqrt T x) \sZ_\ell(\sqrt T y) \Psi(x,y) \dd x \dd y\\
&\approx_{L^1}\iint_{\mathbb{R}^d\times \mathbb{R}^d}f(x) f(y) \left(\left\{\overline{\sZ}_\ell(\sqrt T x) \overline{\sZ}_\ell(\sqrt T y)\right\} \land T^{\frac{1}{2}-a}\right)\Psi(x,y) \dd x \dd y\\
&\approx_{L^1}\iint_{\mathbb{R}^d\times \mathbb{R}^d}f(x) f(y)\IE\left[  \left(\left\{\overline{\sZ}_\ell(\sqrt T x) \overline{\sZ}_\ell(\sqrt T y)\right\} \land T^{\frac{1}{2}-a}\right)\right] \Psi(x,y) \dd x \dd y\\
&\sim\iint_{\mathbb{R}^d\times \mathbb{R}^d}f(x) f(y)\IE\left[  \overline{\sZ}_\ell(\sqrt T x) \overline{\sZ}_\ell(\sqrt T y)\right] \Psi(x,y) \dd x \dd y\\
&= \iint_{|x-y|\leq  C_\phi T^{a-\frac{1}{2}}} f(x) f(y)\, \IE\left[  \overline{\sZ}_\ell(\sqrt T x) \overline{\sZ}_\ell(\sqrt T y)\right] \Psi(x,y) \dd x \dd y \nonumber\\
  &+ \iint_{{|x-y|>  C_\phi T^{a-\frac{1}{2}}}}f(x) f(y)\,\IE\left[ \overline{\sZ}_\ell(\sqrt T x)\right]\IE\left[ \overline{\sZ}_\ell(\sqrt T y)\right]  \Psi(x,y)\dd x \dd y\\
  &\to \iint f(x)f(y)\Psi(x,y)\dd x\dd y,
\end{align*}
where we denote by $a_n\sim b_n$ when $a_n-b_n\to 0$ as $n\to \infty$, and we have used the same argument in \eqref{trunc2} in the first approximation,  \eqref{eq:VtEVt} in the second approximation, and uniform integrability of $\{\overline{\sZ}_\ell(\sqrt T x) \overline{\sZ}_\ell(\sqrt T y)\}_{T\geq 0}$ follows from \eqref{eq:UIZZsup} in the third approximation, and the last convergence follows from  $\IE\left[ \overline{\sZ}_\ell(\sqrt T y)\right]\to \IE\left[\sZ_\infty\right]=1$ and $\IE\left[ \overline{\sZ}_\ell(\sqrt T x) \overline{\sZ}_\ell(\sqrt T y)\right] \leq \IE\left[\sZ_\infty^2\right]<\infty$. 
\end{proof}

\subsection{Proof of Theorem \ref{th:CV_stationarySHE} (1-dimensional in time)}  \label{subsec:conclusionGFF}
 When $u_0\equiv 1$ in \eqref{eq:definitionNTftauIntroGeneral}, we have for all $\tau\in [0,\infty)$:
\begin{equation} \label{eq:NtaufWhenu0isOne} 
W_\tau^{\ssup{T}}(f) = Z_\tau^{\ssup{T}}(f) = 
T^{\frac{d-2}{4}}\int f(x) \left(\sZ_{T \tau }\left({\sqrt{T}x}\right)-1\right)dx.
\end{equation}
We have seen in Section \ref{ref:proofOfEWOnetimeGeneralIC} (see also Section \ref{sec:heuristicsEW}) that for all $\tau\geq 0$, $Z_\tau^{\ssup{T}}(f)\cvlaw {Z_\tau(f)}$ where   the latter quantity is a centered Gaussian variable of variance 
\begin{equation} \label{eq:TheVarianceformula}
 \gamma^2(\b) \int_{0}^\tau \iint f(x)f(y) \rho(2\sigma,x-y) \dd \sigma \dd x \dd y.
\end{equation}
When $\tau = \infty$, this last quantity is also the variance of $\mathscr H(f) := \int f(x) \mathscr H(x) \dd x$, where $\mathscr H(x)$ is the Gaussian free field defined in \eqref{cov:H:GFF}. Note that it is equally the variance of $\mathscr U_2(t,f)=\int f(x) \mathscr U_2(t,x)\dd x$ from \eqref{EW2}.

Now, we define $Z_\infty^{(T)}(f)$ by formula \eqref{eq:NtaufWhenu0isOne} when $\tau =\infty$, and $Z_\infty(f)$ to be the centered Gaussian variable of variance given by $\eqref{eq:TheVarianceformula}$ when $\tau=\infty$.
Since $u^{stat}_\e(t,\cdot) \eqlaw \sZ_\infty(\e^{-1}\cdot)$ for all $t\geq 0$, we see that one-dimensional time convergences in Theorem \ref{th:CV_stationarySHE} is proven if we can justify that $Z_\infty^{\ssup{T}}(f) \cvlaw Z_\infty(f)$ as $T\to\infty$ (note that this also implies Corollary \ref{th:PolymerCV}).

To this end,  it is enough to show  that the following (commutative) diagram holds  \cite[Theorem 3.2]{B99}:
\[\begin{CD} 
Z_\tau^{\ssup{T}}(f) @>{(d)}>{T \to \infty}> Z_\tau(f)\\
@V{L^2, \text{ unif.\ in } T\geq 1}V{\tau\to \infty}V @V{(d)}V{\tau \to \infty}V\\
Z_\infty^{\ssup{T}}(f) @.  \mathscr H(f)
\end{CD}\]

Indeed, we already mentioned above that we know that the first line convergence holds. By convergence of the variance, the Gaussian variable $Z_\tau(f)$ converges as $\tau\to\infty$ to $\langle \mathscr H,f\rangle$.  So what is left to prove is the uniform convergence; we have
\begin{align*}
\left\Vert Z_\tau^{\ssup{T}}(f) - Z_\infty^{\ssup{T}}(f) \right\Vert_{L^2(\IP)} & \leq T^{\frac{d-2}{4}} \int \left\Vert \sZ_{\tau T}\big(\sqrt T x\big) - \sZ_{\infty}\big(\sqrt T x\big)  \right\Vert_{L^2(\IP)} |f(x)| \dd x\\
&\leq \Vert f \Vert_\infty T^{\frac{d-2}{4}} \Vert \sZ_{\tau T}(0) - \sZ_{\infty}(0)  \Vert_{L^2(\IP)},
\end{align*}
which vanishes uniformly in $T\geq 1$ as $\tau\to\infty$ since $\Vert \sZ_{\tau} - \sZ_{\infty}  \Vert_{L^2} = O(\tau^{-\frac{d-2}{4}})$, {which follows for example from \eqref{eq:bracketExpansion_Wbis} (recall that $\kW_\tau = \sZ_\tau$ when $u_0=0$), see also estimate (3.5) in \cite{CCM19b}.}

\subsection{Multidimensional convergence in the EW limits for SHE}\label{sec:multi-dim}
For ease of notation, we present the proof of multidimensional convergence in Theorem \ref{th:EWlimit} for the flat initial condition case ($u_0 \equiv 1$), although the case of general initial conditions is obtain by the same considerations. Multidimensional case of Theorems \ref{th:CV_stationaryKPZ} are obtained by commutating limits $T\to\infty$ and $\tau\to\infty$ limit as in Section \ref{subsec:conclusionGFF}.
\subsubsection{Preliminaries}
Let $t>0$ and set $u_0\equiv 1$. We directly see from \eqref{eq:uZ} that for all $0\leq s_1\leq \dots \leq s_n = t$, $x_1,\dots,x_n\in \mathbb{R}^d$,
\begin{align} 
& \big(u_\e(s_1,x_1),\dots, u_\e(s_n,x_n)\big)\nonumber\\
& = \Bigg( \sZ_{T(t-s_1), T t}\bigg(\sqrt{T}\, x_1,\xi^{\ssup{\e,t}}\bigg),\dots, \sZ_{T(t-s_n),T t}\bigg(\sqrt{T}  \, x_n, \xi^{\ssup{\e,t}}\bigg)\Bigg), \label{eq:couplingProporty}
\end{align}
where for any $0\leq U \leq T$ and $X\in \rd$, we set
\begin{align}
&\sZ_{U,T}(X;\xi)= 
E\big[\Phi_{U,T}(B) \big \vert {B_U=X}\big],\qquad\mbox{with}\nonumber\\
&\Phi_{U,T}(B)=\exp\bigg\{\beta\int_U^T\int_{\rd} \phi(B_s-y) \xi(s,y) \dd s \dd y- \frac{\beta^2}2 (T-U) V(0)\bigg\}.\label{PhiAB}
\end{align}
Therefore, since $\xi^{\ssup{\e,t}} \eqlaw \xi$, the proof reduces to showing that for all $t>0$, jointly for finitely many $u\in [0,t]$ and $f\in\mathcal C^\infty_c$, as $T\to\infty$, \footnote{{When dealing with both joint convergence in $t$ and $f$, we need to properly prove joint convergence for different $f$ since the Cramer-Wold device will not be enough in this case}}
\begin{equation}\label{eq:JointCVlogZ}
 T^{\frac{d-2}{4}}\int f(x) \left( \sZ_{T u,T t}(\sqrt T x)- 1\right)\dd x  \cvlaw  \int f(x) \mathscr U_1(t-u,x)\dd x,
\end{equation}

We further define, for $u\geq 0$ and $f\in\mathcal C^\infty_c$, the process
\begin{equation}\label{eq:rescaledMg}
(Z^{\ssup{T}}(u,f)):\tau\in[0,\infty)\to 
\begin{cases}
T^{\frac{d-2}{4}}\int_{\mathbb{R}^d}f(x) \left(\sZ_{Tu,{T\tau}}(\sqrt{T} x)-1\right)\rmd x & \text{if } \tau \geq u,\\
\qquad \qquad 0 & \text{if } 0\leq \tau < u.
\end{cases}
\end{equation}
Then, for all $u\geq 0$ and $f\in\mathcal C^\infty_c$, $\tau \to Z_\tau^{\ssup{T}}(u,f)$ is a continuous martingale and such that $Z_0^{\ssup{T}}(u,f) = 0$. In view of the desired convergence \eqref{eq:JointCVlogZ}, we have again in mind the functional CLT for martingales \cite[Theorem 3.11]{JS87}, so we are interested in the limit of the cross-bracket $\langle Z^{\ssup T}(u_1,f_{{ 1}}),Z^{\ssup T}(u_2,f_{{ 2}})\rangle_\tau$.
We have:
\begin{proposition}\label{lm:CVLone}
For all test functions $f_1$ and $f_2$ in $\mathcal{C}_c^\infty$ and $0\leq u_2\leq u_1\leq t$, for all $\tau \geq u_1$,
\begin{align} 
&\langle Z^{\ssup T}(u_1,f_1),Z^{\ssup T}(u_2,f_2)\rangle_\tau \nonumber\\
&\overset{L^1}{\longrightarrow} \int_{\mathbb{R}^d}\int_{\mathbb{R}^d} \left( \gamma^2(\b) \int_{0}^{\tau-u_1} \rho(2\sigma + (u_1-u_2),x-y) \dd \sigma \right) f_1(x)f_2(y)\dd x \dd y, \label{eq:LimcovStruc}
\end{align}
as $T\to\infty$, where
\[
\mathrm{Cov}\big(\mathscr U(t-u_1,x), \mathscr U(t-u_2,y)\big)= \gamma^2(\b)\int_0^{t-u_1} \rho(2\sigma +(u_1-u_2),x-y) \dd \sigma.
\]
\end{proposition}
\begin{proof}
First observe that by a simple shift in time, we can again restrict ourselves to showing the lemma when $u_2 = 0$. Then, for all $\tau \geq u$,
\begin{align*}
&\langle Z^{\ssup T}(u,f_1),Z^{\ssup T}(0,f_2)\rangle_\tau  = \b^2 \int_u^\tau 
\int_{\mathbb{R}^d\times \mathbb R^d} f_1(x)f_2(y) \dd x \dd y\\
& \times \int_{\mathbb R^d} V(Z) \rho_{\sigma-u}(z-x) \rho_\sigma( z-y+T^{-\frac{1}{2}}Z) \DE_{Tu,\sqrt T x}^{T\sigma,\sqrt T z} \left[\Phi_{Tu,T\sigma}\right] \DE_{0,\sqrt T y}^{T\sigma,\sqrt T z+Z} \left[\Phi_{T\sigma}\right] \dd  z \dd Z.
\end{align*}
By a repetition of the arguments that lead to \eqref{eq:CVbracketNtauf} (namely \eqref{eq:decomposemartingale}-\eqref{eq:centralLimitThMGeneral}, with in particular Lemma \ref{lem:asymptPsi} and the proof of Proposition \ref{prop:CVbracketSHEalterGeneral}),
 we find that
\begin{equation*} 
\langle Z^{\ssup T}(u,f_1),Z^{\ssup T}(0,f_2)\rangle_\tau \approx_{L^1} \iint_{\mathbb{R}^d\times \mathbb{R}^d}f_1(x) f_2(y) \IE\left[\sZ_{Tu,Tu+\ell}(\sqrt T x)\right] \IE\left[ \sZ_\ell(\sqrt T y)\right] \Theta_T(x,y) \dd x \dd y,
\end{equation*}
where
\begin{align*}
\Theta_T(x,y) & = \b^2
\int_{u}^\tau \iint V(Z) \rho_{\sigma-u}(z-x) \rho_\sigma( z-y) \IE\left[\overset{\leftarrow}{\sZ}_{T\sigma,\ell}(\sqrt T z) \overset{\leftarrow}{\sZ}_{T\sigma,\ell}(\sqrt T z + Z)\right] \dd  z \dd Z \dd \sigma,\\
& = \gamma^2(\b) \int_{u}^\tau \rho_{2\sigma -u}(x-y) \dd \sigma,
\end{align*}
which gives \eqref{eq:LimcovStruc} when $u_2=0$.

\end{proof}

\subsubsection{Concluding lines for the multidimensional case} \label{sec:conclusionOfProofOfSHE}
We show here the desired convergence \eqref{eq:JointCVlogZ}, which also writes
\begin{equation}\label{eq:finalGoal}
 T\to\infty, \quad Z_t^{\ssup T}(u,f)\cvlaw\int_{\mathbb{R}^d} f(x)\mathscr U_1(t-u,x)\dd x \quad \text{jointly in } u\in[0,t], f\in \mathcal C^\infty_c,
\end{equation}
We first observe that \eqref{eq:LimcovStruc} holds for all $0\leq u_2\leq u_1\leq t$, while if $0\leq \tau <u_1$, the cross-bracket $\langle Z^{\ssup T}(u_1,f_1),Z^{\ssup T}(u_2,f_2)\rangle_\tau $ is null and converges trivially to zero. Then, we define the quantity $C_\tau(u_1,u_2,f_1,f_2)$ to be equal to the RHS of \eqref{eq:LimcovStruc} when $\tau \geq u_1$ and set it to $0$ when $\tau < u_1$. We also define $C_\tau(u_1,u_2,f_1,f_2)$ when $u_1>u_2$ to be $C_\tau(u_2,u_1,f_1,f_2)$. 

Therefore, by the multidimensional functional central limit for martingales (\cite[Theorem 3.11]{JS87}), we obtain that the family of continuous martingales $(Z^{\ssup T}(u,f))_{u}$ converges towards the family of Gaussian continuous martingales $(\tau\to(G_\tau(u,f))_{u}$, whose covariance structure is given by $C_\tau(u_1,u_2,f_1,f_2)$, where this convergence holds in terms of continuous processes in $\tau$ and jointly for finitely many $u\in[0,t]$, $f\in \mathcal C^\infty_c$.

Hence, \eqref{eq:finalGoal} follows from letting $\tau=t$ in this convergence and by observing that there is equality in law between the centered Gaussian families $G_t(u,f))_{u\in[0,t],f}$ and $(\int f(x)\mathscr U_1(t-u,x)\dd x)_{u\in[0,t],f}$  since they have the same covariance structure (recall that we have set $u_0(\cdot)\equiv 1$).

\section{Proofs of the Edwards-Wilkinson limits for KPZ}\label{sec:proofEWKPZ}
We use a constant $C>0$ depending on $d$ and $u$, which may change from line to line.
Recall that 
  \al{
    \kW_{\tau}(x)=\kW_{\tau,t, T}(x)=\DE_x\left[ \Phi_\tau \, u_0\left(\frac{B_{t T}}{\sqrt{T}}\right)\right],
  }
and note that for fixed $t$,
\begin{equation} \label{eq:hepsWtT}
h_{\e}(t,x)\eqlaw \log{\kW_{t T}(\sqrt{T}\,x)}.
\end{equation}
  Since $u_0$ is both bounded and bounded away from $0$, both $|u_0^{-1}|_{\infty}$ and $|u_0|_{\infty}$ are finite and 
  $$|u_0^{-1}|^{-1}_{\infty} \sZ_{\tau}(x)\leq \kW_{\tau}(x)\leq |u_0|_{\infty} \sZ_{\tau}(x).$$
  Hereafter, we use this without any comment.
  
  For fixed $t>0$, we define a martingale increment $\dd\overline{M}_\tau(x)$ {(we omit the superscript $T$ and $t$ to make notation simple)} as
\begin{align}
&\dd \overline{M}_\tau(x)\nonumber \\
&=\frac{\beta}{\bar{u}(t,T^{-\frac{1}{2}}\,x)} \DE_{x}\left[ \int_{\R^d} \phi(B_\tau-b) \xi(\dd \tau, \dd b)\,\overset{\leftarrow}{\sZ}_{t,\ell(T)}(B_\tau) \, u_0\left(\frac{B_{t T}}{\sqrt{T}}\right)\right]\nonumber\\
&= \frac{\beta}{\bar{u}(t,T^{-\frac{1}{2}}\,x)} \int_{\R^d} \xi(\dd \tau, \dd b) \int_{\R^d}\rho_\tau(z-x) \phi(z-b) \overset{\leftarrow}{\sZ}_{t,\ell(T)}(z) \, \DE_z\left[ u_0\left(\frac{B_{t T-\tau}}{\sqrt{T}}\right)\right]\dd z,\label{eq:defdMbartau}
\end{align}
for   $\tau\in [\ell(T),t T]$  with $\ell( T)=T^a$, $a\in (0,\frac{1}{4})$ for $T> 0$  large enough.

\subsection{Structure of the proof for 1-dimensional time convergence in time}
\label{sec:explainKPZ}
It\^o's decomposition gives
\begin{equation}\label{eq:Itodecomposition}
\log \kW_\tau(x) = \overline{H}_\tau(x) - \frac{1}{2} \langle \overline{H}(x) \rangle_\tau
\end{equation}
with $t\to \overline{H}_\tau(x)$ being the martingale satisfying $\overline{H}_0(x) = 0$ and
\begin{equation} \label{eq:diffOfH}
\dd \overline{H}_\tau(x) = \frac{\dd \kW_\tau(x)}{\kW_\tau(x)}.\end{equation}
The main idea of our approach is to rely again on the central limit theorem for martingales and show that:
\begin{proposition} \label{prop:mainpropKPZ} For any test function $f$, as $T\to \infty$
\begin{equation} \label{eq:CLTmartH}
T^{\frac{d-2}{4}} \int_{\mathbb{R}^d} f(x) \overline{H}_{T t} (\sqrt T x) \dd x \cvlaw \mathscr U_3(t,f),
\end{equation}
where $\mathscr U_3(t,f) = \int f(x) \mathscr U_3(t,x) \dd x$ with $\mathscr U_3(t,x)$ from \eqref{eq:defU3}.
\end{proposition}
Then, we will prove that the It\^o correction can be neglected in the limit: 
\begin{proposition}\label{prop:VanishBracket}
 For all $t>0$, as $T\to\infty$,
\begin{equation}
T^{\frac{d-2}{4}} \int_{\mathbb{R}^d} f(x)\,\left( \big\langle \overline{H}(\sqrt T x) \big\rangle_{T t} - \IE \big\langle \overline{H}(\sqrt T x) \big\rangle_{Tt} \right)  \cvLone 0.
\end{equation}
\end{proposition}

Proposition \ref{prop:mainpropKPZ}  and Proposition \ref{prop:VanishBracket} combined {with equations \eqref{eq:hepsWtT} and  \eqref{eq:Itodecomposition}} imply Theorem~\ref{th:EWlimitKPZ}. The proofs of the results for the GFF limit and the EW limits are then concluded as in Section \ref{subsec:conclusionGFFKPZ} (or for the multi-time correlation see Section \ref{sec:KPZmultidim}).

The proof for Proposition \ref{prop:mainpropKPZ} follows the martingale strategy of Section \ref{sec:EWproof}. This time, because of the presence of $\kW_\tau$ in the denominator of \eqref{eq:diffOfH}, the  quadratic variation of $H$ is slightly more delicate to deal with compared to the previous sections and we wish to replace $H$ in \eqref{eq:CLTmartH} by another martingale with no denominator. To do so, by the local limit theorem and a homoginaization, one may expect
\al{
  \kW_\tau(x)&= \DE_x\left[ \Phi_\tau \, u_0\left(\frac{B_{t T}}{\sqrt{T}}\right)\right]\\
  &=\int_{\R^d}\rho_\tau(z-x) \DE_{0,x}^{\tau,z}[\Phi_\tau] \DE_z\left[ u_0\left(\frac{B_{t T-\tau}}{\sqrt{T}}\right)\right] \dd z\\  
  &\approx \int_{\R^d}\rho_\tau(z-x) \sZ_{\ell(t)}(x) \overset{\leftarrow}{\sZ}_{t,\ell(T)}(z) \DE_z\left[ u_0\left(\frac{B_{t T-\tau}}{\sqrt{T}}\right)\right] \dd z\\  
  &\approx  \sZ_{\ell(t)}(x)  \int_{\R^d}\rho_\tau(z-x)\IE\left[\overset{\leftarrow}{\sZ}_{t,\ell(T)}(z)\right] \DE_z\left[ u_0\left(\frac{B_{t T-\tau}}{\sqrt{T}}\right)\right] \dd z=  \sZ_{\ell(t)}(x) \bar{u}(t,T^{-\frac{1}{2}}\,x).
}
Therefore, one may observe that
\begin{align*}
\frac{\dd \kW_\tau(x)}{\kW_\tau(x)} & = \frac{\b}{\kW_\tau(x)}  \int_{\mathbb R^d} \xi(\dd \tau,\dd b) \DE_x\left[ \phi(B_\tau-b) \Phi_\tau   { \, u_0\left(\frac{B_{t T}}{\sqrt{T}}\right)} \right]\\
& = \frac{\b}{\kW_\tau(x)}  \int_{\mathbb R^d} \xi(\dd \tau,\dd b) \int_{\mathbb R^d} \rho_\tau(z-x) \phi(z-b) \DE_{0,x}^{ \tau,z}\left[\Phi_\tau\right] { \, \DE_z\left[ u_0\left(\frac{B_{t T-\tau}}{\sqrt{T}}\right)\right]} \dd z\\
&\approx \frac{\beta}{\bar{u}(t,T^{-\frac{1}{2}}\,x)} \int_{\R^d} \xi(\dd \tau, \dd b) \int_{\R^d}\rho_\tau(z-x) \phi(z-b) \overset{\leftarrow}{\sZ}_{\tau , \ell(T)}(z) \, \DE_z\left[ u_0\left(\frac{B_{t T-\tau}}{\sqrt{T}}\right)\right]\dd z\\
&= \dd \overline{M}_\tau(x),
\end{align*}
where we have used the local limit theorem in the third line. Hence $\dd \overline{H}_\tau(x)$ may be replaced by $\dd \overline{M}_\tau(x)$. In fact, we precisely show that
\begin{proposition} \label{prop:replaceByM}
{For $\e>0$ small enough, for all $t>0$ },  as $T\to \infty$, 
  \al{
    T^{\frac{d-2}{4}}\IE\left| \int_{\R^d} f(x) \left(\overline{H}_{T t}(\sqrt T x)-\overline{M}_{T t}(\sqrt{T}x)\right)\dd x \right| \to 0,
  }
   where
    \[
\overline{M}_{\tau}(x) := \int_{T^{1-\varepsilon}}^{\tau} \dd \overline{M}_s(x).
\]
\end{proposition}

The following proposition will conclude the proof of Proposition \ref{prop:mainpropKPZ}:

  
\begin{proposition} \label{prop:MwidetildeIsGaussianFlat}
 For $\e$ small enough, for any test function $f$, as $T\to\infty$,
\begin{equation} \label{eq:cvMtoU3}
 T^{\frac{d-2}{4}}\int_{\mathbb R^d} f(x) \, \overline{M}_{Tt}(\sqrt T x) \dd x \cvlaw \mathscr U_3(t,f).
\end{equation}
\end{proposition}
\begin{proof}[Proof of Proposition \ref{prop:MwidetildeIsGaussianFlat}] 

Let $0<\tau_0\leq \tau\leq t$. Thanks to  Lemma~\ref{lem:asymptPsi}, we know that as $T\to\infty$,
\begin{align} \label{eq:CVbracketMt}
&T^{\frac{d-2}{2}} \int_{T\tau_0}^{T\tau} \dd \left\langle \overline{M}(\sqrt T x), \overline{M}(\sqrt T y) \right\rangle_{s}\notag\\
& =T^{\frac{d-2}{2}}	\frac{\beta^2}{\bar{u}(t,x)\bar{u}(t,y)} \int_{T\tau_0}^{T\tau} \int_{(\R^d)^2}\rho_s(z_1-\sqrt{T}x)\rho_s(z_2-\sqrt{T}y)V(z_1-z_2)\notag\\
&{\hspace{7em} \times \overleftarrow{\sZ}_{s,\ell(T)}(z_1)\overleftarrow{\sZ}_{s,\ell(T)}(z_2)  \DE_{z_1}\left[u_0\left(\frac{B_{tT}}{\sqrt{T}}\right)\right]\DE_{z_2}\left[u_0\left(\frac{B_{tT}}{\sqrt{T}}\right)\right]\dd z_1\dd z_2\dd s}\notag\\
& \cvLone \frac{\gamma^2(\b)}{\bar{u}(t,x)\bar{u}(t,y)}\int_{\tau_0}^\tau\int_{\R^d}\rho_\sigma(x-z)\rho_\sigma(y-z)\bar{u}(t-\sigma,z)^2\dd z.
\end{align} 
Therefore, again by the functional martingale central limit theorem, we find that the centered martingale $\big(T^{\frac{d-2}{4}} \big(\overline{M}_{T \tau} (\sqrt T x)- \overline{M}_{T \tau_0} (\sqrt T x)\big)\big)_{\tau_0\leq \tau\leq t}$ converges in law (after integration with respect to a test function) to a Gaussian process with covariance given by the RHS of \eqref{eq:CVbracketMt}. Moreover, by proof of Lemma \ref{lem:unifNeglTau0}, we have as $\tau_0 \to 0$,
\begin{equation} \label{eq:MtildeTau0negl}
 \lim_{\tau_0\to 0}\varlimsup_{T\to\infty}\IE\left[ T^{\frac{d-2}{2}}\left(\int f(x) \overline{M}_{T\tau_0}(\sqrt T x)\dd x\right)^2\right]=0,
\end{equation}
Considering that when $\tau_0=0$ and $\tau=t$, the RHS of \eqref{eq:CVbracketMt} is exactly the covariance function of the Gaussian process $\bar u(t,x)^{-1} \mathscr U_1(t,x)$, we thus obtain that \eqref{eq:cvMtoU3} holds with $\mathscr U_3(t,x) = \bar u(t,x)^{-1} \mathscr U_1(t,x)$. Let us now justify why defined this way, $\mathscr U_3(t,x)$  is a solution of the SPDE \eqref{eq:defU3}.

%
%
%
%

     Since $\kU_1(t,f)=\int_{\R^d} \kU_1(t,x) f(x)\dd x $ satisfies in a distributional sense,
        $$\partial_t \kU_1(t,x)=\frac{1}{2}\Delta \kU_1(t,x) + \bar{u}(t,x) \xi(t,x),$$
        then $\kU_3(t,x)=\bar{u}(t,x)^{-1}\kU_1(t,x)$,
        \aln{
          &\Leftrightarrow \partial_t [\bar{u}(t,x) \kU_3(t,x)]=\frac{1}{2}\Delta [\bar{u}(t,x) \kU_3(t,x)] + \bar{u}(t,x) \xi(t,x)\notag\\
          &\Leftrightarrow \bar{u}(t,x) \partial_t  \kU_3(t,x)=\frac{1}{2} \bar{u}(t,x)\Delta \kU_3(t,x) + \nabla \bar{u}(t,x)\cdot\nabla \kU_3(t,x) +\bar{u}(t,x) \xi(t,x).\nonumber\\
            &\Leftrightarrow  \partial_t  \kU_3(t,x)= \frac{1}{2} \Delta \kU_3(t,x) + \nabla \log{ \bar{u}(t,x)} \cdot\nabla \kU_3(t,x) +\xi(t,x).\label{final}
        }
        \end{proof}
      
The next section is divided in two parts: one is dedicated to the proof of Proposition \ref{prop:replaceByM} (which in turn will conclude the proof of Proposition \ref{prop:mainpropKPZ}),
 and one to the proof of Proposition \ref{prop:VanishBracket}.

\subsection{Elements of proof}
{We denote by $\langle M\rangle_s':=\frac{\dd \langle M\rangle_s}{\dd s}$ and $\langle M,N\rangle_s':=\frac{\dd \langle M,N\rangle_s}{\dd s}$ the derivatives of quadratic variation and the cross-bracket of continuous  martingale $M$ and $N$ in this section.}

\subsubsection{Surgery of martingales (proof of Proposition \ref{prop:replaceByM})}
We start with the following lemma.
\begin{lemma} For all $\e\in(0,1)$, we have
\begin{equation} \label{eq:vanishingSmalltSurgery}
 \lim_{T\to \infty} T^{\frac{d-2}{2}}\IE\left[\left(\int^{T^{1-\ve}}_0\int f(x) \frac{\dd \kW_s(\sqrt{T}x)}{\kW_s(\sqrt{T}x)}\dd x \right)^2\right]=0.
 \end{equation}
\end{lemma}
\begin{proof}
  
We begin by making a few observations that we will use repeatedly in the following. For $\e>0$, define the event
  $$A_s(x)=\{ \sZ_s(\sqrt{T}x)^{-1} \leq T^{\frac{\ve}{8}} \}.$$
By Theorem 1.3 in \cite{CCM20}{, we know that $\sZ_\infty$ admits all negative moments. Hence by Doob's inequality applied to the submartingale $(\sZ^{-1}_s)_{s\geq 0}$, we find using} shift invariance that for all $\e>0$, for sufficiently large $T$,
\begin{equation} \label{eq:negligibleEvent}
\sup_{x\in \mathbb R^d} \sup_{s\geq 0} \IP\left(A_s(x)^c \right)\leq T^{-2d}.
\end{equation}
Moreover, observe that for all $x,y$, almost-surely
 \begin{equation}\label{uniform-estFLAT}
   \begin{split}
    &\quad   \left(\kW_s(\sqrt{T}x)\kW_s(\sqrt{T}y)\right)^{-1}
 {\left\lan \kW(\sqrt{T}x),\kW(\sqrt{T}y) \right\ran_s'}\\
     &= \frac{\DE_{\sqrt{T}x,\,\sqrt{T} y}^{\otimes 2} \left[\Phi_s^{(1)} \Phi_s^{(2)} V(B^{(1)}_s-B^{(2)}_s) u_0\left(T^{-\frac{1}{2}}\,B_{t T}^{(1)}\right)u_0\left(T^{-\frac{1}{2}}\,B_{t T}^{(2)}\right) \right]}{\DE_{\sqrt{T}x,\,\sqrt{T} y}^{\otimes 2} \left[\Phi_s^{(1)} \Phi_s^{(2)} u_0\left(T^{-\frac{1}{2}}\,B_{t T}^{(1)}\right)u_0\left(T^{-\frac{1}{2}}\,B_{t T}^{(2)}\right) \right]}\leq \Vert V \Vert_\infty,
     \end{split}
      \end{equation}
while expanding as in \eqref{eq:bracketExpansion_intro} and thanks to \eqref{eq:unifboundForP2P}, we have for all $x,y\in \mathbb R^d$ and $s\geq 0$,
\begin{align}
 \IE 
 {\langle \kW(\sqrt T x), \kW(\sqrt T y) \rangle_{s}'} & \leq  C T^{-d/2}  \int_{(\mathbb{R}^d)^2} V(Z) \rho_\frac{s}{T}(z -x) \rho_{\frac{s}{T}}(z + T^{-\frac{1}{2}} Z - y)  \dd z \dd Z\nonumber\\
 &{=C T^{-d/2}  \int_{\mathbb{R}^d} V(Z)  \rho_{\frac{2s}{T}}(x + T^{-\frac{1}{2}} Z - y)   \dd Z,}
 \label{eq:bracketExpansion_W}
\end{align}
for some finite $C=C(\b,\Vert u_0 \Vert_\infty)$. 

Now, from the last expression, we see that there is  $C=C(\b, \|f\|_\infty,\|f\|_1,\|u_0\|_\infty,\|V\|_1)$ which is finite and such that 
\begin{equation} \label{eq:boundOnBracketAt0}
\IE\left[\int_0^{T^{1-\ve}}\int f(x)f(y) \dd \lan \kW(\sqrt{T}x),\kW(\sqrt{T}y) \ran_s\dd x \dd y\right]\leq C T^{-\frac{d-2}{2}-\ve},
\end{equation}
(in fact this is \eqref{eq:L2upperBoundZTf} with $ \tau_0=T^{-\ve}$). Therefore, the expectation in \eqref{eq:vanishingSmalltSurgery} equals and satisfies
  \al{
 &\IE\left[\int_0^{T^{1-\ve}}\int f(x)f(y) \frac{\dd \lan \kW(\sqrt{T}x),\kW(\sqrt{T}y) \ran_s}{\kW_s(\sqrt{T}x)\kW_s(\sqrt{T}y)}\dd x \dd y\right]\\
     &\leq T^{\frac{\ve}{4}}\,  \int^{T^{1-\ve}}_0\int f(x)f(y) \IE\left[ 
   {\lan \kW(\sqrt{T}x),\kW(\sqrt{T}y) \ran_s'} \right] \dd x \dd y 
    \\
    &\hspace{8mm}+ \|V\|_\infty  \int^{T^{1-\ve}}_0\int f(x)f(y){ \IP\left(A_s(x)^c\cup A_s(y)^c\right)} \dd x \dd y \dd s\\
    &\leq C T^{\frac{\ve}{4}}\, T^{-\frac{d-2}{2}-\ve}+ C T^{-2d+1}=o(T^{-\frac{d-2}{2}}).
  }
\end{proof}
 Recall $\dd \overline{ M}_s$ from \eqref{eq:defdMbartau} and define
\begin{equation} \label{eq:defdL}
\dd L_s(z)=\bar{u}(t,z/\sqrt{T}) \sZ_{\ell(T)}\dd \overline{M}_s(z).
\end{equation}
\begin{lemma}\label{Diff W and L}
  If $\e$ is sufficiently small depending on $a$ and $\b<\b_{L^2}$, for all $t>0$,
  \al{
    \lim_{T\to\infty}T^{\frac{d-2}{4}}\IE\left|\int^{T t}_{T^{1-\ve}} \int f(x) \frac{\dd \kW_s(\sqrt{T}x)-\dd L_s(\sqrt{T}x)}{\kW_s(\sqrt{T}x)}\dd x \right| =0.
  }

\end{lemma}
\begin{proof} We have:
  \al{
    &\quad \IE\left|\int^{T t}_{T^{1-\ve}} \int f(x) \frac{\dd \kW_s(\sqrt{T}x)-\dd L_s(\sqrt{T}x)}{\kW_s(\sqrt{T}x)}\dd x \right|  \\
    &\leq  \int f(x) \IE\left|\int^{T t}_{T^{1-\ve}}\frac{\dd (\kW(\sqrt{T}\,x)- L(\sqrt{T}\,x))_s}{\kW_s(\sqrt{T}\,x)} \right|\dd x \\
    &\leq \int f(x) \IE\left[\left( \int^{Tt}_{T^{1-\ve}}\frac{\dd \lan \kW(\sqrt{T}\,x)- L(\sqrt{T}\,x)\ran_s }{\kW_s(\sqrt{T}\,x)^2} \right)^{\frac{1}{2}}\right] \dd x,}
  where we have used the B\"urkholder-Davis-Gundy inequality in the second inequality. Again by Theorem 1.3 in \cite{CCM20} and Doob's inequality, we know that the variable $\sup_{\sigma\geq 0} \sZ_\sigma(\sqrt{T}\,x)^{-1}$ admits any moment. Therefore, by Cauchy-Schwarz inequality and $\sZ_{\sigma}\leq |u_0^{-1}|_{\infty}\, \kW_{\sigma}$, the last expectation is further bounded from above by
  \al{
    &\quad   \IE\left[\left(\sup_{s\geq 0}\kW_s(\sqrt{T}\,x)^{-1}\right) \left( \int^{Tt}_{T^{1-\ve}}\dd \lan \kW(\sqrt{T}\,x)- L(\sqrt{T}\,x)\ran_s  \right)^{\frac{1}{2}}\right]\\
    &\leq  C \IE\left[ \int^{T t}_{T^{1-\ve}}\dd \lan \kW(\sqrt{T}\,x)- L(\sqrt{T}\,x)\ran_s  \right]^{\frac{1}{2}}.
}
Now, recalling the definition in \eqref{eq:DefGamma} of $\Gamma(s,z,x)$, the last expectation equals
\begin{align}
  & \int_{T^{1-\ve}}^{Tt} \int_{(\mathbb R^d)^2} V(z_1-z_2) \rho_s(\sqrt{T} x,z_1) \rho_s(\sqrt{T} y,z_2) \IE[\Gamma(s,z_1,\sqrt{T}\,x) \Gamma(s,z_2,\sqrt{T}\,x)]\nonumber\\
  & \qquad \times\DE_{z_1}\left[ u_0\left(\frac{B_{t T-s}}{\sqrt{T}}\right)\right]\DE_{z_2}\left[ u_0\left(\frac{B_{t T-s}}{\sqrt{T}}\right)\right]  \dd s \dd z_1 \dd z_2 \nonumber\\
& = T^{-\frac{d-2}{2}} \int_{T^{-\ve}}^t \int_{(\mathbb R^d)^2} \dd \sigma \dd Z \dd z\, V(Z) \rho_\sigma(z) \rho_\sigma(z + T^{-\frac{1}{2}}Z)\times \nonumber \\
& \IE[\Gamma(T\sigma,\sqrt T z,\sqrt{T}\,x) \Gamma(T\sigma,\sqrt T z + Z,\sqrt{T}\,x)] \DE_{\sqrt{T} z}\left[ u_0\left(\frac{B_{t T-\sigma}}{\sqrt{T}}\right)\right]\DE_{\sqrt{T} z+Z}\left[ u_0\left(\frac{B_{t T-\sigma}}{\sqrt{T}}\right)\right]  \nonumber \\
& \leq C T^{-\frac{d-2}{2}} \int_{T^{-\ve}}^t \int \dd \sigma \dd Z \dd z\, V(Z) \rho_\sigma(z)\rho_\sigma(z+T^{-\frac{1}{2}}Z) \nonumber\\
&\hspace{5cm}\times \left(\ell^{-\frac{d-2}{2}} + (1+|z|^2) e^{c\frac{\ell}{T}|z|^2} (\ell T^{-1})^{\frac{1}{2}}\right) \nonumber\\
& \leq C T^{-\alpha} T^{-\frac{d-2}{2}} \left(\int_{T^{-\ve}}^t \sigma^{-\frac{d}{2}}\dd \sigma\right) \label{eq:simComputation}
\end{align}
with $\alpha>0$ where we used in the first inequality the local limit theorem for polymers (Theorem \ref{th:errorTermLLT}) and in the last inequality we used that $sup_{x,y\in\R^d}\rho_{\sigma}(x,y)\leq (2\pi \sigma)^{-\frac{d}{2}}$ and we have observed that
\[
\int \rho_\sigma(z) (1+|z|^2) e^{c\frac{\ell}{T}|z|^2} \dd z = \int e^{-\frac{|y|^2}{2}} (1+\sigma |y|^2) e^{\sigma c\frac{\ell}{T}|y|^2} \dd y \leq C
\]
for $T$ large enough and where $C=C({t})<\infty$ since $\sigma\leq t$.

 Then for $\e$ small enough, the RHS of \eqref{eq:simComputation} is negligible with respect to $T^{-\frac{d-2}{2}}$, which by the previous displays implies the statement of the lemma.
\end{proof}
The next lemma ends the proof of Proposition \ref{prop:replaceByM}.

\begin{lemma}\label{LAST LEMMA}
  If $\e$ is sufficiently small depending on $a$, for $t>0$,
\al{  \lim_{T\to\infty}T^{\frac{d-2}{4}}\IE\left|\int^{ Tt}_{T^{1-\ve}} \int f(x) \left(\frac{\dd L_\sigma(\sqrt{T}x)}{\kW_\sigma(\sqrt{T}x)}-\dd \overline{M}_\sigma(\sqrt{T}\,x)\right) \dd x \right| =0.}
\end{lemma}
\begin{proof}
  We compute as
  \aln{
    &\quad \IE\left|\int^{Tt}_{T^{1-\ve}} \int f(x) \frac{\dd L_s(\sqrt{T}x)}{\kW_s(\sqrt{T}x)}-\dd \overline{M}_s(\sqrt{T}x) \dd x \right| \nonumber \\
    &\leq   C \int f(x) \IE\left|\int^{T t}_{T^{1-\ve}} \frac{\dd L_s(\sqrt{T}x)}{\kW_s(\sqrt{T}x)}-\dd \overline{M}_s(\sqrt{T}x)\right|\dd x\\
    &= C\int f(x)\IE\left|\int^{T t}_{T^{1-\ve}} \left(\frac{\sZ_{\ell(T)}(\sqrt{T}x)\, \bar{u}(t,x)}{\kW_s(\sqrt{T}x)}-1\right)\dd \overline{M}_s(\sqrt{T}x) \dd x \right| \dd x\nonumber\\
    &\leq C \int f(x) \IE\left[\left( \int^{T t}_{T^{1-\ve}} \left(\frac{\sZ_{\ell(T)}(\sqrt{T}x)\, \bar{u}(t,x)}{\kW_s(\sqrt{T}x)}-1\right)^2 \dd \lan \overline{M}(\sqrt{T}x)\ran_s \right)^{\frac{1}{2}}\right]\dd x\label{PPPP},
  }
  where we have used Burkh\"older-Davis-Gundy in the last line. For ease of notation, we simply write $\kW_s=\kW_s(\sqrt{T}x)$ and $\sZ_s=\sZ_{s}(\sqrt{T}x)$ in the proof. By Cauchy-Schwarz inequality, the expectation inside the integral of \eqref{PPPP}  is further bounded from above by
  \aln{
    &\quad \IE\left[\left(\sup_{\ell(T) \leq s\leq T t}\left|\frac{\sZ_{\ell(T)}\, \bar{u}(t,x)}{\kW_s}-1\right| \right) \left( \int^{T t}_{T^{1-\ve}}\dd \lan \overline{M}(\sqrt{T} x)\ran_s \right)^{\frac{1}{2}}\right]\nonumber\\
    &\leq  \IE\left[ \sup_{\ell(T) \leq s\leq T t}\left(\frac{\sZ_{\ell(T)}\, \bar{u}(t,x)}{\kW_s}-1\right)^2 \right]^{\frac{1}{2}} \IE\left[ \int^{T t}_{T^{1-\ve}}\dd \lan \overline{M}(\sqrt{T} x)\ran_s  \right]^{\frac{1}{2}}.\label{conclusion}
  }
  Since
  $$\frac{\sZ_{\ell(T)}\, \bar{u}(t,x)}{\kW_s}-1=\kW_s^{-1}(\sZ_{\ell(T)}\, \bar{u}(t,x)-\kW_{T t}+(\kW_{tT}-\kW_s)),$$
  the first term of \eqref{conclusion} can be bounded from above as
  \aln{
    &\quad \IE\left[ \sup_{\ell(T) \leq s\leq T t}\left(\frac{\sZ_{\ell(T)}\, \bar{u}(t,x) }{\kW_s}-1\right)^2 \right]\notag\\
    &\leq  \IE\left[ \sup_{\ell(T) \leq s\leq T t} \kW_s^{-2} \right]^{\frac{1}{2}} \IE\left[ \sup_{\ell(T) \leq s\leq T t}\left( \sZ_{\ell(T)}\, \bar{u}(t,x)-\kW_{T t}+(\kW_{tT}-\kW_s)  \right)^2 \right]^{\frac{1}{2}}\notag  \\
    &\leq C \left(\IE\left[\left(\sZ_{\ell(T)}(\sqrt{T}x)\, \bar{u}(t,x) -\kW_{T t}(\sqrt{T}x)\right)^2\right]+\IE\left[ \sup_{\ell(T) \leq s\leq T t}\left(\kW_{T t} -\kW_s\right)^2 \right]\right)^{\frac{1}{2}}, \label{star777}
  }
  where we have used $\IE\left[ \sup_{\ell(T) \leq s\leq T t} \kW_s^{-2} \right]\leq C \IE[\sZ_{\infty}^{-2}]<\infty$ by Doob's inequality and finite negative moments of $\sZ_{\infty}$. With the $\overline{\sZ}$ and $\Gamma(s,x,y)$ notations from \eqref{eq:barZ} and \eqref{eq:DefGamma}, the first term of \eqref{star777} can be bounded from above as
  \al{
 &\quad  \IE\left[\left(\sZ_{\ell(T)}(\sqrt{T}x)\, \bar{u}(t,x) -\kW_{T t}(\sqrt{T}x)\right)^2\right]\\
  &\leq  2\IE\left[\left(\sZ_{\ell(T)}(\sqrt{T}x)\, \bar{u}(t,x) -\sZ_{\ell(T)}(\sqrt{T}x) \DE_{\sqrt{T}x}\left[ \overset{\leftarrow}{\sZ}_{T t,\ell(T)}(B_{T t})u_0\left(\frac{B_{T t}}{\sqrt{T}}\right)\right]\right)^2\right]\\
  &\hspace{8mm} +  2\IE\left[\left(\kW_{T t}(\sqrt{T}x)-\sZ_{\ell(T)}(\sqrt{T}x) \DE_{\sqrt{T}x}\left[ \overset{\leftarrow}{\sZ}_{T t,\ell(T)}(B_{T t})u_0\left(\frac{B_{T t}}{\sqrt{T}}\right)\right]\right)^2\right]\\
    &= 2\IE[\sZ_{\ell(T)}^2]\IE\left[\left(\DE_{\sqrt{T}x}\left[ \left(\sZ_{\ell(T)}(B_{T t}) -1 \right) u_0\left(\frac{B_{T t}}{\sqrt{T}}\right) \right]\right)^2 \right]\\
    &\hspace{8mm} +  2\IE\left[ \DE_{\sqrt{T}x}\left[ \Gamma(T t,\sqrt{T}x,B_{T t}) u_0\left(\frac{B_{T t}}{\sqrt{T}}\right)\right]^2\right]\\
    &= 2\IE[\sZ_{\ell(T)}]^2 \IE\left[\left(\DE_{\sqrt{T}x}\left[ \left(\overline{\sZ}_{\ell(T)}(B_{T t}) -\IE\left[\overline{\sZ}_{\ell(T)}(B_{T t})\right] \right) u_0\left(\frac{B_{T t}}{\sqrt{T}}\right) \right]\right)^2 \right] +O(T^{-\alpha})\\
    &\leq C \DP^{\otimes 2}_{0}\left(|B_{T t}^{(1)}-B_{T t}^{(2)}|\leq \ell(t)\right)+ O(T^{-\alpha})=O(T^{-\alpha}),\\
  }
 with some $\alpha>0$ where we have used the local limit theorem for polymers \eqref{Application LLT KPZ} and \eqref{Difference of sZ and bar-sZ}.
  For the second term of \eqref{star777}, {we use that by \eqref{eq:bracketExpansion_W} there is a constant $C=C(\b)$ such that for all $x\in\R^d$, $\sigma > 0$,
\begin{equation} \label{eq:bracketExpansion_Wbis}
\IE\left[\langle \kW(\sqrt T x) \rangle_{T\sigma }' \right]  \leq C T^{-d/2} \sigma^{-d/2},
\end{equation} 
and hence}
  \al{
    \IE\left[ \sup_{\ell(T) \leq s\leq T t}\left(\kW_{T t}-\kW_s\right)^2 \right] &\leq C \IE \left[ \int_{\ell(T)}^{T t} \dd {\langle \kW\rangle_s} \right]\\
    &{\leq  C T^{-\frac{d-2}{2}} \int_{\frac{\ell (T)}{T}}^{t}\int_{\R^d}V(Z)\sigma^{-d/2}\dd Z\dd \sigma}\\   
&= O\left(T^{-\frac{(d-2)a}{2}}\right).
}
Thus, the first term of \eqref{conclusion} is bounded by $C\,T^{-\alpha'}$ with some $\alpha'>0$ that depends on $a$. If $\e$ is sufficiently small depending on $a$,  \eqref{conclusion} can be further bounded from above by
  \al{
    C T^{-\alpha'}  \IE\left[ \int^{T t}_{T^{1-\ve}} \dd \lan \overline{M}(\sqrt{T} x)\ran_s  \right]^{\frac{1}{2}}=o(T^{\frac{d-2}{4}}),
    }
which we obtain thanks to {\eqref{eq:bracketExpansion_Wbis} and the fact that}
 there is a constant $C=C(\b)$ such that for $s\in [\ell(T),Tt]$,	
\[
\IE\left[\langle \overline{M}(\sqrt T x) \rangle_{s}' \right]\leq C'\IE\left[\langle \kW(\sqrt T x) \rangle_{s}' \right].
\]
\end{proof}

\subsubsection{Reduction to a martingale (proof of Proposition \ref{prop:VanishBracket})}
Our goal is to prove that: 
\begin{equation} \label{eq:goalOfReduction}
T^{\frac{d-2}{2}}  \int^{tT}_0 \int_{\R^d}f(x) \left(\frac{\dd \lan \kW(\sqrt T  x) \ran_s}{\kW_s^2(\sqrt T x)}-\IE \left[\frac{\dd \lan \kW(\sqrt T x) \ran_s}{\kW_s^2 (\sqrt T x)}\right]\right) \dd x \cvLone 0.
\end{equation}
The proof is composed of three steps. We first deal with the big $s$ regime.

\begin{lemma}[Step 1]\label{lem:firstStep}
  For any $\e>0$,
 \aln{
 \lim_{T\to\infty} T^{\frac{d-2}{4}}  \IE\left[\int f(x) \int^{tT}_{T^{\frac{1}{2}+\ve}} {\frac{\dd \lan \kW (\sqrt{T}x)\ran_s}{\kW_s(\sqrt T  x)^2 }\dd x }\right]=0.
  }
  \end{lemma}
\begin{proof}
  Recall that
\[
A_s(x)=\{ \sZ_s(x)^{-1} \leq T^{\frac{\ve}{8}} \}.
\]
Then, by {\eqref{eq:bracketExpansion_Wbis}} and \eqref{eq:negligibleEvent}
    \al{
      \IE\left[\int^{tT}_{T^{\frac{1}{2}+\ve}} \frac{\dd \lan \kW(\sqrt{T} x) \ran_s}{\kW_s(\sqrt T  x)^2} \right]&\leq C T^{\frac{\ve}{4}}\int^{T t}_{T^{\frac{1}{2}+\ve}} s ^{-\frac{d}{2}}\,\dd s +\|V\|_{\infty} \int  ^{tT}_{T^{\frac{1}{2}+\ve}}  \IP({ A_s(\sqrt{T} x)}^c)\dd s \\
      &\leq C T^{\frac{\ve}{4}} T^{-\left(\frac{1}{2}+\e\right)\frac{d-2}{2}}+C T^{-2d+1}\\
      &\leq C T^{\frac{\ve}{4}} T^{-\frac{d-2}{4}} T^{-\frac{d-2}{2}\ve}=o(T^{-\frac{d-2}{4}}).
}
  \end{proof}

We now consider smaller $s$. The idea behind \eqref{eq:goalOfReduction} for small $s$ is that $\sZ_s(x)$ and $\sZ_s(x')$ decorrelate when $x$ and $x'$ are far away. To use this property, it is convenient to work with partition functions where the trajectories do not exit a ball of super diffusive scale. With the $\mathtt{R}_{\ell,T}(z)$ notation from \eqref{Def:CUBE}, we thus define the event $F_s(x, B)= \{B[0,s]  \subset  \mathtt{R}_{\sqrt{s}\log T,s}(x) \}$ and denote $\widetilde{\kW}_s(x) : = \DE_x\left[\Phi_s \,u_0(T^{-\frac{1}{2}}\,B_{T t}); F_s(x ,B)\right]$. Recalling that:
\[ \lan \kW (\sqrt{T}x) \ran_s^\prime =\DE_{\sqrt{T}x}\left[\Phi_s^{(1)} \Phi_s^{(2)} V(B^{(1)}_s-B^{(2)}_s) u_0\left(T^{-\frac{1}{2}}\,B_{t T}^{(1)}\right)u_0\left(T^{-\frac{1}{2}}\,B_{t T}^{(2)}\right) \right],\]
we also use the notation\footnote{with a slight abuse of notation since $s \to \widetilde{\kW}_s$ is not exactly a martingale}:
 \begin{align*}
\lan \widetilde{\kW}(x)\ran_s' ={\DE}_x^{\otimes 2}\left[\Phi_s^{(1)}\Phi_s^{(2)} V(\kW_s^{(1)}-\kW_s^{(2)})\,u_0\left(T^{-\frac{1}{2}}\,B_{t T}^{(1)}\right)u_0\left(T^{-\frac{1}{2}}\,B_{t T}^{(2)}\right) ; F_s^{(1)}(x),F_s^{(2)}(x)\right]
 \end{align*}
 Hence, define $\dd \lan \widetilde{\kW}(x)\ran_s= \lan \widetilde{\kW}(x)\ran_s^\prime \dd s$. Then, one can replace $W$ by $\widetilde{\kW}$ to prove \eqref{eq:goalOfReduction}:
\begin{lemma}\label{lem:ChangeZtoWidetildeZ} For all $\e\in(0,\frac{1}{2})$,
\[
\lim_{T\to\infty} T^{\frac{d-2}{2}}  \IE\left[\int^{T^{\frac{1}{2}+\ve}}_0 \int_{\R^d}f(x) \left|\frac{ \lan \kW(\sqrt T x) \ran'_s }{\kW_s(\sqrt T x)^{ 2}}- \frac{ \lan \widetilde \kW(\sqrt T x) \ran_s'}{\widetilde \kW_s (\sqrt T x)^{ 2}}\right| \dd x \dd s \right]=0.
\]
\end{lemma}
\begin{proof} The lemma is proven if we can show that the expectation of the integrand decays faster than for example $T^{-2d}$ uniformly in $t,x$.  For simplicity of notation, we write $\kW_s=\kW_s(\sqrt{T} x)$, $\widetilde{\kW}_s=\widetilde{\kW}_s(\sqrt{T} x)$ etc. On the event $\widetilde{\kW}_s\geq T^{-\frac{\ve}{2}}$, we have \begin{align*}
&{\bigg|}\frac{\lan W\ran_s'}{\kW_s^2}-\frac{\lan \widetilde W\ran_s'}{\widetilde \kW_s^2}{\bigg|}={\bigg|}\frac{\lan  \kW\ran_s' (\widetilde{\kW}_s^2- \kW_s^2)+(\lan  \kW\ran'_s -\lan \widetilde{\kW}\ran_s') \kW_s^2}{ \kW_s^2\widetilde \kW_s^2}{\bigg|}\\
&\leq |V|_\infty T^\ve  \left|\widetilde{\kW}_s^2- \kW_s^2\right|+T^\ve \left|\lan  \kW_s\ran' -\lan \widetilde{\kW}\ran_s'\right|\\
&\leq 2|V|_\infty T^\ve\left|\kW_s^2-\widetilde \kW_s^2\right|, 
\end{align*}
hence,
\begin{equation} \label{eq:upperBoundRemoveOutsideBall}
\IE\left[\left|\frac{\lan  \kW\ran_s'}{ \kW_s^2}-\frac{\lan \widetilde W\ran_s'}{\widetilde \kW_s^2}\right|\right]\leq 2\|V\|_\infty T^\ve \IE \left[\left|\widetilde{\kW}_s^2- \kW_s^2\right|\right]+\|V\|_\infty \IP(\widetilde{\kW}_s\leq T^{-\frac{\ve}{2}}).
\end{equation}
By Cauchy-Schwarz inequality, the first expectation on the RHS is bounded by
\[
\IE \left[\left|\kW_s-\widetilde \kW_s\right|^2\right]^{\frac{1}{2}}\IE \left[\left|\kW_s+\widetilde \kW_s\right|^2\right]^{\frac{1}{2}}\leq \sqrt{2}\IE \left[\kW_s^2\right]^{\frac{1}{2}}\IE \left[\left|\kW_s-\widetilde \kW_s\right|^2\right]^{\frac{1}{2}},
\]
while by H\"older's inequality for $p>1$ such that $p\beta^2<\beta_{L^2}^2$, there is a finite constant $C=C(p,\b)$ verifying:
\begin{align}
\IE \left|\kW_s-\widetilde \kW_s\right|^2 & \leq 2 \|u_0\|_{\infty}\, \DE_0^{\otimes 2}\left[	e^{\b^2\int_0^s V(B_s^{(1)}-B_{s}^{(2)})\dd s}; \,  F_s(0,B^{(1)})^c	\right] \nonumber\\
&\leq C\DP_0\left(  F_s(0,B^{(1)})^c	\right)^{1-1/p},\label{eq:boundExitProba}
\end{align}
where the last quantity goes to $0$ faster than any power of $T$ uniformly in  $s > 0$.

Furthermore, the second term of the RHS of \eqref{eq:upperBoundRemoveOutsideBall} satisfies
\begin{align}
\IP(\widetilde{\kW}_s\leq T^{-\frac{\ve}{2}}) &\leq \IP(\kW_s\leq 2T^{-\frac{\ve}{2}})+\IP(\overline \kW_s-\widetilde{\kW}_s\geq T^{-\frac{\ve}{2}})\nonumber\\
&\leq \IP(\sZ_s\leq 2|u_0^{-1}|_{\infty}T^{-\frac{\ve}{2}})+T^{\frac{\ve}{2}} \IE\left[|\kW_s-\widetilde \kW_s|\right],\nonumber\\
& = O(T^{-2d}), \label{eq:LowerboundonZwidetilde}
\end{align}
for all $k\geq 0$, where we have used Markov's inequality and where the first term of the second line controlled by \eqref{eq:negligibleEvent}, while the second term is dealt with by \eqref{eq:boundExitProba}.
\end{proof}
  
\begin{lemma}[Step 2] \label{eq:lem:removingTheBegining}
  For any $\e\in(0,\frac{1}{2})$,
\begin{equation*} \label{eq:beginingIsNegl}
T^{\frac{d-2}{2}}  \IE\left(\int^{T^{\frac{1}{2}-\ve}}_0 \int_{\R^d}f(x) \left(\frac{ \lan \widetilde{\kW}(\sqrt T x) \ran'_s}{\widetilde{\kW}_s(\sqrt T x)^{ 2}}-\IE \frac{\lan \widetilde{\kW}(\sqrt T x) \ran'_s}{\widetilde{\kW}_s (\sqrt T x)^{ 2}}\right) \dd x \dd s \right)^2\cvLone 0.
\end{equation*}
\end{lemma}
\begin{proof}
Denote,
\begin{equation}
\Delta_s(x):=\frac{\lan \widetilde{\kW} (x) \ran_s^\prime}{\widetilde{\kW}_s(x)^{ 2}}-\IE\left[ \frac{ \lan \widetilde{\kW} (x) \ran_s^\prime}{\widetilde{\kW}_s(x)^{ 2}}\right].
\end{equation}
We have
 \begin{align*}
    &  \IE\left[\left(\int^{T^{\frac{1}{2}-\ve}}_0 \int_{\R^d}f(x) \Delta_s(\sqrt{T}x) \dd x \dd s\right)^2\right]\\
    &\leq  T^{\frac{1}{2}-\ve} \int^{T^{\frac{1}{2}-\ve}}_0 \int_{\R^d\times \R^d} \left|f(x)f(y) \IE\left[\Delta_s(\sqrt{T}x)\Delta_s(\sqrt{T}y)\right] \right|  \dd x\dd y \dd s\\
    & =  T^{\frac{1}{2}-\ve} \int^{T^{\frac{1}{2}-\ve}}_0 \int \left|f(x) f(y) \IE\left[\Delta_s(\sqrt{T}x)\Delta_s(\sqrt{T}y) \right]\right| \mathbf{1}_{\{|x-y|\leq 2d\,T^{-\frac{1}{2}}s^{\frac{1}{2}} \log T\}} \dd x \dd y \dd s\\
    &\leq C T^{\frac{1}{2}-\ve} \int^{T^{\frac{1}{2}-\ve}}_0 T^{\frac{\ve}{2}}\, s^{-\frac{d}{2}}\, {T^{-\frac{d}{2}}}\, s^{\frac{d}{2}}(\log{T})^{d}\dd s+O(T^{-2d+1}) =o(T^{-\frac{d-2}{2}}),
\end{align*}
where we have used in the third line that $\Delta_s(\sqrt T x)$ is independent of $\Delta_s(\sqrt T y)$ whenever $|x-y| > 2d\,T^{-\frac{1}{2}}s^{\frac{1}{2}  } \log T$
and where in the last line we referred to \eqref{eq:negligibleEvent} and \eqref{eq:boundOnBracketAt0}  with $|\Delta_s(x)|\leq  \|V\|_{\infty}$ a.s. to find that
\begin{equation} \label{eq:boundOnDeltas}
 \left|\IE\left[\Delta_s(\sqrt{T}x)\Delta_s(\sqrt{T}y)\right] \right| \leq C T^{\frac{\ve}{2}} s^{-\frac{d}{2}} + CT^{-2d},
\end{equation}
for $T$ large enough and let $\e$ be small.
\end{proof}
  
   \begin{lemma}[Step 3]\label{lem:step3}
  For sufficiently small $\e>0$,
  \aln{
 \lim_{T\to\infty} T^{\frac{d-2}{2}}  \IE\left[{\left(\int^{T^{\frac{1}{2}+\ve}}_{T^{\frac{1}{2}-\ve}} \int_{\R^d}f(x) \left(\frac{ \lan \kW \ran_s'}{\kW_s^2}-\IE \frac{\lan \kW \ran_s'}{\kW_s^2}\right) \dd x \dd s\right)^2}\right]=0.
  }  
   \end{lemma}
   \begin{proof}
   Note that
 \al{
    &\quad \IE\left[{ \left(\int^{T^{\frac{1}{2}+\ve}}_{T^{\frac{1}{2}-\ve}} \int_{\R^d}f(x) \left(\frac{ \lan \kW \ran_s'}{\kW_s^2}-\IE \frac{\lan \kW \ran_s'}{\kW_s^2}\right) \dd x \dd s\right)^2}\right]\\
    &=2\int_{T^{\frac{1}{2}-\ve} \leq s_1\leq s_2\leq T^{\frac{1}{2}+\ve}}\int_{\R^d\times \R^d}f(x)f(y)\IE[\Delta_{s_1}(\sqrt{T}x)\Delta_{s_2}(\sqrt{T}y)]\dd x \dd y \dd s_1 \dd s_2.
  }
    To estimate $\IE[\Delta_{s_1}(\sqrt{T}x)\Delta_{s_2}(\sqrt{T}y)]$ above, we use the following lemma.
  \begin{lemma}\label{homgeniation in t} 
For all $\b<{ \b_{L^2}}$, there exists $\alpha>0$ such that for all sufficiently smalle $\e>0$ and $s_1,s_2\in[T^{\frac{1}{2} - \varepsilon}, T^{\frac{1}{2} + \varepsilon}]$,
    
\begin{equation} \label{eq:boundCovDelta}
     \IE[\Delta_{s_1}(\sqrt{T}x)\Delta_{s_2}(\sqrt{T}y)]\leq
     \begin{cases}
       0 & \text{ if $|x-y|>T^{-\frac{1}{4}+\ve}$}\\
       T^{-\frac{d}{4}-\alpha}& \text{ if $|s_2-s_1|>T^{\frac{1}{4}}$ and $|x-y|\leq T^{-\frac{1}{4}+\ve}$}\\
       T^{-\frac{d}{4}+\frac{1}{8}}& \text{ otherwise}.
    \end{cases}
\end{equation}
  \end{lemma}
  We first prove Lemma~\ref{lem:step3} using the lemma above. Indeed, for $\varepsilon$ small enough,
  \al{
    &\quad \int_{T^{\frac{1}{2}-\ve} \leq s_1\leq s_2\leq T^{\frac{1}{2}+\ve}}\int_{\R^d\times \R^d}f(x)f(y)\IE[\Delta_{s_2}(\sqrt{T}x)\Delta_{s_1}(\sqrt{T}y)]\dd x \dd y { \mathbf{1}_{\{|s_1-s_2|>T^{\frac{1}{4}}\}}}\dd s_1 \dd s_2\\
    &\leq T^{-\frac{d}{4}-\frac{1}{32}} \int_{T^{\frac{1}{2}-\ve} \leq s_1 \leq s_2 \leq T^{\frac{1}{2}+\ve}}\int_{\R^d\times \R^d}f(x)f(y){ \mathbf{1}_{\{|x-y|\leq T^{-\frac{1}{4}+\ve}\}}}\dd x \dd y \dd s_1 \dd s_2\\
    &\leq T^{-\left(\frac{1}{4}-\e\right)d}\,T^{-\frac{d}{4}-\alpha}\,T^{1+2\ve}=o(T^{-\frac{d-2}{2}}),
    }
 On the other hand,
    \al{
      &\quad \int_{T^{\frac{1}{2}-\ve} \leq s_1\leq s_2\leq T^{\frac{1}{2}+\ve}}\int_{\R^d\times \R^d}f(x)f(y)\IE[\Delta_{s_2}(\sqrt{T}x)\Delta_{s_2}(\sqrt{T}y)]\dd x \dd y \mathbf{1}_{{ \{|s_1-s_2|\leq T^{\frac{1}{4}}\}}}\dd s_1 \dd s_2\\
    &\leq T^{\frac{1}{2}+\ve} T^{\frac{1}{4}}T^{-\frac{d}{4}+\frac{1}{8}} \int_{\R^d\times \R^d}\,f(x)f(y)\mathbf{1}_{{ \{|x-y|\leq T^{-\frac{1}{4}+\ve}\}}}\dd x \dd y =o(T^{-\frac{d-2}{2}}).
    }
    Thus, we complete the proof of Lemma~\ref{lem:step3}.
    \end{proof}
    \begin{proof}[Proof of Lemma~\ref{homgeniation in t}]
      The first line in the RHS of \eqref{eq:boundCovDelta} is immediate from the definition of $\widetilde{\kW}$.  The third line comes from \eqref{eq:boundOnDeltas}.  We  now prove the second line.  Recall that $\ell(T)=T^a$ with a small constant $a>0$ and 
      note that 
\begin{align}
&\IE[\Delta_{s_2}(\sqrt{T}x)\Delta_{s_2}(\sqrt{T}y)]\nonumber\\
&=\IE\left[\frac{ \lan \widetilde{\kW}(\sqrt{T}x) \ran_{s_1}^\prime}{\widetilde{\kW}_{s_1}(\sqrt{T}x)^2}\cdot \frac{ \lan \widetilde{\kW}(\sqrt{T}y) \ran_{s_2}^\prime}{\widetilde{\kW}_{s_2}(\sqrt{T}y)^2}\right]-\IE\left[\frac{ \lan \widetilde{\kW}(\sqrt{T}x) \ran_{s_1}^\prime}{\widetilde{\kW}_{s_1}(\sqrt{T}x)^2}\right]\IE\left[\frac{ \lan \widetilde{\kW}(\sqrt{T}y) \ran_{s_2}^\prime}{\kW_{s_2}(\sqrt{T}y)^2}\right],\label{eq:exprCovDelta}
\end{align}
where the second term in the RHS of \eqref{eq:exprCovDelta} is $O(T^{-\frac{d}{2}+4\e d})\leq CT^{-\frac{d}{4}-\alpha}$ with some $\alpha>0$ for sufficiently small $\e$. Next, we focus on the first term. Since $\lan \widetilde{\kW} \ran_{s_1}^\prime \widetilde{\kW}_{s_1}^{-2} \leq |V|_\infty$, we can safely exchange back $ \lan \widetilde{\kW} \ran_{s_i}^\prime \widetilde{\kW}_{s_i}^{-2}$ with $ \lan {W} \ran_{s_i}^\prime {W}_{s_i}^{-2}$ inside this expectation for $i=1,2$ (this has a cost decaying faster than any power of $T$, see the proof of Lemma \ref{lem:ChangeZtoWidetildeZ}).  Now, recall $\dd L_s(z)$ from \eqref{eq:defdL} and $\lan {L(z)} \ran_{s}^\prime:=\frac{\dd}{\dd s}\lan {L(z)} \ran_{s}$.

  Then, by the error in LLT (Theorem \ref{th:errorTermLLT}) and the proof of Lemma~\ref{Diff W and L}, for $\e$ small, there is some $\alpha>0$ such that for all $z\in \R^d$ and $s_2\in[T^{\frac{1}{2}-\ve},T^{\frac{1}{2}+\ve}]$ with $T$ large enough,

\begin{align*}
\IE|\lan \kW(z) \ran_{s_2}^\prime-\lan {L(z)} \ran_{s_2}^\prime|\leq C T^{-\frac{d}{4}-\alpha},
\end{align*}
for some $\alpha > 0$.
      Thus, for $s_1<s_2$ with $|s_2-s_1| \geq T^{\frac{1}{4}}$ and $\e$ small enough,
     \al{
     & \quad \IE\left[\frac{ \lan \kW(\sqrt{T}x) \ran_{s_1}^\prime}{\kW_{s_1}(\sqrt{T}x)^2}\cdot \frac{ \lan \kW(\sqrt{T}y) \ran_{s_2}^\prime}{\kW_{s_2}(\sqrt{T}y)^2}\right]\\
       & \leq \IE\left[\frac{ \lan \kW(\sqrt{T}x) \ran_{s_1}^\prime}{\kW_{s_1}(\sqrt{T}x)^2}\cdot \frac{ \lan \kW(\sqrt{T}y) \ran_{s_2}^\prime}{\kW_{s_2}(\sqrt{T}y)^2};~\kW_{s_2}(\sqrt{T}y)\geq T^{-\ve},~\frac{\kW_{\ell(T)}(\sqrt{T}y)}{\kW_{s_2}(\sqrt{T}y)}\leq T^{\ve}\right]+T^{-2d}\\
      &\leq  \IE\left[\frac{ \lan \kW(\sqrt{T}x) \ran_{s_1}^\prime}{\kW_{s_1}(\sqrt{T}x)^2}\cdot \frac{ |\lan \kW(\sqrt{T}y) \ran_{s_2}^\prime- \lan L((\sqrt{T}y)) \ran_{s_2}^\prime|}{\kW_{s_2}(\sqrt{T}y)^2};~\kW_{s_2}(\sqrt{T}y)\geq T^{-\ve}\right]\\
      &\qquad+ \IE\left[\frac{ \lan \kW(\sqrt{T}x) \ran_{s_1}^\prime}{\kW_{s_1}({ \sqrt{T}x)^2}}\cdot \frac{ \lan L((\sqrt{T}y)) \ran_{s_2}^\prime}{\kW_{s_2}(\sqrt{T}y)^2};~\frac{\kW_{\ell(T)}(\sqrt{T}y)} {\kW_{s_2}({ \sqrt{T}y})}\leq T^{\ve}\right]+T^{-2d}\\
        &\leq \|V\|_{\infty} T^{2\ve} \IE\left[| \lan \kW \ran_{s_2}^\prime- \lan L \ran_{s_2}^\prime|\right]\\
        &\qquad+C \IE\left[\frac{ \lan \kW(\sqrt{T}x) \ran_{s_1}^\prime}{\kW_{s_1}(\sqrt{T}x)^2}\cdot \left(\frac{\kW_{\ell(T)}(\sqrt{T}y)} {\kW_{s_2}(\sqrt{T}y)}\right)^2\lan \overline{M}(\sqrt{T}y) \ran_{s_2}^\prime;~\frac{\kW_{\ell(T)}(\sqrt{T}y)}{\kW_{s_2}(\sqrt{T}y)}\leq T^{\ve}\right]+T^{-2d}\\
        &\leq C T^{2\ve}\, T^{-\frac{d}{4}-\alpha}+ C T^{2\ve}\, \IE\left[\frac{ \lan \kW(\sqrt{T}x) \ran_{s_1}^\prime}{\kW_{s_1}(\sqrt{T}x)^2}\cdot \lan \overline{M}(\sqrt{T}y) \ran_{s_2}^\prime \right]+T^{-2d}\\
        & \leq  C T^{2\ve}\, T^{-\frac{d}{4}-\alpha}+ C T^{4\ve}\, \IE\left[ \lan \kW(\sqrt{T}x) \ran_{s_1}^\prime\right] \IE\left[ \lan \overline{M}(\sqrt{T}y) \ran_{s_2}^\prime \right]+2\,T^{-2d}\leq CT^{-\frac{d}{4}-\alpha'},
      }
 with $\alpha'>0$ and
where in the first inequality we used Markov inequality and Lemma~\ref{lower-tail} below to get
$$\IP\left(\kW_{s_1}< T^{-\ve}\right)+\IP\left(\frac{\kW_{\ell(T)}}{\kW_{s_2}}> T^{\ve}\right)\leq O(T^{-2d}),$$
and in the last line, that $\sZ_{s_1}^{-2} \lan \sZ \ran_{s_1}^\prime$ and $\lan \overline{M} \ran_{s_2}^\prime $ are independent for $s_2-s_1>T^{\frac{1}{4}}$, with expectations of the later quantities being respectively bounded above by $CT^{\frac{\ve}{2}} s_1^{-\frac{d}{2}}+T^{-2d}$ (via \eqref{eq:boundOnBracketAt0} and \eqref{eq:LowerboundonZwidetilde})  and $C s_2^{-\frac{d}{2}}$.      
      \end{proof}

\begin{lemma}\label{lower-tail}
  For any $k\in\N$, $s>0$,
  \al{
    \IE\left[ \sup_{t\geq s}\left(\frac{\sZ_{s}}{\sZ_t}\right)^k\right]\leq \IE\left[\sup_{t\geq 0}\left(\sZ_{t}\right)^{-k}\right]<\infty.
    }
\end{lemma}
\begin{proof}
  Recall:
  \al{
 \Phi_{s,t}=\exp{\left(\beta\int_s^t \int_{{\mathbb R}^d} \phi(z- B_r) \ \xi({\rm d} r,{\rm d} z)-\frac{\b^2 (t-s) V(0)}{2}\right)}.
  }
  We have defined the polymer measure as
  $$\dd \mathtt{P}_{x}^{\b,s}=\frac{1}{\sZ_{s}}\rho_{s}(0,x) \DE_{s,x}[\Phi_{s,t}] \dd x.$$
  Since $\dd \mathtt{P}_{x}^{\b,s}$ is a probability measure, by Jensen's inequality,
  \al{
    \left(\frac{\sZ_{s}}{\sZ_t}\right)^k
    &= \left(\int_{\R^d}\dd \mathtt{P}_{x}^{\b,s} \DE_{s,x}[\Phi_{s,t}]\right)^{-k}\\
    &\leq \int_{\R^d}\dd \mathtt{P}_{x}^{\b,s} \left(\DE_{s,x}[\Phi_{s,t}]\right)^{-k}.
  }
  Thus,
  \al{
    \IE\left[ \left(\sup_{t\geq s}\frac{\sZ_{s}}{\sZ_t}\right)^k\right]&\leq \IE\left[\int_{\R^d}\dd \mathtt{P}_{x}^{\b,s}\sup_{t\geq s} \left(\DE_{s,x}[\Phi_{s,t}]\right)^{-k}\right]\\
    &=\IE\left[\int_{\R^d}\dd \mathtt{P}_{x}^{\b,s}\right] \IE\left[\sup_{t\geq s}\left(\sZ_{t-s}\right)^{-k}\right]\\
    & = \IE\left[\sup_{t\geq 0}\left(\sZ_{t}\right)^{-k}\right].
    }
By Doob's submartingale inequality, we further see that $\IE\sup_{t\geq 0}\left(\sZ_{t}\right)^{-k}<\infty$ since $t\to \left(\sZ_{t}\right)^{-k}$ is a submartingale and by finite negative moments of the partition function.
\end{proof}

\subsection{Proof of Theorem \ref{th:CV_stationaryKPZ} in the 1-dimensional case}
\label{subsec:conclusionGFFKPZ}
Note that the proof in the 1-dimensional case already entails Corollary \ref{th:LogPolymerCV}.

Define, for $\tau\in[0,\infty]$,
\begin{equation}  
Y_\tau^{\ssup{T}} :=  
T^{\frac{d-2}{4}}\int f(x) \left(\log \sZ_{T \tau }\left({\sqrt{T}x}\right)-\IE \log{\sZ_{T \tau }\left({\sqrt{T}x}\right)}\right)dx,
\end{equation}
and observe that when $h_0\equiv 0$, then $W_\tau = \sZ_\tau$ and $\mathscr U_3(\tau,f) \eqlaw Z_\tau(f)$, where the latter quantity is a Gaussian variable with variance given by \eqref{eq:TheVarianceformula}. Hence, we know from Section \ref{sec:proofEWKPZ} that for all $\tau\geq 0$, $Y_\tau^{\ssup{T}}(f)\cvlaw Z_\tau(f)$.

Recall that Gaussian free field $\mathscr H(f)$ has variance given by formula $\eqref{eq:TheVarianceformula}$ when $\tau=\infty$ and that $h^{stat}_\e(t,\cdot) \eqlaw \log \sZ_\infty(\e^{-1}\cdot)$ for all $t\geq 0$. Therefore, one-dimensional time convergences in Theorem \ref{th:CV_stationaryKPZ} and Theorem \ref{th:LogPolymerCV} follow once we prove that $Y_\infty^{\ssup{T}}(f) \cvlaw \mathscr H(f)$ as $T\to\infty$. Again, it is enough to prove that the following (commutative) diagram holds:
\[\begin{CD} 
Y_\tau^{\ssup{T}}(f) @>{(d)}>{T \to \infty}> Z_\tau(f)\\
@V{\IP, \text{ unif.\ in } T\geq 1}V{\tau\to \infty}V @V{(d)}V{\tau \to \infty}V\\
Y_\infty^{\ssup{T}}(f) @. \mathscr H(f)\rangle
\end{CD}\]

The convergence is the first line is settled. By convergence of the variance, the Gaussian variable ${Z}_\tau(f)$ converges as $\tau\to\infty$ to $\langle \mathscr H,f\rangle$.  To prove the uniform convergence, we write that
\begin{align*}
\IE \left| Y_\tau^{\ssup{T}}(f) - Y_\infty^{\ssup{T}}(f) \right| & \leq T^{\frac{d-2}{4}} \int 2 \IE\left| \log \sZ_{\tau T}\big(\sqrt T x\big) - \log \sZ_{\infty}\big(\sqrt T x\big)  \right| |f(x)| \dd x\\
&\leq 2 \Vert f \Vert_\infty T^{\frac{d-2}{4}} \IE\left| \log \sZ_{\tau T}(0) - \log \sZ_{\infty}(0)  \right|.
\end{align*}
We will show the following:
\begin{equation} \label{eq:conveLogZtauT}
\sup_{T\geq 1} T^{\frac{d-2}{4}} \IE\left|\log \sZ_\infty - \log \sZ_{T\tau}\right| \to 0 \text{ as } \tau\to\infty.
\end{equation}
By It\^o's formula, we have
\[T^{\frac{d-2}{4}}\left(\log \sZ_\infty - \log \sZ_{T\tau}\right) = T^{\frac{d-2}{4}} \left(\int_{T\tau}^\infty \frac{\dd \sZ_s}{\sZ_s}
- \frac{1}{2} \int_{T\tau}^\infty \frac{\dd \langle \sZ \rangle_s}{\sZ_s^2}\right).
\]
Let $\e>0$ and define
\[A_{\tau,T} = \left\{\sup_{s\geq 0} \sZ_s^{-1} \leq \tau^{\e} T^{\e}\right\}.\]
By the same argument as for \eqref{eq:negligibleEvent}, we know that there is a constant $C>0$ such that $\IP\left(A_{\tau,T}^c\right) \leq C \tau^{-2d} T^{-2d}$ and moreover that $\sup_{s\geq 0} \IE |\log \sZ_s|^2 <\infty$.
Therefore by Cauchy-Schwarz inequality,
\[\sup_{T\geq 1} T^{\frac{d-2}{4}}\IE\left[\left|\log \sZ_\infty - \log \sZ_{\tau T}\right|\mathbf{1}_{A_{\tau,T}^c} \right]\leq C\tau^{-d}. \]
On the other hand, thanks to {\eqref{eq:bracketExpansion_Wbis}},
 \[T^{\frac{d-2}{4}} \IE\left[  \int_{T\tau}^\infty \frac{\dd \langle \sZ \rangle_s}{\sZ_s^2} \mathbf{1}_{A_{\tau,T}} \right] \leq  T^{\frac{d-2}{4}} \tau^{\e} T^\e \int_{T\tau}^\infty \IE {\dd \langle \sZ \rangle_s} \leq  C\tau^{-\frac{d-2}{2}+\e} T^{-\frac{d-2}{4} + \e},
\]
while by Burkholder-Davis-Gundy and Cauchy-Schwarz inequalities,
\begin{align*}
T^{\frac{d-2}{4}} \IE\left|\int_{T\tau}^\infty \frac{\dd \sZ_s}{\sZ_s}\right|&\leq T^{\frac{d-2}{4}} \IE\left(\int_{T\tau}^\infty \frac{\dd \langle \sZ \rangle_s}{\sZ_s^2}\right)^{1/2} \\
&\leq \IE\left[\sup_{s\geq 0} \sZ_s^{-2}  \right]^{1/2} T^{\frac{d-2}{4}} \left(\IE \int_{T\tau}^\infty  \dd \langle \sZ\rangle_s\right)^{1/2}\\
& \leq C \tau^{-\frac{d-2}{4}}.
\end{align*}
For $\e$ small enough we finally obtain \eqref{eq:conveLogZtauT}.

%

\subsection{Multidimensional convergence in the EW limits for KPZ}
\label{sec:KPZmultidim}
 To ease the presentation, we restrict ourselves to the case where $h_0\equiv 0$, although a repetition of the argument would lead to the result for the general initial conditions that we have been considering. Note that in this case $\mathscr U_3 = \mathscr U_1$.

By identity \eqref{eq:couplingProporty}, we see that what we need to show is that jointly for finitely many $u\in[0,t]$, $f\in\mathcal C^\infty_c$, as $T\to\infty$, (see \eqref{PhiAB} for definitions)
\begin{align*}
 T^{\frac{d-2}{4}} & \int f(x) \left( \log \sZ_{T u,T t}(\sqrt T x)- \IE\left[ \log \sZ_{T u,T t}(\sqrt T x)\right]\right)\dd x \\
& \cvlaw \quad \int_{\mathbb{R}^d} f(x)\mathscr U_3(t-u,x)\dd x.
\end{align*}
Following the same strategy as in Section \ref{sec:explainKPZ}, we are reduced to showing that
\begin{equation}\label{eq:finalGoalKPZmuli}
 \overline{M}^{\ssup T}_\tau(u,f)\cvlaw\int_{\mathbb{R}^d} f(x)\mathscr U_3(\tau-u,x)\dd x, \quad \text{jointly in } u\in[0,\tau], f\in\mathcal C^\infty_c,
\end{equation}
where (see \eqref{eq:defdMbartau} for definition)
\[
\overline{M}^{\ssup T}_\tau(u,f) := T^{\frac{d-2}{4}}\int_{\R^d} f(x) \int_{Tu  \lor T^{1-\varepsilon}}^{T\tau} \dd \overline{M}_{\tau_1}(\sqrt{T}x)\dd x.
\]
As in Section \ref{sec:conclusionOfProofOfSHE},  convergence in \eqref{eq:finalGoalKPZmuli} follows from convergence of the { cross-brackets} $\langle\overline{M}^{\ssup T}(u_1,f),\overline{M}^{\ssup T}(u_2,g)\rangle_\tau$ towards the RHS of \eqref{eq:LimcovStruc} and  the multidimensional functional central limit for martingales (\cite[Theorem 3.11]{JS87}). See the proof of Proposition \ref{lm:CVLone}. Note that convergence of the quadratic variation comes from the argument of Lemma \ref{lem:asymptPsi}, c.f.\ the proof of Proposition \ref{prop:MwidetildeIsGaussianFlat}.
{Multidimensional convergence for the stationary case (Theorem \ref{th:CV_stationaryKPZ}) comes again via exchanging limits as in Section \ref{subsec:conclusionGFFKPZ}.}

\appendix
\section{Law of large numbers for Dirac}\label{app:LLNDirac}
Similar arguments as in [CY] can be applied for the proof of Theorem 2.2 (Case of Dirac initial condition), that is 
\begin{align*}
\sZ_{\e^{-2}t}\mathtt{E}_0^{\b,\e^{-2}t}\left[\vphi\left(\e B_{T t}+x_0\right)\right]\to \sZ_{\infty}\DE_{x_0}[\vphi(B_t)].
\end{align*}

This is an analogue of Theorem 5.1 in \cite{CY06} and our proof is a modification.

Let $(\mathbb{W},\mathcal{G}^\mathbb{W},\DP_0)$ be the $d$-dimensional Wiener space: \begin{align*}
\mathbb{W}:=\{w\in C([0,\infty)\to \R^d); w(0)=0\}
\end{align*}
with  the $\sigma $-field $\mathcal{G}^{\mathbb{W}}$ defined as the smallest $\sigma$-field on $\mathbb{W}$ for which the coordinate maps $t\mapsto w(t)$ are measurable for each $t\geq 0$.

The proof should be divided into several steps.

\begin{lemma}{\cite[Lemma 5.3, Proposition 4.1]{CY06}}
Let $\cG_t=\sigma[B_s:0\leq s\leq t]$ and $\cG_\infty=\bigvee_{t\geq 0}\cG_t$. For any $F\in \cG_\infty$ and $G\in \cG_\infty^{\otimes 2}$, \begin{align*}
\mathtt{P}_{0}^{\b,\infty}(F)&:=\lim_{t\to \infty}\mathtt{P}_0^{\b,t}(F)\\
\mathtt{P}_{0,(2)}^{\b,\infty}(G)&:=\lim_{t\to\infty}\left (\mathtt{P}_0^{\b,t}\right)^{\otimes 2}(G)
\end{align*}
exists a.s. Moreover, $\mathtt{P}_0^{\beta,\infty}$ is a probability measure on $\cG_{t}$ for any $t\geq 0$ and \begin{align}
\IE\left[ \mathtt{P}_{0,(2)}^{\b,\infty}(G)\right]=\IE \left[ \left(\mathtt{P}_0^{\b,\infty}\right)^{\otimes 2}(G)\right] \quad \text{for any  }G\in \bigcup_{t\geq 0}\cF_t^{\otimes 2} \label{prodmeas}
\end{align} Moreover, $\IE \mathtt{P}_{0,(2)}^{\b,\infty}$ 
can be extended to the probability measure on $\cG_\infty^{\otimes 2}$ we also denote by $\IE \mathtt{P}_{0,\b}^{\otimes 2}$ and we have \begin{align}
\IE \mathtt{P}_{0,\b}^{\otimes 2}\ll \DP_0^{\otimes 2} \text{ on }\cF_\infty^{\otimes 2}.\label{abscont}
\end{align}
\end{lemma}

\begin{proof}[Proof]
The existence of the limit $\mathtt{P}_{0}^{\b,\infty}(F)$ follows from the martingale convergence theorem. 
We consider a non-negative  submartingale \[X_t=X_t(G):=\DP_0^{\otimes 2}\left[\Phi_{t}(B,\xi)\Phi_t(B',\xi)1_G(B,B')\right],\] and its Doob-Meyer decomposition \begin{align*}
X_t(G)&=\DP_0^{\otimes 2}(G)+M_t+A_t\\
M_t&=\DE_0^{\otimes 2}\left[\left(\int_0^t \Phi_s(B, \xi) {\rm d}\Phi_s(B',\xi)+\int_0^t \Phi_s(B',\xi){\rm d}\Phi_s(B,\xi)\right)1_G(B,B')\right] \\
A_t&=\beta^2\int_0^t \DP_0^{\otimes 2}\left[V(B_s-B_s')\Phi_s(B,\xi)\Phi_s(B',\xi)1_G(B,B')\right]{\rm d} s
\end{align*}
Then, it is easy to see that there exists a constant $C>0$ such that \begin{align*}
\langle M\rangle_t &\leq 4\left\langle 	\DE_0^{\otimes 2}\left[\int_0^\cdot \Phi_s(B, \xi) {\rm d}\Phi_s(B',\xi)\right]				\right\rangle_t \\
&\leq 4\beta^2\DE^{\otimes 4}\left[\int_0^t	V(B_s^{(3)}-B_s^{(4)})	\Phi_s(B^{(1)},\xi)\Phi_s(B^{(2)},\xi)\Phi_s(B^{(3)},\xi)\Phi_s(B^{(4)},\xi){\rm d}s			\right]\\
&		\leq C\sup_{u\geq 0}\sZ_u^4 \int_0^tI_s{\rm d}s, \\
A_t&\leq \beta^2\sup_{u\geq 0}\sZ_u^2 \int_0^tI_s{\rm d}s,
\end{align*}
where $I_t=\mathtt{ E}_0^{\b,t}\left[V(B_t-B_t')\right]$. In the weak disorder phase, we have that \begin{align*}
\left\{\sZ_\infty<\infty\right\}\stackrel{\text{a.s.}}{=}\left\{ \int_0^\infty I_tdt<\infty \right\}.
\end{align*}
Thus,  $X_t(G)$, as well as $\left(\mathtt{P}_{0}^{\b,t}\right)^{\otimes 2}(G)$, converges $\IP$-a.s. 

Also, it follows from Markov property that for $G\in \cG_t^{\otimes 2}$ \begin{align*}
&\left(\mathtt{P}_0^{\b,s+t}\right)^{\otimes 2}(G)\\
&=\frac{1}{\sZ_{s+t}^2}\iint_{\R^d\times \R^d}\DP_0^{\otimes 2}\left[\Phi_t(W,\xi)\Phi_t(B',\xi):G, B_t\in dx,B_t'\in dx'\right]\left(\theta_{t,x}\circ \sZ_{s}\right)\left(\theta_{t,x'}\circ \sZ_s\right)
\end{align*}
where $\theta_{t,x}$ is the time-space shift of $\xi$. The dominated convergence theorem  implies \eqref{prodmeas}. 

As in the proof of \cite[Lemma 4.2]{CY06}, we can show \eqref{prodmeas} and \eqref{abscont} by proving that for any $\{G_m\}_{m=1}^\infty\subset \cG_\infty^{\otimes 2}$ with $\lim_{m\to\infty }\DP_0^{\otimes 2}(G_m)=0$,  we can prove \begin{align*}
\lim_{m\to\infty}\IE\left[\mathtt{P}_{0,(2)}^{\b, \infty}(G_m)\right]=0.
\end{align*}
Hence $\IE \mathtt{P}_{0,(2)}^{\beta,\infty}$ can be extended to the probability measure on $\cG_\infty^{\otimes 2}$ and \eqref{abscont} follows.
\end{proof}

Now, we give the proof of Theorem \ref{LNDirac} in sveral steps by following the proof of \cite[Theorem 5.1]{CY06}.

\begin{proposition}\label{prop:CLTinfty}
Suppose  $\beta <\beta_c $. Then, we have for $G\in C_b(\mathbb{W}_1)$\begin{align*}
\lim_{t\to \infty}\IE \mathtt{P}_{0,(2)}^{\otimes 2}\left[G(w^{(t)},w'^{(t)})		\right]=\DP_0^{\otimes 2}[G(w,w')],
\end{align*} 
where $\mathbb{W}_1=C([0,1]\to \mathbb{R}^d)$ and $w^{(t)}=\displaystyle \left(\frac{w_{st}}{\sqrt{t}}\right)_{0\leq s\leq 1}\in \mathbb{W}_1$ is Brownian rescaling of $w\in \mathbb{W}$. In particular, we have \begin{align*}
\lim_{t\to\infty}\IE \left[\left|\mathtt{E}_{0}^{\beta,\infty}\left[F\left(w^{(t)}\right)\right]-\DE_0\left[F\left(w\right)\right]\right|\right]=0.
\end{align*}

\end{proposition}
\begin{proof} This proposition follows from  the same argument as the step 1 in the proof of  \cite[Proposition 5.2]{CY06} and the latter part of the proof of \cite[Theorem 5.1]{CY06}.

\end{proof}

\begin{proof}[Proof of Theorem \ref{LNDirac}] 

The proof is  almost  the same as the one in \cite[Proposition 5.2]{CY06}. Let $\bar{F}\in C_b(\mathbb{W})$ with $\DE\left[\bar{F}(w)\right]=0$. Then, it is enough to show that $\IE\left[\left|\mathtt{E}_0^{\b,t}\left[\bar{F}(w^{(t)})\right]\right|\right]\to 0$.

For each $0\leq s\leq t$,
\begin{align}
\IE\left[\left|\mathtt{P}_0^{\b,t}(\bar{F}(w^{(t)}))\right|\right]& \leq \IE\left[\left|\mathtt{E}_t^{\beta,t}\left[\bar{F}(w^{(t)})-\bar{F}(w^{(t-s)})\right]\right|\right]+\|\bar{F}\|_\infty \sup_{0\leq u\leq t}\IE\left[\|\mathtt{P}_0^{\b,t}-\mathtt{P}_0^{\b,\infty}\|_{\mathcal{G}_{u}}\right]\notag\\
&+\IE\left[\left|\mathtt{E}_0^{\b,\infty}\left[\bar{F}(w^{(t-s)})\right]\right|\right],\label{eq:PtPi}
\end{align}
where $\|F\|_\infty=\sup_{w\in \mathbb{W}_1}|F(w)|$ is the uniform norm of $F\in C_b(\mathbb{W}_1)$  and $\|\mu-\nu\|_{\mathcal{G}_t}=2\sup\{\mu(A)-\nu(A): A\in \mathcal{G}_t\}$ is total variation for $\mu$ and $\nu$ with probability measures on $(\mathbb{W},\mathcal{G}_t)$.
Then, the last term converges to $0$ as $t\to \infty$ by Proposition \ref{prop:CLTinfty}. 

Thus, we need to estimate the first term and the second term. The second term converges to $0$ by a similar argument to the proof of Proposition 4.3 in \cite{CY06} so we omit the proof with one remark: By Markov property of $B_t$, \begin{align*}
\mathtt{P}_0^{\b,\infty}(A)=\frac{1}{\sZ_\infty}\int_{\R^d}\DE_{0,0}^{t,x}\left[\Phi_t(B): A\right]\sZ^{(t,x)}_\infty \dd x,
\end{align*}
where $\sZ_\infty^{(t,x)}=\sZ_\infty(\xi_{(t,x)})$ is the time-space shift of $\sZ_\infty$ with $\xi_{(t,x)}(s,y)=\xi(s+t,x+y)$.

The first term in \eqref{eq:PtPi} is estimated as follows: For any $\ve>0$, there exists $c>0$ such that \begin{align*}
2\varlimsup_{t\to \infty}\IP(\sZ_t< c)\|\bar{F}\|_\infty<\ve.
\end{align*} so that  we have that \begin{align*}
&\varlimsup_{t\to\infty} \IE\left[\left| \mathtt{P}_{0}^{\beta,t}\left[\bar{F}(w^{(t)})-\bar{F}(w^{(t-s)})\right]\right|\right]\\
&	\leq 2\varlimsup_{t\to\infty}\IP\left(\sZ_t< c\right)\|\bar{F}\|_\infty+c\IE\left[\DE\left[\Phi_t\left|\bar{F}(w^{(t)})-\bar{F}(w^{(t-s)})\right|\right]			\right]	\\
&\leq \ve+\varlimsup_{t\to\infty}\DP_0\left[\left|\bar{F}(w^{(t)})-\bar{F}(w^{(t-s)})\right|\right] =\, \ve,
 \end{align*}
 where we have used the fact that the pair $(w^{(t)},w^{(t-s)})$ converges in law to the pair $(B,B)$  where $B$ is a standard Brownian motion on $\mathbb{R}^d$.
Thus, Theorem \ref{LNDirac} follows.

\end{proof}
\section{Error term in the local limit theorem for polymers}
Recall the definition of $\overset{\leftarrow}{\sZ}_{T,\ell}(z)$ in \eqref{eq:timeRevPartitionf}.

\begin{theorem} \label{th:errorTermLLT}
There exists a positive constant $c$  such that for all $\b< \b_{L^2}$, there exist $C=C(\b)$ and some $\delta=\delta(\b)>0$, such that for all positive $\ell,T$ verifying $\ell/T\leq \delta$ and for all $x,y\in\mathbb R^d$,
\[
\IE\left(\DE_{0,\sqrt T x}^{T,\sqrt T y} [\Phi_T] - \sZ_\ell(\sqrt T x) \overset{\leftarrow}{\sZ}_{T,\ell}(\sqrt T y)\right)^2 \leq C \ell^{-\frac{d-2}{2}} + C\left(1+|x-y|^2\right) e^{c\frac{\ell}{T}(x-y)^2} \left(\frac{\ell}{T}\right)^{\frac{1}{2}}.
\]

\begin{rem}
Note that in addition to giving a control on the error term, the theorem states that the error term still vanishes for starting and terminal points that can be distant up to sub-linear scale (i.e.\ $\sqrt T |x-y| = o(T)$), provided that $\ell$ is chosen accordingly. This is a significant improvement compared to the result of \cite{V06,Si95}. Another slight improvement is that we can let $\ell$ go up to $o(T)$ instead of $o(T^{1/2})$. 
\end{rem}
\end{theorem}
Although we follow in our proof the two main steps of \cite{V06}, some new arguments are needed in order to obtain good and uniform estimates on the error. The main difference in our approach is the use of uniform bound on the second moments of point-to-point partition functions.
\subsection{First Step}
\begin{lemma} There exist a constant $C=C(\b)>0$ such that for all $T>0$, uniformly for $\ell\leq \frac{T}{2}$,
\[
\sup_{x\in\mathbb R^d} \IE\left[\left(\DE_{0,0}^{T,\sqrt T x} [\Phi_T] - \DE_{0,0}^{T,\sqrt T x}[\Phi_{\ell}\Phi_{T-\ell,T}]\right)^2\right] \leq C\,\ell^{-\frac{d-2}{2}}. 
\]
\end{lemma}
\begin{proof}
Since $B_s^{(1)}-B_s^{(2)} \eqlaw \sqrt 2 B_s$ for two independent brownian motions, we have
\begin{align}
&\IE\bigg[\bigg(\DE_{0,0}^{T,\sqrt T x} [\Phi_T - \Phi_{\ell}\Phi_{T-\ell,T}]\bigg)^2\bigg] \nonumber\\
& = \DE_{0,0}^{T,0} \left[\mathrm e^{\b^2 \int_0^T V(\sqrt 2 B_s)ds} -\mathrm e^{\b^2 \int_0^{\ell} V(\sqrt 2 B_s)ds}\mathrm e^{\b^2 \int_{T-\ell}^T V(\sqrt 2 B_s)ds}\right]\nonumber\\
& \leq  \DE_{0,0}^{T,0}\left[ e^{\b^2\int_0^{tT} V(\sqrt 2  B_s)}\,\int_\ell^{T-\ell} V(\sqrt 2 B_t)\dd t\right]. 
\end{align}
The last expectation equals
\begin{equation}\label{Application LLT KPZ}
  \begin{split}
 &  \int_{[\ell,T-\ell]\times \mathbb R ^d} V(\sqrt 2 z)  \frac{\rho_t(z)\rho_{T-t}(z)}{\rho_T(0)} \DE_{0,0}^{t,z} \left[e^{\b^2\int_0^{t} V(\sqrt 2  B_s)\dd s} \right] \DE_{0,z}^{T-t,0} \left[ e^{\b^2\int_0^{T-t} V(\sqrt 2  B_s)\dd s}\right]  \dd z \dd t\\
 & \leq \left( \sup_{t} \sup_{|z|\leq R} \DE_{0,z}^{t,0}\left[ e^{\b^2\int_0^{t} V(\sqrt 2  B_s) \dd s} \right]\right)^2 \int_\ell^{T-\ell} \int_{\mathbb R^d} V(\sqrt 2 z)  \frac{\rho_t(z)\rho_{T-t}(z)}{\rho_T(0)} \dd t \dd z,
  \end{split}
  \end{equation}
with $\mathrm{supp}(V) \subset B(0,\sqrt 2 R)$ where the supremum on the last line is finite by the continuous analogue of \cite[Lemma 3.1]{CN19}  (see also \cite[Corollary 3.8]{V06}). Finally,
\begin{align*}
 & \int_\ell^{T-\ell} \int_{\mathbb R^d} \frac{\rho_t(z)\rho_{T-t}(z)}{\rho_T(0)} V\left(\sqrt 2 z\right)\dd t \dd z\\
& \leq \Vert V \Vert_\infty \left(\int_\ell^{\frac{T}{2}}  \frac{T^{\frac{d}{2}}}{t^{\frac{d}{2}}(T-t)^{\frac{d}{2}}}\dd t + \int_{\frac{T}{2}}^{T-\ell}  \frac{T^{\frac{d}{2}}}{t^{\frac{d}{2}}(T-t)^{\frac{d}{2}}}\dd t\right)\\
& \leq 2\Vert V \Vert_\infty\,2^{\frac{d}{2}}\int^{\frac{T}{2}}_{\ell} t^{-\frac{d}{2}}\dd t \leq C \frac{T^{\frac{d}{2}}}{(T-\ell)^{\frac{d}{2}}} \ell^{-\frac{d-2}{2}},
\end{align*}
and the statement of the lemma follows.
\end{proof}
Define 
\begin{equation}
A_{T,\ell,x} := \DE_{0,0}^{T,\sqrt T x}\left[\Phi_{\ell}\Phi_{T-\ell,T}\right] - \DE_0\left[\Phi_{\ell}\right] \overset{\leftarrow}{\sZ}_{T,\ell}(\sqrt T x).
\end{equation}

\subsection{Second step}
\begin{lemma} There exists a positive constant $c$ such that for all $\b< \b_{L^2}$ there exist positive constants $C=C(\b)$ and some $\delta=\delta(\b)>0$ such that for all positive $\ell,T$ verifying $\ell/T\leq \delta$ and all $x\in\mathbb R^d$, 
\[
 \IE\left[A_{T,\ell,x}^2\right]\leq  C(|x|^2+1) e^{c\frac{\ell}{2T}|x|^2} \left(\frac{\ell}{T}\right)^{\frac{1}{2}}.
\]
\end{lemma}
\begin{proof}
Let $\rho(s,y;t,x) =\frac{\rho_s(y) \rho_{t-s}(x-y)}{\rho_t(x)}$ denote the Brownian bridge density. 
Integrating over the middle point and using that $\rho_{T/2}(\sqrt{T} x)=T^{-\frac{d}{2}}\rho_{1/2}(x)$, we find:
\begin{align}
&\DE_{0,0}^{T,\sqrt T x}\left[\Phi_{\ell}\Phi_{T-\ell,T}\right] \nonumber \\
& =T^{\frac{d}{2}} \int_{\mathbb{R}^d} \rho(T/2,\sqrt T y;T,\sqrt Tx)  \DE_{0,0}^{T/2,\sqrt T y} \left[\Phi_{\ell}\right] \DE_{T/2,\sqrt T y}^{T,\sqrt T x}  \left[\Phi_{T-\ell,T}\right] \dd y,\nonumber\\
& = \int_{\mathbb{R}^d} \rho(1/2, y;1,x)  \DE_{0,0}^{T/2,\sqrt T y} \left[\Phi_{\ell}\right] \DE_{T/2,\sqrt T y}^{T,\sqrt T x}  \left[\Phi_{T-\ell,T}\right] \dd y.
\end{align}
Hence by Jensen's inequality and the indendence structure,
\begin{align}
&\IE\left(\DE_{0,0}^{T,\sqrt T x}\left[\Phi_{\ell}\Phi_{T-\ell,T}\right] - \int_{\mathbb{R}^d} \rho(1/2, y;1,x) \DE_0 \left[\Phi_{\ell}\right] \DE_{T/2,\sqrt T y}^{T,\sqrt T x}  \left[\Phi_{T-\ell,T}\right] \dd   y\right)^2\nonumber\\
& \leq \int_{\mathbb{R}^d} \rho(1/2, y;1,x) \IE\left(\DE_{0,0}^{T/2,\sqrt T y}\left[\Phi_{\ell} \right] - \DE_0\left[\Phi_{\ell} \right] \right)^2 \IE\, \DE_{T/2,\sqrt T y}^{T,\sqrt T x}  \left[\Phi_{T-\ell,T}\right]^2 \dd  y. \label{eq:jensenLLT}
\end{align}
We first see that by time-reversal invariance  and \eqref{eq:unifboundForP2P},
\begin{equation} \label{eq:unifL2again}
\sup_{T,\ell\leq \frac{T}{2},x,y} \IE\,\DE_{T/2,\sqrt T y}^{T,\sqrt T x}  \left[\Phi_{T-\ell,T}\right]^2 <\infty.
\end{equation} 
Then, we use \eqref{eq:unifboundForP2P} again with Cauchy-Schwarz inequality to find that
\begin{align}
&\IE\left(\DE_{0,0}^{T/2,\sqrt T y}\left[\Phi_{\ell}\right] - \DE_0\left[\Phi_{\ell} \right]\right)^2\nonumber\\
&= \IE \left(\int \rho_\ell(Z) \left( \frac{\rho_{\frac{T}{2}-\ell}(\sqrt T y-Z)}{\rho_{\frac{T}{2}}(\sqrt T y)} - 1\right) \DE_{0,0}^{\ell,Z}\left[\Phi_{\ell} \right]\dd Z\right)^2\nonumber\\
& \leq C\int \rho_\ell(Z_1) \rho_\ell(Z_2) \left| \frac{\rho_{\frac{T}{2}-\ell}(\sqrt T y-Z_1)}{\rho_{\frac{T}{2}}(\sqrt T y)} - 1\right| \left| \frac{\rho_{\frac{T}{2}-\ell}(\sqrt T y-Z_2)}{\rho_{\frac{T}{2}}(\sqrt T y)} - 1\right| \dd Z_1 \dd Z_2\nonumber\\
& = C  \left(\int \rho_\ell(Z) \left| \frac{\rho_{\frac{T}{2}-\ell}(\sqrt T y-Z)}{\rho_{\frac{T}{2}}(\sqrt T y)} - 1\right|\dd Z\right)^2. \label{eq:upperBoundCondvsFree}
\end{align}
for some $C=C(\b)$.

 Now, if we let $\delta = \ell/T$ and
\[
F_\delta(B)=\frac{\rho_{\frac{1}{2}-\delta}\left(y -  B_{\delta}\right)}{\rho_{\frac{1}{2}}(y)}=(1-2\delta)^{-\frac{d}{2}}e^{-\frac{|y-B_{\delta}|^2}{1-2\delta}+|y|^2},
\] 
we see via the heat kernel scaling property and identity $T^{-\frac{1}{2}}B_\ell \eqlaw B_{\delta}$ that
\begin{equation} \label{eq:backToProba}
\int \rho_\ell(Z) \left| \frac{\rho_{\frac{T}{2}-\ell}(\sqrt T y-Z)}{\rho_{\frac{T}{2}}(\sqrt T y)} - 1\right|\dd Z =\DE_{0}\left[ \left|F_\delta\left(B\right)-1\right|\right].
\end{equation}
By a straightforward computation, we further find that for $\delta$ small enough,
\begin{align*}
&\DE_{0}\left[ \left|F_\delta\left(B\right)-1\right|\right]\\
& \leq (1-2\delta)^{-\frac{d}{2}} \DE_0 \left|e^{(1-2\delta)^{-1}(-2\delta|y|^2 + 2\langle y,B_{\delta} \rangle - B_{\delta}^2)}-1 \right| + |(1-2\delta)^{-\frac{d}{2}} - 1|\\
& \leq 2^{\frac{d}{2} +1} \DE_0\left[ e^{4|\langle y,B_{\delta} \rangle|}\,\left|-2\delta|y|^2 + 2\langle y,B_{\delta} \rangle - B_{\delta}^2\right| \right] + C\delta,
\end{align*}
where we have used that $|e^x - 1| \leq e^{x\vee 0}|x|$ for all $x\in \mathbb R$. By Cauchy-Schwarz inequality, the last expectation is further bounded from above by 
\begin{align}
&(\DE_0[e^{8|\langle y,B_{\delta} \rangle|}])^{\frac{1}{2}} \DE_0\left[\left(-2\delta|y|^2 + 2\langle y,B_{\delta} \rangle - B_{\delta}^2\right)^2\right]^{\frac{1}{2}} \nonumber\\
&\leq C e^{c \delta | y|^2} \delta ^{\frac{1}{2}}  \left( | y |^2+1 \right), \label{eq:upperBoundCondvsFreeEnd}
\end{align}
for some positive constants $C$ and $c$ and $\delta$ small enough. 

Coming back to \eqref{eq:jensenLLT}, we obtain combining \eqref{eq:unifL2again}, \eqref{eq:upperBoundCondvsFree} and \eqref{eq:upperBoundCondvsFreeEnd} (recall $\delta = \ell/T$):
\begin{align}
&\IE\left(\DE_{0,0}^{T,\sqrt T x}\left[\Phi_{\ell}\Phi_{T-\ell,T}\right] - \int_{\mathbb{R}^d} \rho(1/2, y;1,x) \DE_0 \left[\Phi_{\ell}\right] \DE_{T/2,\sqrt T y}^{T,\sqrt T x}  \left[\Phi_{T-\ell,T}\right] \dd   y\right)^2\nonumber\\
& \leq  C \left(\frac{\ell}{T}\right)^{\frac{1}{2}} \int_{\mathbb{R}^d} \rho(1/2, y;1,x)  e^{c\frac{\ell}{T}|y|^2}(|y|^2+1)  \dd  y.\label{eq:BBintegral}
\end{align}
Then, if $X_s$ denotes a Brownian bridge from $(0,0)$ to $(1,0)$, the integral appearing on the RHS of \eqref{eq:BBintegral} is bounded from above by
\begin{align*} & \DE_{0,0}^{1,x} \left[(|B_{1/2}|^2+1) e^{c\frac{\ell}{T} |B_{1/2}|^2}\right]\\
& = \DE_{0,0}^{1,0} \left[\left(|X_{1/2}+x/2|^2 +1\right) e^{c\frac{\ell}{T} |X_{1/2}+x/2|^2}\right]\\
& \leq 2 e^{c\frac{\ell}{4T}x^2}\left(\DE_{0,0}^{1,0} \left[|X_{1/2}|^{2} e^{\frac{\ell}{T} |X_{1/2}|^2}\right] +\left(\frac{|x|^2}{4}+1\right)\,\DE_{0,0}^{1,0}  \left[ e^{c\frac{\ell}{T} |X_{1/2}|^2}\right] \right)\\
& \leq C e^{c\frac{\ell}{4T}|x|^2} (|x|^2+1),
\end{align*}
for $T$ large enough.

By the same strategy and translation invariance, we also get that
\begin{align}
&\IE\left( \int_{\mathbb{R}^d} \rho(1/2, y;1,x) \DE_0 \left[\Phi_{\ell}\right] \DE_{T/2,\sqrt T y}^{T,\sqrt T x}  \left[\Phi_{T-\ell,T}\right] \dd  y - \DE_0\left[\Phi_{\ell}\right] \overset{\leftarrow}{\sZ}_{T,\ell}(\sqrt T x)\right)^2 \nonumber\\
& \leq  C \left(\frac{\ell}{T}\right)^{\frac{1}{2}} \int_{\mathbb{R}^d} \rho(1/2, y;1,x)  e^{c\frac{\ell}{T}|x-y|^2}\left(|x-y|^2+1\right)  \dd  y,\nonumber\\
& \leq C e^{c\frac{\ell}{4T}|x|^2} (|x|^2+1),\label{eq:BBintegral2} 
\end{align}
where we obtain the last line via the change of variable $z=x-y$ and the previous computation.
Combining \eqref{eq:BBintegral} and \eqref{eq:BBintegral2} ends the proof of the lemma.
\end{proof}

\section*{Acknowledgments}
We thank Nikos Zigouras, Dimitris Lygkonis,  Ofer Zeitouni and David Belius for interesting discussions and helpful suggestions. The work of S. Nakajima is partially supported by  JSPS KAKENHI 19J00660 and SNSF grant 176918. C. Cosco acknowledges that this project has received funding from the European Research Council (ERC) under the European Union's Horizon 2020 research and innovation program (grant agreement No. 692452). We also thank Chiranjib Mukherjee for bringing to our attention the work \cite{LZ20}.

\end{document}